\documentclass{amsart}
\usepackage{amsmath}
\usepackage{amssymb}
\usepackage{amsfonts}
\usepackage{geometry}

\newtheorem{theorem}{Theorem}
\theoremstyle{plain}

\newtheorem{lemma}{Lemma}
\newtheorem{notation}{Notation}

\newtheorem{remark}{Remark}

\numberwithin{equation}{section}
\input{tcilatex}
\begin{document}
\title[Restricted testing]{Restricted testing for the Hardy-Littlewood
maximal function}
\author{Kangwei Li}
\author{Eric Sawyer}

\begin{abstract}
We answer a special case of a question of T. Hyt\"{o}nen regarding the two
weight norm inequality for the maximal function $\mathcal{M}$ in the
affirmative, namely that there is a constant $D>1$, depending only on
dimension $n$, such that the norm inequality%
\begin{equation*}
\int_{\mathbb{R}^{n}}\mathcal{M}\left( f\sigma \right) ^{2}d\omega \leq
C\int_{\mathbb{R}^{n}}f^{2}d\sigma 
\end{equation*}%
holds for all $f\geq 0$ \emph{if and only if} the $A_{2}$ condition holds,
and the testing condition%
\begin{equation*}
\int_{Q}\mathcal{M}\left( 1_{Q}\sigma \right) ^{2}d\omega \leq C\left\vert
Q\right\vert _{\sigma }
\end{equation*}%
holds for all cubes $Q$ satisfying $\left\vert 2Q\right\vert _{\sigma }\leq
D\left\vert Q\right\vert _{\sigma }$.
\end{abstract}

\maketitle
\tableofcontents

\section{Introduction}

We begin with a brief history of `testing conditions' as used in this paper.
One of the earliest uses of testing conditions to characterize a weighted
norm inequality occurs in the 1982 paper \cite{Saw3} on the maximal function 
$M$ that showed%
\begin{equation*}
\int_{\mathbb{R}^{n}}Mf\left( x\right) ^{2}w\left( x\right) dx\leq C\int_{%
\mathbb{R}^{n}}f\left( x\right) ^{2}v\left( x\right) dx,\ \ \ \ \ \text{for
all }f\left( x\right) \geq 0,
\end{equation*}%
if and only if the following testing condition holds:%
\begin{equation*}
\int_{\mathbb{R}^{n}}M\left( \mathbf{1}_{Q}v^{-1}\right) \left( x\right)
^{2}w\left( x\right) dx\leq C\int_{Q}v\left( x\right) ^{-1}dx,\ \ \ \ \ 
\text{for all cubes }Q\text{ in }\mathbb{R}^{n}.
\end{equation*}%
Thus it suffices to test the weighted norm inequality over the simpler
collection of test functions $f=\mathbf{1}_{Q}v^{-1}$ for cubes $Q$.

Two years later, David and Journ\'{e} showed in their celebrated $T1$
theorem \cite{DaJo}, that the unweighted inequality%
\begin{equation*}
\int_{\mathbb{R}^{n}}Tf\left( x\right) ^{2}dx\leq C\int_{\mathbb{R}%
^{n}}f\left( x\right) ^{2}dx,\ \ \ \ \ \text{for all }f\in L^{2}\left( 
\mathbb{R}^{n}\right) ,
\end{equation*}%
holds if and only if both a weak boundedness property and the following pair
of dual testing conditions held:%
\begin{equation*}
\int_{\mathbb{R}^{n}}T\left( \mathbf{1}_{Q}\right) \left( x\right)
^{2}dx\leq C\int_{Q}dx\text{ and }\int_{\mathbb{R}^{n}}T^{\ast }\left( 
\mathbf{1}_{Q}\right) \left( x\right) ^{2}dx\leq C\int_{Q}dx,\ \ \ \ \ \text{%
for all cubes }Q\text{ in }\mathbb{R}^{n}.
\end{equation*}%
Here $T$ is a general Calder\'{o}n-Zygmund singular integral on $\mathbb{R}%
^{n}$ and the testing functions are simply the indicators $\mathbf{1}_{Q}$
for cubes $Q$\footnote{%
The moniker `$T1$ theorem' refers to the equivalent formulation of the
testing conditions as $T1\in BMO$ and $T^{\ast }1\in BMO$.}. The following
year David, Journ\'{e} and Semmes extended the $T1$ theorem to a $Tb$
theorem \cite{DaJoSe} in which the testing conditions become $b\mathbf{1}%
_{Q} $ and $b^{\ast }\mathbf{1}_{Q}$ for appropriately accretive functions $%
b $ and $b^{\ast }$ on $\mathbb{R}^{n}$.

A couple of decades later, and motivated by the Painlev\'{e} problem of
characterizing removable singularities for bounded analytic functions,
Nazarov, Treil and Volberg solved in 2003 a particular one-weight
formulation of the norm inequality for Riesz transforms $\mathcal{R}$,
including the Cauchy transform $\mathcal{C}g\left( z\right) \equiv \int_{%
\mathbb{C}}\frac{1}{w-z}g\left( w\right) dw$ \cite{NTV},%
\begin{equation*}
\int_{\mathbb{R}^{n}}\left\vert \mathcal{R}\left( f\mu \right) \left(
x\right) \right\vert ^{2}d\mu \left( x\right) \leq C\int_{\mathbb{R}%
^{n}}f\left( x\right) ^{2}d\mu \left( x\right) ,\ \ \ \ \ \text{for all }%
f\in L^{2}\left( \mathbb{R}^{n};\mu \right) ,
\end{equation*}%
if and only if a weak boundedness property and the following testing
condition held: 
\begin{equation*}
\int_{Q}\left\vert \mathcal{R}\left( \mathbf{1}_{Q}\mu \right) \left(
x\right) \right\vert ^{2}d\mu \left( x\right) \leq C\int_{Q}d\mu \left(
x\right) ,\ \ \ \ \ \text{for all cubes }Q\text{ in }\mathbb{R}^{n}.
\end{equation*}%
Here the testing functions are $f=\mathbf{1}_{Q}$. The Painlev\'{e} problem
was solved\ in the same year by Tolsa \cite{Tol}, a culmination of an
illustrious body of work by many mathematicians.

Finally, building on the work of Nazarov, Treil and Volberg in their 2004
paper \cite{NTV3} on the Hilbert transform, that in turn used the random
dyadic grids of \cite{NTV} (that followed on those of Fefferman and Stein 
\cite{FeSt}, Garnett and Jones \cite{GaJo}, and Sawyer \cite{Saw3}), and the
weighted Haar wavelets of \cite{NTV} (that followed on those of Coifman,
Jones and Semmes \cite{CoJoSe}), the two weight norm inequality for the
Hilbert transform was characterized in 2014 in the two-part paper Lacey,
Sawyer, Shen and Uriarte-Tuero \cite{LaSaShUr3} - Lacey \cite{Lac} as
follows:

\begin{equation*}
\int_{\mathbb{R}^{n}}H\left( f\sigma \right) \left( x\right) ^{2}d\omega
\left( x\right) \leq C\int_{\mathbb{R}^{n}}f\left( x\right) ^{2}d\sigma
\left( x\right) ,\ \ \ \ \ \text{for all }f\in L^{2}\left( \mathbb{R}%
^{n}\right) ,
\end{equation*}%
if and only if both the strong Muckenhoupt $A_{2}$ condition%
\begin{equation*}
\mathcal{A}_{2}\left( \sigma ,\omega \right) \equiv \int_{\mathbb{R}}\frac{%
\ell \left( I\right) }{\left( \ell \left( I\right) +\left\vert
x-c_{I}\right\vert \right) ^{2}}d\omega \left( x\right) \cdot \int_{\mathbb{R%
}}\frac{\ell \left( I\right) }{\left( \ell \left( I\right) +\left\vert
x-c_{I}\right\vert \right) ^{2}}d\sigma \left( x\right) <\infty ,
\end{equation*}%
and the following dual testing conditions hold:%
\begin{equation*}
\int_{Q}H\left( \mathbf{1}_{I}\sigma \right) \left( x\right) ^{2}d\omega
\left( x\right) \leq A\int_{Q}d\sigma \text{ and }\int_{Q}H\left( f\omega
\right) \left( x\right) ^{2}d\sigma \left( x\right) \leq C\int_{Q}d\omega ,\
\ \ \ \ \text{for all intervals }I.
\end{equation*}%
The extension to permitting common point masses in the measure pair $\left(
\sigma ,\omega \right) $ was added\ shortly after by Hyt\"{o}nen \cite{Hyt2}%
, where again the weighted norm inequality is tested over indicator
functions $f=\mathbf{1}_{I}$ of intervals $I$. The two-weight inequality for
the $g$ function was then characterized by testing conditions in \cite{LaLi}%
, and a further extension to a $Tb$ theorem for the Hilbert transform is in 
\cite{SaShUr3}.

\begin{description}
\item[Point of departure] The point of departure for the present paper
begins with an observation of T. Hyt\"{o}nen, namely that in the one-weight
formulation above of the norm inequality for Riesz transforms by Nazarov,
Treil and Volberg, their testing condition is%
\begin{equation*}
\int_{Q}\left\vert \mathcal{R}\left( \mathbf{1}_{Q}\mu \right) \left(
x\right) \right\vert ^{2}d\mu \left( x\right) \leq C\int_{2Q}d\mu \left(
x\right) ,\ \ \ \ \ \text{for all cubes }Q\text{ in }\mathbb{R}^{n},
\end{equation*}%
where the double $2Q$ of the cube $Q$ appears on the right hand side.
Moreover, one may restrict the testing functions to those functions $f=%
\mathbf{1}_{Q}$ for which $Q$ is a $\mu $\emph{-doubling cube} for some
appropriate positive constant $D$\footnote{%
This philosophy was successfully carried out in the context of the
one-weight $Tb$ theorem for nonhomogeneous square functions by Martikainen,
Mourgoglou and Vuorinen in \cite{MaMoVu}.}:%
\begin{equation*}
\int_{2Q}d\mu \leq D\int_{Q}d\mu .
\end{equation*}
This then motivated Hyt\"{o}nen to ask\footnote{%
private communication with the first author circa November 2014.} to what
extent one can similarly restrict testing functions to doubling cubes for
classical operators in other two-weight situations, including those
discussed above.
\end{description}

An initial step in the two-weight setting was taken by the authors in \cite%
{LiSa}, where it was shown that such a restriction to doubling cubes is
possible in the two weight norm inequality for fractional integrals. The
maximal function $M$ was also considered in \cite{LiSa}, but only a much
weaker result along these lines was obtained for $M$. The purpose of this
paper is to prove the full result for $M$, namely that it suffices to
restrict testing to doubling cubes in the two weight norm inequality for $M$.

\begin{description}
\item[Motivation] Besides the intrinsic interest in minimizing the functions
over which an inequality must be tested in order to verify its validity,
even a partial resolution of the question of restricted testing for singular
integrals has the potential to characterize two weight norm inequalities for
such operators - including Riesz transforms in higher dimensions, currently
a very difficult open problem, see e.g. \cite{SaShUr7}, \cite{LaWi} and \cite%
{LaSaShUrWi}. Indeed, the \emph{nondoubling cubes} have traditionally been
viewed as the enemy in two weight inequalities for singular integrals, and
(the techniques used in) the restriction of the testing conditions to just
doubling cubes could help circumvent the difficulty that \emph{energy
conditions} fail to be necessary for two weight inequalities in higher
dimensions \cite{Saw3} - the point being that a similarly restricted energy
condition could suffice.
\end{description}

Let $\mathcal{P}=\mathcal{P}^{n}$ be the collection of cubes in $\mathbb{R}%
^{n}$ with sides parallel to the coordinate axes, and side lengths $\ell
\left( Q\right) \in \left\{ 2^{\ell }\right\} _{\ell \in \mathbb{Z}}$ equal
to an integral power of $2$. For $Q\in \mathcal{P}$ and $\Gamma \geq 1$, let 
$\Gamma Q$ denote the cube concentric with $Q$ but having $\Gamma $ times
the side length, $\ell \left( \Gamma Q\right) =\Gamma \ell \left( Q\right) $%
. As mentioned above, the purpose of this paper is to prove an answer to a
question of T. Hyt\"{o}nen, in the context of the maximal function $\mathcal{%
M}$. For a locally signed measure $\mu $ on $\mathbb{R}^{n}$ (meaning the
total variation $\left\vert \mu \right\vert $ of $\mu $ is locally finite),
we define the maximal function $\mathcal{M}\mu $ of $\mu $ at $x\in \mathbb{R%
}^{n}$ by\footnote{%
The supremum over $Q\in \mathcal{P}^{n}$ used here is pointwise equivalent
to the usual supremum over all cubes $Q$ with sides parallel to the
coordinate axes.} 
\begin{equation*}
\mathcal{M}\mu \left( x\right) \equiv \sup_{\,Q\in \mathcal{P}^{n}:\ x\in Q}%
\frac{1}{\left\vert Q\right\vert }\int_{Q}d\left\vert \mu \right\vert \ .
\end{equation*}%
Given a pair $\left( \sigma ,\omega \right) $ of weights (i.e. positive
Borel measures) in $\mathbb{R}^{n}$ and $\Gamma >1$, we say that $\left(
\sigma ,\omega \right) $ satisfies the $\Gamma $\emph{-testing condition}
for the maximal function $\mathcal{M}$ if there is a constant $\mathfrak{T}_{%
\mathcal{M}}\left( \Gamma \right) \left( \sigma ,\omega \right) $ such that 
\begin{equation}
\int_{Q}\left\vert \mathcal{M}\left( \mathbf{1}_{Q}\sigma \right)
\right\vert ^{2}d\omega \leq \mathfrak{T}_{\mathcal{M}}\left( \Gamma \right)
\left( \sigma ,\omega \right) ^{2}\left\vert \Gamma Q\right\vert _{\sigma }\
,\ \ \ \ \ \text{for all }Q\in \mathcal{P}^{n}\ ,  \label{Gamma testing}
\end{equation}%
and if so we denote by $\mathfrak{T}_{\mathcal{M}}\left( \Gamma \right)
\left( \sigma ,\omega \right) $ the least such constant.

There is also the following weaker testing condition, in which one need only
test the inequality over cubes that are `doubling'. Given a pair $\left(
\sigma ,\omega \right) $ of weights in $\mathbb{R}^{n}$ and $D,\Gamma >1$,
we say that $\left( \sigma ,\omega \right) $ satisfies the $D$\emph{-}$%
\Gamma $\emph{-testing condition} for the maximal function $\mathcal{M}$ if
there is a constant $\mathfrak{T}_{\mathcal{M}}^{D}\left( \Gamma \right)
\left( \sigma ,\omega \right) $ such that 
\begin{equation}
\int_{Q}\left\vert \mathcal{M}\mathbf{1}_{Q}\sigma \right\vert ^{2}d\omega
\leq \mathfrak{T}_{\mathcal{M}}^{D}\left( \Gamma \right) \left( \sigma
,\omega \right) ^{2}\left\vert Q\right\vert _{\sigma }\ ,\ \ \ \ \ \text{for
all }Q\in \mathcal{P}^{n}\text{ with }\left\vert \Gamma Q\right\vert
_{\sigma }\leq D\left\vert Q\right\vert _{\sigma }\ ,
\label{D Gamma testing}
\end{equation}%
and if so we denote by $\mathfrak{T}_{\mathcal{M}}^{D}\left( \Gamma \right)
\left( \sigma ,\omega \right) $ the least such constant. Note that the $%
\Gamma $-testing condition implies the $D$-$\Gamma $-testing condition for
all $D>1$.

As shown in \cite{LiSa}, these restricted testing conditions are not by
themselves sufficient for the norm inequality - the classical Muckenhoupt
condition is needed as well:%
\begin{equation*}
A_{2}\left( \sigma ,\omega \right) \equiv \sup_{Q\in \mathcal{P}^{n}}\left( 
\frac{1}{\left\vert Q\right\vert }\int_{Q}d\sigma \right) \left( \frac{1}{%
\left\vert Q\right\vert }\int_{Q}d\omega \right) <\infty .
\end{equation*}%
Finally we let $\mathfrak{N}_{\mathcal{M}}\left( \sigma ,\omega \right) $ be
the operator norm of $\mathcal{M}$ as a mapping from $L^{2}\left( \sigma
\right) \rightarrow L^{2}\left( \omega \right) $, i.e. the best constant $%
\mathfrak{N}_{\mathcal{M}}\left( \sigma ,\omega \right) $ in the inequality%
\begin{equation*}
\int_{\mathbb{R}^{n}}\left\vert \mathcal{M}\left( f\sigma \right)
\right\vert ^{2}d\omega \leq \mathfrak{N}_{\mathcal{M}}\left( \sigma ,\omega
\right) ^{2}\int_{\mathbb{R}^{n}}\left\vert f\right\vert ^{2}d\sigma ,\ \ \
\ \ \text{for all }f\in L^{2}\left( \sigma \right) .
\end{equation*}

\begin{theorem}
\label{weak}Let $\Gamma >1$. Then there is $D>1$ depending only on $\Gamma $
and the dimension $n$ such that%
\begin{equation*}
\mathfrak{N}_{\mathcal{M}}\left( \sigma ,\omega \right) \approx \mathfrak{T}%
_{\mathcal{M}}^{D}\left( \Gamma \right) \left( \sigma ,\omega \right) +\sqrt{%
A_{2}\left( \sigma ,\omega \right) },
\end{equation*}%
for all locally finite positive Borel measures $\sigma $ and $\omega $ on $%
\mathbb{R}^{n}$.
\end{theorem}

\begin{remark}
An inspection of the proof (Step 2 in Section \ref{Sec weak}) shows that the
supremum over cubes $Q$ in the testing constant $\mathfrak{T}_{\mathcal{M}%
}^{D}\left( \Gamma \right) \left( \sigma ,\omega \right) $ in Theorem \ref%
{weak} may be further restricted to those cubes $Q$ having \emph{null
boundary}, i.e. $\left\vert \partial Q\right\vert _{\sigma +\omega }=0$ (%
\emph{cf.} the one-weight theorem in \cite{MaMoVu} where this type of
reduction first appears).
\end{remark}

The proof of this theorem splits neatly into two parts. In the first part of
the proof, we adapt the argument in our previous paper \cite{LiSa} to handle
the difficulties arising when a tripled cube spills outside a supercube -
and this requires a careful application of a probabilistic argument of the
type pioneered by Nazarov, Treil and Volberg (\cite{NTV}). With this
accomplished, the sufficiency of the stronger $\Gamma $-testing condition (%
\ref{Gamma testing}) is already proved. In the second part of the proof we
use this interim result to establish an \textit{a priori} bound on the
operator norm $\mathfrak{N}_{\mathcal{M}}\left( \sigma ,\omega \right) $ in
order to absorb additional terms arising from the absence of any testing
condition at all in (\ref{D Gamma testing}) when the cubes are not doubling
- and this requires a reduction to mollifications of the measures $\sigma $
and $\omega $. As a consequence of this splitting, we will give the proof in
two stages, beginning with the proof of the following weaker theorem, which
requires probability, but not mollification, and which is then used to prove
our main result Theorem \ref{weak}. We emphasize that this paper is
self-contained, and in particular does not rely on results from our earlier
paper \cite{LiSa}.

\begin{theorem}
\label{maximal}For $\Gamma >1$ we have%
\begin{equation*}
\mathfrak{N}_{\mathcal{M}}\left( \sigma ,\omega \right) \approx \mathfrak{T}%
_{\mathcal{M}}\left( \Gamma \right) \left( \sigma ,\omega \right) +\sqrt{%
A_{2}\left( \sigma ,\omega \right) },
\end{equation*}%
for all pairs $\left( \sigma ,\omega \right) $ of locally finite positive
Borel measures on $\mathbb{R}^{n}$, and where the implicit constants of
comparability depend on both $\Gamma $ and dimension $n$.
\end{theorem}

For convenience we will restrict our proof of Theorem \ref{maximal} to the
case $\Gamma =3$, the general case of $\Gamma $ large being an easy
modification of this one.

\section{Preliminaries}

Here we introduce some standard tools we will use in the proof of Theorem %
\ref{maximal}.

\subsection{Dyadic grids and conditional probability\label{Sub dyadic}}

In this subsection we introduce two parameterizations of grids, explain the
conditional probability estimates we will need, and recall how the maximal
function is controlled by an average over dyadic maximal functions.

To set notation we begin with the standard family of random dyadic grids $%
\mathcal{G}$ on $\mathbb{R}^{n}$. Let 
\begin{equation*}
\mathcal{D}_{0}:=\{2^{j}([0,1)^{n}+k),j\in \mathbb{Z},k\in \mathbb{Z}^{n}\}.
\end{equation*}%
Then for $\beta =\{\beta _{j}\}_{j=-\infty }^{\infty }\in (\{0,1\}^{n})^{%
\mathbb{Z}}$, define 
\begin{equation}
\mathcal{D}^{\beta }:=\left\{ Q+\sum_{j:2^{-j}<\ell (Q)}2^{-j}\beta
_{j},Q\in \mathcal{D}_{0}\right\} .  \label{def infinite dyadic grid}
\end{equation}%
Denote by $\boldsymbol{P}_{\Omega }$ the natural probability measure on $%
\Omega :=(\{0,1\}^{n})^{\mathbb{Z}}$, which we identify with the
corresponding collection of grids $\mathcal{G}=\left\{ \mathcal{D}^{\beta
}\right\} _{\beta \in \Omega }$, i.e. we write $\Omega =\mathcal{G}$. We
will use grids in $\Omega =\left\{ \mathcal{D}^{\beta }\right\} _{\beta \in
(\{0,1\}^{n})^{\mathbb{Z}}}$ to construct Whitney collections $\mathcal{W}%
^{\beta }$ of cubes relative to a monotone family of open sets in Subsection %
\ref{Sub dyadic} below. In probability calculations, we will use truncated
versions of these grids. More precisely, given $\mathcal{D}=\mathcal{D}%
^{\beta }$ with $\beta \in \Omega $, and $M,N\in \mathbb{Z}$ with $N\leq M$,
define the associated `truncated' grid 
\begin{equation*}
\mathcal{D}_{M}^{N}\equiv \left\{ Q\in \mathcal{D}=\mathcal{D}^{\beta
}:2^{-M}\leq \ell \left( Q\right) \leq 2^{-N}\right\} .
\end{equation*}%
Thus each $\mathcal{D}_{M}^{N}$ is a grid on the set of cubes $\left\{ Q\in 
\mathcal{D}:2^{-M}\leq \ell \left( Q\right) \leq 2^{-N}\right\} $. In
particular, $\mathcal{D}_{M}^{M}=\left\{ Q\in \mathcal{D}:\ell \left(
Q\right) =2^{-M}\right\} $ is the tiling of $\mathbb{R}^{n}$ by the dyadic
cubes in $\mathcal{D}$ of side length $2^{-M}$. If $\mathcal{D},\mathcal{E}%
\in \Omega $ (using our identification of $\Omega $ with $\mathcal{G}$
above), then $\mathcal{D}\cap \mathcal{E}\neq \emptyset $ if and only if $%
\mathcal{D}_{M}^{M}=\mathcal{E}_{M}^{M}$ for some $M\in \mathbb{Z}$, in
which case $\mathcal{D}_{M^{\prime }}^{M^{\prime }}=\mathcal{E}_{M^{\prime
}}^{M^{\prime }}$ for all $M^{\prime }\in \mathbb{Z}$ with $M^{\prime }\geq
M $. We will develop further properties of grids of the form $\mathcal{D}%
_{M}^{N}$ in the next subsubsection, including the fact that there are only
finitely many (namely $2^{n\left( M-N\right) }$) different grids of the form 
$\mathcal{D}_{M}^{N}$.

\subsubsection{Parameterizations of a finite set of dyadic grids\label%
{Subsub parameter}}

Here we recall two constructions from \cite{SaShUr10} of \emph{special}
collections of truncated grids of cubes - \emph{special} because the origin
is a vertex of any cube in which it is contained. We momentarily fix a large
positive integer $M\in \mathbb{N}$, and consider the tiling of $\mathbb{R}$
by the family of intervals $\mathbb{D}_{M}\equiv \left\{ I_{\alpha
}^{M}\right\} _{\alpha \in \mathbb{Z}}$ having side length $2^{-M}$ and
given by $I_{\alpha }^{M}\equiv I_{0}^{M}+2^{-M}\alpha $ where $I_{0}^{M}=%
\left[ 0,2^{-M}\right) $. A \emph{dyadic grid} $\mathcal{D}$ built on $%
\mathbb{D}_{M}$ is\ defined to be a family of intervals $\mathcal{D}$
satisfying:

\begin{enumerate}
\item Each $I\in \mathcal{D}$ has side length $2^{-\ell }$ for some $\ell
\in \mathbb{Z}$ with $\ell \leq M$, and $I$ is a union of $2^{M-\ell }$
intervals from the tiling $\mathbb{D}_{M}$,

\item For $\ell \leq M$, the collection $\mathcal{D}_{\left[ \ell \right] }$
of intervals in $\mathcal{D}$ having side length $2^{-\ell }$ forms a
pairwise disjoint decomposition of the space $\mathbb{R}$,

\item Given $I\in \mathcal{D}_{\left[ i\right] }$ and $J\in \mathcal{D}_{%
\left[ j\right] }$ with $j\leq i\leq M$, it is the case that either $I\cap
J=\emptyset $ or $I\subset J$.
\end{enumerate}

We denote the collection of all dyadic grids built on $\mathbb{D}_{M}$ by $%
\boldsymbol{A}_{M}$. We now also momentarily fix an integer $N\in \mathbb{Z}$
with $N\leq M$, and consider the collection $\boldsymbol{A}_{M}^{N}$ of
dyadic grids obtained by restricting the grids in $\boldsymbol{A}_{M}$ to
containing only intervals of side length at most $2^{-N}$. We refer to the
dyadic grids in $\boldsymbol{A}_{M}^{N}$ as (\emph{special truncated})
dyadic grids built on $\mathbb{D}_{M}$ of size $2^{-N}$.

\begin{notation}
We denote the collection of all intervals belonging to the grids in $%
\boldsymbol{A}_{M}^{N}$ by $\mathcal{S}_{M}^{N}$ ($\mathcal{S}$ for
special), and reserve $\mathcal{P}_{M}^{N}$ ($\mathcal{P}$ for parallel) for
the collection of all intervals $Q$ in $\mathcal{P}$ with $2^{-M}\leq \ell
\left( Q\right) \leq 2^{-N}$.
\end{notation}

There are now two traditional means of constructing probability measures on
collections of such dyadic grids, namely parameterization by choice of
parent, and parameterization by translation. We will typically use $\mathcal{%
D}$ to denote one of these truncated grids when the underlying parameters $M$
and $N$ are understood. Here are the two constructions from \cite{SaShUr10}.

\textbf{Construction \#1}: For any 
\begin{equation*}
\beta =\{\beta _{i}\}_{i\in \mathbb{Z}_{M}^{N}}\in \omega _{M}^{N}\equiv
\left\{ 0,1\right\} ^{\mathbb{Z}_{M}^{N}},
\end{equation*}%
where $\mathbb{Z}_{M}^{N}\equiv \left\{ \ell \in \mathbb{Z}:N\leq \ell \leq
M\right\} $, define the dyadic grid $\mathcal{D}_{\beta }$ built on $\mathbb{%
D}_{m}$ of size $2^{-N}$ by 
\begin{equation}
\mathcal{D}_{\beta }=\left\{ 2^{-\ell }\left( [0,1)+k+\sum_{i:\ \ell <i\leq
M}2^{-i+\ell }\beta _{i}\right) \right\} _{N\leq \ell \leq M,\,k\in {\mathbb{%
Z}}}\ .  \label{def dyadic grid}
\end{equation}%
Place the uniform probability measure $\rho _{M}^{N}$ on the finite index
space $\omega _{M}^{N}=\left\{ 0,1\right\} ^{\mathbb{Z}_{M}^{N}}$, namely
that which charges each $\beta \in \omega _{M}^{N}$ equally.

\textbf{Construction \#2}: Momentarily fix a (truncated) dyadic grid $%
\mathcal{D}$ built on $\mathbb{D}_{m}$ of size $2^{-N}$. For any 
\begin{equation*}
t\in \gamma _{M}^{N}\equiv \left\{ 2^{-m}\mathbb{Z}_{+}:\left\vert
t_{i}\right\vert <2^{-N}\right\} ,
\end{equation*}%
define the dyadic grid $\mathcal{D}^{t}$ built on $\mathbb{D}_{m}$ of size $%
2^{-N}$ by%
\begin{equation*}
\mathcal{D}^{t}\equiv \mathcal{D}+t.
\end{equation*}%
Place the uniform probability measure $\sigma _{M}^{N}$ on the finite index
set $\gamma _{M}^{N}$, namely that which charges each multiindex $t$ in $%
\gamma _{M}^{N}$ equally.

These constructions are then extended to Euclidean space $\mathbb{R}^{n}$ by
taking products in the usual way and using the product index spaces $\Omega
_{M}^{N}\equiv \omega _{M}^{N}\times ...\times \omega _{M}^{N}$ and $\Gamma
_{M}^{N}\equiv \gamma _{M}^{N}\times ...\times \gamma _{M}^{N}$, together
with the uniform product probability measures $\mu _{M}^{N}=\rho
_{M}^{N}\times ...\times \rho _{M}^{N}$ and $\nu _{M}^{N}=\sigma
_{M}^{N}\times ...\times \sigma _{M}^{N}$, where there are $n$ factors in
each product above.

The two probability spaces $\left( \left\{ \mathcal{D}_{\beta }\right\}
_{\beta \in \Omega _{M}^{N}},\mu _{M}^{N}\right) $ and $\left( \left\{ 
\mathcal{D}^{t}\right\} _{t\in \Gamma _{M}^{N}},\nu _{M}^{N}\right) $ are
isomorphic since both collections $\left\{ \mathcal{D}_{\beta }\right\}
_{\beta \in \Omega _{M}^{N}}$ and $\left\{ \mathcal{D}^{t}\right\} _{t\in
\Gamma _{M}^{N}}$ describe the \emph{finite} set $\boldsymbol{A}_{M}^{N}$ of 
\textbf{all} (truncated) dyadic grids $\mathcal{D}$ built on $\mathbb{D}_{M}$
of size $2^{-N}$, and since both measures $\mu _{M}^{N}$ and $\nu _{M}^{N}$
are the uniform measure on this space. The first construction may be thought
of as being \emph{parameterized by scales} - each component $\beta _{i}$ in $%
\beta =\{\beta _{i}\}\in \omega _{M}^{N}$ amounting to a choice of the $%
2^{n} $ possible tilings at level $i$ that respect the choice of tiling at
the level below - and since any grid in $\boldsymbol{A}_{M}^{N}$ is
determined by a choice of scales , we see that $\left\{ \mathcal{D}_{\beta
}\right\} _{\beta \in \Omega _{M}^{N}}=\boldsymbol{A}_{M}^{N}$. The second
construction may be thought of as being \emph{parameterized by translation}
- each $t\in \gamma _{M}^{N}$ amounting to a choice of translation of the
grid $\mathcal{D}$ fixed in construction \#2\ - and since any grid in $%
\boldsymbol{A}_{M}^{N}$ is determined by any of the intervals at the top
level, i.e. with side length $2^{-N}$, we see that $\left\{ \mathcal{D}%
^{t}\right\} _{t\in \Gamma _{M}^{N}}=\boldsymbol{A}_{M}^{N}$ as well, since
every interval at the top level in $\boldsymbol{A}_{M}^{N}$ has the form $%
Q+t $ for some $t\in \Gamma _{M}^{N}$ and $Q\in \mathcal{D}$ at the top
level in $\boldsymbol{A}_{M}^{N}$ (i.e. every cube at the top level in $%
\boldsymbol{A}_{M}^{N}$ is a union of small cubes in $\mathbb{D}_{M}\times
...\times \mathbb{D}_{M}$, and so must be a translate of some $Q\in \mathcal{%
D}$ by $2^{-M}$ times an element of $\mathbb{Z}_{+}^{n}$). Note also that in
all dimensions, $\#\Omega _{M}^{N}=\#\Gamma _{M}^{N}=2^{n\left( M-N\right) }$%
. We will use $\boldsymbol{E}_{\boldsymbol{A}_{M}^{N}}$ to denote
expectation with respect to this common probability measure on the finite
set $\boldsymbol{A}_{M}^{N}$.

We will invoke these special collections of truncated grids in order to
prove a conditional probability estimate (\ref{comp}) below. Then we will
take limits as in Lemma \ref{domination} below to complete our argument. For
this we will use the following observations.

Given a dyadic grid $\mathcal{D}\in \Omega $, there is a unique $s\in \left[
0,2^{-M}\right) ^{n}$ such that%
\begin{equation*}
\left( \mathcal{D}-s\right) _{M}^{N}=\mathcal{D}_{M}^{N}-s\in \boldsymbol{A}%
_{M}^{N}\ ,
\end{equation*}%
and so we have the decomposition%
\begin{equation}
\left\{ \mathcal{D}_{M}^{N}:\mathcal{D}\in \Omega \right\}
=\dbigcup\limits_{s\in \left[ 0,2^{-M}\right) ^{n}}\boldsymbol{A}_{M}^{N}+s\
,  \label{slice decomp}
\end{equation}%
that expresses the fact that the collection of truncations of arbitrary
dyadic grids coincides with the collection of translations by a point in $%
\left[ 0,2^{-M}\right) ^{n}\,$\ of the \emph{special} collection of
truncated grids $\boldsymbol{A}_{M}^{N}$ constructed above.

\subsubsection{Conditional probability\label{Subsub conditional}}

Here we consider the finite collection of grids $\boldsymbol{A}_{M}^{N}$
depending on a pair of integers $M,N$ that was introduced in the previous
subsubsection. Fix attention on a given cube $I\in \mathcal{S}_{M}^{N}$. Let 
$\boldsymbol{P}_{\boldsymbol{A}_{M}^{N}}$ and $\boldsymbol{E}_{\boldsymbol{A}%
_{M}^{N}}$ denote probability and expectation over the family $\boldsymbol{A}%
_{M}^{N}$ with respect to the measure $\mu _{\Omega _{M}^{N}}$. When we wish
to emphasize the variable grid $\mathcal{D}$ being averaged in $\boldsymbol{E%
}_{\boldsymbol{A}_{M}^{N}}$ we will include this as a superscript in the
notation $\boldsymbol{E}_{\boldsymbol{A}_{M}^{N}}^{\mathcal{D}}$ (in order
to avoid confusion with any other variable grids $\mathcal{G}$ that might be
in consideration). Define the collection 
\begin{equation*}
\left( \boldsymbol{A}_{M}^{N}\right) _{I}\equiv \left\{ \mathcal{G}\in 
\boldsymbol{A}_{M}^{N}:I\in \mathcal{G}\right\}
\end{equation*}%
of all dyadic grids $\mathcal{G}$ in $\boldsymbol{A}_{M}^{N}$ that contain
the cube $I$. Then we claim that for $p\left( \mathcal{D}\right) \equiv
\sum_{I\in \mathcal{D}}q\left( I\right) $ where $q:\mathcal{S}%
_{M}^{N}\rightarrow \left[ 0,\infty \right) $, we have the following
identity by Fubini's theorem:%
\begin{eqnarray}
\boldsymbol{E}_{\boldsymbol{A}_{M}^{N}}p &=&\int_{\boldsymbol{A}%
_{M}^{N}}p\left( \mathcal{D}\right) d\mu \left( \mathcal{D}\right) =\int_{%
\boldsymbol{A}_{M}^{N}}\left( \sum_{I\in \mathcal{D}}q\left( I\right)
\right) d\mu _{M}^{N}\left( \mathcal{D}\right)  \label{cond form} \\
&=&\sum_{I\in \mathcal{S}_{M}^{N}}q\left( I\right) \int_{\left( \boldsymbol{A%
}_{M}^{N}\right) _{I}}d\mu _{M}^{N}\left( \mathcal{D}\right) =\sum_{I\in 
\mathcal{S}_{M}^{N}}q\left( I\right) \boldsymbol{P}_{\boldsymbol{A}%
_{M}^{N}}\left( \left( \boldsymbol{A}_{M}^{N}\right) _{I}\right) \ .  \notag
\end{eqnarray}%
This identity can be rigorously proved simply by using the construction in
the previous subsubsection and writing out explicitly the sums involved.
Note however, that we make crucial use of the fact that counting measure on $%
\mathcal{S}_{M}^{N}$ is $\sigma $-finite, so that Fubini's theorem applies%
\footnote{%
The analogous assertion that $\boldsymbol{E}_{\Phi _{M}^{N}}p=\sum_{I\in 
\mathcal{P}_{M}^{N}}q\left( I\right) \boldsymbol{P}_{\Phi _{M}^{N}}\left(
\left( \Phi _{M}^{N}\right) _{I}\right) $, where $\Phi _{M}^{N}$ is given by
(\ref{def PHI M,N}), \emph{fails} because counting measure on $\mathcal{P}%
_{M}^{N}$ is not $\sigma $-finite, and this explains our ubiquitous use of
the \emph{finite} collections of grids $\Omega _{M}^{N}$.}. Here are the
details.

If we consider the parameterization of the family of grids $\mathcal{D}$ in $%
\boldsymbol{A}_{M}^{N}$ by scale as above, then the expectation of a
quantity $p\left( \mathcal{D}\right) $, defined for all grids $\mathcal{D}%
\in \boldsymbol{A}_{M}^{N}$, is given by%
\begin{equation*}
\boldsymbol{E}_{\Omega }p\equiv \frac{1}{\#\Omega _{M}^{N}}\sum_{\beta \in
\Omega _{M}^{N}}p\left( \mathcal{D}_{\beta }\right) =\frac{1}{\#\boldsymbol{A%
}_{M}^{N}}\sum_{\mathcal{D}\in \boldsymbol{A}_{M}^{N}}p\left( \mathcal{D}%
\right) .
\end{equation*}%
A special case arises for a function $q:\mathcal{S}_{M}^{N}\rightarrow \left[
0,\infty \right) $ defined on cubes in $\mathcal{S}_{M}^{N}$, if we set%
\begin{equation*}
p\left( \mathcal{D}\right) \equiv \sum_{I\in \mathcal{D}\cap \mathcal{S}%
_{M}^{N}}q\left( I\right) ,\ \ \ \ \ \text{\ for all }\mathcal{D}\in 
\boldsymbol{A}_{M}^{N}.
\end{equation*}%
Then with the subset 
\begin{equation*}
\Theta _{M}^{N}\equiv \left\{ \left( I,\mathcal{D}\right) \in \mathcal{S}%
_{M}^{N}\times \boldsymbol{A}_{M}^{N}:I\in \mathcal{D}\right\}
\end{equation*}%
of the product $\mathcal{S}_{M}^{N}\times \boldsymbol{A}_{M}^{N}$, we can
write%
\begin{eqnarray*}
\boldsymbol{E}_{\Omega _{M}^{N}}p &=&\frac{1}{\#\boldsymbol{A}_{M}^{N}}\sum_{%
\mathcal{D}\in \boldsymbol{A}_{M}^{N}}p\left( \mathcal{D}\right) =\frac{1}{\#%
\boldsymbol{A}_{M}^{N}}\sum_{\mathcal{D}\in \boldsymbol{A}%
_{M}^{N}}\sum_{I\in \mathcal{D}}q\left( I\right) \\
&=&\frac{1}{\#\boldsymbol{A}_{M}^{N}}\sum_{\left( I,\mathcal{D}\right) \in 
\mathcal{S}_{M}^{N}\times \boldsymbol{A}_{M}^{N}}\mathbf{1}_{\Theta
_{M}^{N}}\left( \left( I,\mathcal{D}\right) \right) q\left( I\right) \\
&=&\frac{1}{\#\boldsymbol{A}_{M}^{N}}\sum_{I\in \mathcal{S}_{M}^{N}}\sum_{%
\mathcal{D}\in \left( \boldsymbol{A}_{M}^{N}\right) _{I}}q\left( I\right)
=\sum_{I\in \mathcal{S}_{M}^{N}}q\left( I\right) \frac{\#\left( \boldsymbol{A%
}_{M}^{N}\right) _{I}}{\#\boldsymbol{A}_{M}^{N}}\ .
\end{eqnarray*}%
Later, in the estimate (\ref{alto}) near the end of the paper, we will take
a limit as $M\rightarrow \infty $ and $N\rightarrow -\infty $.

We now illustrate, in a simple situation, the type of conditional estimate
we will use in our proof below. For $I\in \mathcal{S}_{M}^{N}$ let $%
\boldsymbol{P}_{\left( \boldsymbol{A}_{M}^{N}\right) _{I}}$ denote uniform
probability on the finite set $\left( \boldsymbol{A}_{M}^{N}\right) _{I}$.
For $B\subset \Theta _{M}^{N}$ we have by definition%
\begin{equation*}
\boldsymbol{P}_{\left( \boldsymbol{A}_{M}^{N}\right) _{I}}\left( \left\{ 
\mathcal{D}\in \left( \boldsymbol{A}_{M}^{N}\right) _{I}:\left( I,\mathcal{D}%
\right) \in B\right\} \right) =\mu _{\left( \boldsymbol{A}_{M}^{N}\right)
_{I}}\left( B\cap \left( \boldsymbol{A}_{M}^{N}\right) _{I}\right) .
\end{equation*}%
Suppose that for some $\varepsilon >0$ we have 
\begin{equation*}
\boldsymbol{P}_{\left( \boldsymbol{A}_{M}^{N}\right) _{I}}\left( \left\{ 
\mathcal{D}\in \left( \boldsymbol{A}_{M}^{N}\right) _{I}:\left( I,\mathcal{D}%
\right) \in B\right\} \right) \leq \varepsilon \text{ for all }I\in \mathcal{%
S}_{M}^{N},
\end{equation*}%
and furthermore suppose we are given a nonnegative quantity $q\left(
I\right) $ that is defined for all cubes in $\mathcal{S}_{M}^{N}$. Then we
claim that%
\begin{equation}
\boldsymbol{E}_{\boldsymbol{A}_{M}^{N}}\sum_{I\in \mathcal{D}}q\left(
I\right) \mathbf{1}_{B}\left( I,\mathcal{D}\right) \leq \varepsilon 
\boldsymbol{E}_{\boldsymbol{A}_{M}^{N}}\sum_{I\in \mathcal{D}}q\left(
I\right) \mathbf{.}  \label{in part}
\end{equation}%
Indeed, to see this, we recall that our collections of truncated grids $%
\Omega _{M}^{N}$ are all finite, and write 
\begin{eqnarray*}
\boldsymbol{E}_{\boldsymbol{A}_{M}^{N}}^{\mathcal{D}}\sum_{I\in \mathcal{D}%
}q\left( I\right) \mathbf{1}_{B}\left( I,\mathcal{D}\right) &=&\frac{1}{\#%
\boldsymbol{A}_{M}^{N}}\sum_{\mathcal{D}\in \boldsymbol{A}%
_{M}^{N}}\sum_{I\in \mathcal{D}:\ \left( I,\mathcal{D}\right) \in B}q\left(
I\right) \\
&=&\frac{1}{\#\boldsymbol{A}_{M}^{N}}\sum_{I\in \mathcal{S}_{M}^{N}:}q\left(
I\right) \ \#\left\{ \mathcal{D}\in \left( \boldsymbol{A}_{M}^{N}\right)
_{I}:\ \left( I,\mathcal{D}\right) \in B\right\} \\
&=&\frac{1}{\#\boldsymbol{A}_{M}^{N}}\sum_{I\in \mathcal{S}_{M}^{N}:}q\left(
I\right) \ \mu _{\left( \boldsymbol{A}_{M}^{N}\right) _{I}}\left( B\cap
\left( \boldsymbol{A}_{M}^{N}\right) _{I}\right) \ \#\left\{ \mathcal{D}\in
\left( \boldsymbol{A}_{M}^{N}\right) _{I}:\ \left( I,\mathcal{D}\right) \in
\Theta \right\} \\
&\leq &\frac{1}{\#\boldsymbol{A}_{M}^{N}}\sum_{I\in \mathcal{S}%
_{M}^{N}:}q\left( I\right) \ \varepsilon \ \#\left( \boldsymbol{A}%
_{M}^{N}\right) _{I} \\
&=&\varepsilon \frac{1}{\#\boldsymbol{A}_{M}^{N}}\sum_{\mathcal{D}\in 
\boldsymbol{A}_{M}^{N}}\sum_{I\in \mathcal{D}}q\left( I\right) =\varepsilon 
\boldsymbol{E}_{\boldsymbol{A}_{M}^{N}}\sum_{I\in \mathcal{D}}q\left(
I\right) \mathbf{.}
\end{eqnarray*}%
A similar expectation argument, but complicated by a subtle point regarding
Whitney grids, will be carried out in (\ref{comp}) below.

\subsubsection{Control of the maximal function by dyadic operators\label%
{Subsub domination}}

Recall the finite collections of dyadic grids $\Omega _{M}^{N}$
(equivalently parameterized by $\Gamma _{M}^{N}$) introduced in
Subsubsection \ref{Subsub parameter}, and especially the decomposition (\ref%
{slice decomp}). In particular, construction \#2 in Subsubsection \ref%
{Subsub parameter} shows that%
\begin{equation*}
\boldsymbol{A}_{M}^{N}=\left\{ \left( \mathcal{D}_{0}\right)
_{M}^{N}+t\right\} _{t\in \Omega _{M}^{N}}
\end{equation*}%
where $\mathcal{D}_{0}:=\{2^{j}([0,1)+k),j\in \mathbb{Z},k\in \mathbb{Z}%
^{n}\}$ is the \emph{standard dyadic grid} in $\mathbb{R}^{n}$, $\left( 
\mathcal{D}_{0}\right) _{M}^{N}$ consists of those cubes $Q$ in $\mathcal{D}%
_{0}$ with side lengths between $2^{-M}$ and $2^{-N}$, and where $\Omega
_{M}^{N}$ is the index set%
\begin{equation*}
\Omega _{M}^{N}\equiv \left\{ t=\left( t_{i}\right) _{i=0}^{\infty }\in
2^{-M}\mathbb{Z}_{+}:\left\vert t_{i}\right\vert <2^{-N}\right\} .
\end{equation*}%
Recall also that we denoted by $d\boldsymbol{P}_{\boldsymbol{A}_{M}^{N}}$
the uniform probability measure on the finite set $\boldsymbol{A}_{M}^{N}$.

\begin{notation}
We will now abuse notation by identifying the collection of dyadic grids $%
\boldsymbol{A}_{M}^{N}$ with its associated index set $\Omega _{M}^{N}$.
Thus we henceforth abandon the notation $\boldsymbol{A}_{M}^{N}$ and write $%
\Omega _{M}^{N}$ for the finite collection of dyadic grids built on $\mathbb{%
D}_{M}\times ...\times \mathbb{D}_{M}$ of size $2^{-N}$.
\end{notation}

We now denote the natural product probability measure on the (infinite)
collection of truncated dyadic grids 
\begin{equation}
\Phi _{M}^{N}\equiv \dbigcup\limits_{s\in \left[ 0,2^{-M}\right) }\left(
\Omega _{M}^{N}+s\right)  \label{def PHI M,N}
\end{equation}%
by $d\boldsymbol{P}_{\Phi _{M}^{N}}$. More specifically, $d\boldsymbol{P}%
_{\Phi _{M}^{N}}$ is defined to be the product measure $d\boldsymbol{P}%
_{\Omega _{M}^{N}}\left( \mathcal{D}_{\limfunc{fin}}\right) \times 2^{-Mn}%
\mathbf{1}_{\left[ 0,2^{-M}\right) ^{n}}\left( s\right) ds$ on $\Phi
_{M}^{N}=\left\{ \mathcal{D}_{\limfunc{fin}}+s:\mathcal{D}_{\limfunc{fin}%
}\in \Omega _{M}^{N}\text{ and }s\in \left[ 0,2^{-M}\right) ^{n}\right\} $.
Note that $\Omega _{M}^{N}+s=\Omega _{M}^{N}+s^{\prime }$ if $s-s^{\prime
}\in 2^{-M}\mathbb{Z}^{n}$, and that we can then also write%
\begin{equation*}
\Phi _{M}^{N}=\dbigcup\limits_{t\in \gamma _{M}^{N}}\dbigcup\limits_{s\in 
\left[ 0,2^{-M}\right) ^{n}}\left\{ \left( \mathcal{D}_{0}\right)
_{M}^{N}+t+s\right\} .
\end{equation*}

\begin{notation}
We are here using $\mathcal{D}_{\limfunc{fin}}$ to denote an \emph{%
independent} variable in the collection of finite dyadic grids $\Phi
_{M}^{N} $, so that - unlike the notation $\mathcal{D}_{M}^{N}$, which
depends on the choice $\mathcal{D}$ of an untruncated dyadic grid in $\Omega 
$ - there is for $\mathcal{D}_{\limfunc{fin}}$ \emph{no connection} implied
with an untruncated dyadic grid $\mathcal{D}$ in $\Omega $. We will also use 
$\mathcal{D}_{\limfunc{fin}}$ below to denote an independent variable in the
larger collection of truncated dyadic grids $\Phi _{M}^{N}$.
\end{notation}

Then for each truncated dyadic grid $\mathcal{D}_{\limfunc{fin}}\in \Phi
_{M}^{N}$, we denote the natural probability measure on the collection of
untruncated dyadic grids 
\begin{equation*}
\mathcal{H}_{\mathcal{D}_{\limfunc{fin}}}\equiv \left\{ \mathcal{D}\in
\Omega :\mathcal{D}_{M}^{N}=\mathcal{D}_{\limfunc{fin}}\right\}
\end{equation*}%
by $d\boldsymbol{P}_{\mathcal{H}_{\mathcal{D}_{\limfunc{fin}}}}$. More
specifically, if $\mathcal{D}_{\limfunc{fin}}\in \Phi _{M}^{N}$ and $%
\mathcal{D}^{\beta }\in \Omega $ as in (\ref{def infinite dyadic grid}) is
any fixed untruncated grid in $\Omega $ such that $\left( \mathcal{D}^{\beta
}\right) _{M}^{N}=\mathcal{D}_{\limfunc{fin}}$, then the set $\mathcal{H}_{%
\mathcal{D}_{\limfunc{fin}}}$ is given by%
\begin{equation*}
\mathcal{H}_{\mathcal{D}_{\limfunc{fin}}}=\left\{ \mathcal{D}^{\gamma }\in
\Omega :\gamma _{j}=\beta _{j}\text{ for all }j>N\right\} ,
\end{equation*}%
i.e. $\mathcal{H}_{\mathcal{D}_{\limfunc{fin}}}$ consists of all grids $%
\mathcal{D}^{\gamma }$ whose tiling by cubes of side length $2^{-N}$ agrees
with that of $\mathcal{D}^{\beta }$. The probability measure $d\boldsymbol{P}%
_{\mathcal{H}_{\mathcal{D}_{\limfunc{fin}}}}$ is that unique probability
measure which assigns\ equal probability $2^{-nk}$ to each collection of
grids indexed by the set $S_{\left( \beta _{N-k},\beta _{N-\left( k-1\right)
},...,\beta _{N-1}\right) }$ of indices%
\begin{equation*}
S_{\left( \beta _{N-k},\beta _{N-\left( k-1\right) },...,\beta _{N-1}\right)
}\equiv \left\{ \gamma \in \left( \left\{ 0,1\right\} ^{n}\right) ^{\mathbb{Z%
}}:\gamma _{j}=\beta _{j}\text{ for all }j>N\text{ and }\gamma _{j}=\beta
_{j}\text{ for }N-k\leq j\leq N-1\right\} ,
\end{equation*}%
where $\left( \beta _{N-k},\beta _{N-\left( k-1\right) },...,\beta
_{N-1}\right) \in \left( \left\{ 0,1\right\} ^{n}\right) ^{k}$ has length $k$%
. These probability measures $d\boldsymbol{P}_{\mathcal{H}_{\mathcal{D}_{%
\limfunc{fin}}}}$ are translation invariant in the sense that%
\begin{equation*}
d\boldsymbol{P}_{\mathcal{H}_{\mathcal{D}_{\limfunc{fin}}}+s}=d\boldsymbol{P}%
_{\mathcal{H}_{\mathcal{D}_{\limfunc{fin}}}}\text{ for }s\in \left[
0,2^{-M}\right) ^{n}.
\end{equation*}

For each choice of integers $N<0<M$, we thus have%
\begin{equation}
\Omega =\dbigcup\limits_{\mathcal{D}_{\limfunc{fin}}\in \Phi _{M}^{N}}%
\mathcal{H}_{\mathcal{D}_{\limfunc{fin}}}=\dbigcup\limits_{\mathcal{D}_{%
\limfunc{fin}}\in \Omega _{M}^{N}}\dbigcup\limits_{s\in \left[
0,2^{-M}\right) ^{n}}\mathcal{H}_{\mathcal{D}_{\limfunc{fin}%
}+s}=\dbigcup\limits_{t\in \gamma _{M}^{N}}\dbigcup\limits_{s\in \left[
0,2^{-M}\right) ^{n}}\mathcal{H}_{\left( \mathcal{D}_{0}\right)
_{M}^{N}+t+s}=\dbigcup\limits_{t\in \gamma _{M}^{N}}\dbigcup\limits_{s\in %
\left[ 0,2^{-M}\right) ^{n}}\mathcal{H}_{M}^{N}+t+s\ ,  \label{slice decomp'}
\end{equation}%
where we have set $\mathcal{H}_{M}^{N}\equiv \mathcal{H}_{\left( \mathcal{D}%
_{0}\right) _{M}^{N}}$, the set of dyadic grids $\mathcal{D}\in \Omega $
that agree with the standard grid $\mathcal{D}_{0}$ at level $N$, i.e. that
share the same tiling of cubes with side length $2^{-N}$. For any quantity $%
p\left( \mathcal{D}\right) $ that is defined for all grids $\mathcal{D}\in
\Omega $, and for each choice of integers $N<0<M$, we thus have%
\begin{eqnarray}
\boldsymbol{E}_{\Omega }p &=&\int_{\Omega }p\left( \mathcal{D}\right) d%
\boldsymbol{P}_{\Omega }\left( \mathcal{D}\right)  \label{slice exp} \\
&=&\int_{\Omega _{M}^{N}}\left[ \int_{\left[ 0,2^{-M}\right) ^{n}}\left(
\int_{\mathcal{H}_{M}^{N}}p\left( \mathcal{D}+t+s\right) d\boldsymbol{P}_{%
\mathcal{H}_{M}^{N}}\left( \mathcal{D}\right) \right) \frac{ds}{2^{-Mn}}%
\right] d\boldsymbol{P}_{\Omega _{M}^{N}}\left( t\right)  \notag
\end{eqnarray}%
by Fubini's theorem, since the measure $d\boldsymbol{P}_{\Omega }$ is the
product measure $d\boldsymbol{P}_{\gamma _{M}^{N}}\times \mathbf{1}_{\left[
0,2^{-M}\right) ^{n}}\frac{ds}{2^{-Mn}}\times d\boldsymbol{P}_{\mathcal{H}%
_{M}^{N}}$ on $\Omega _{M}^{N}\times \left[ 0,2^{-M}\right) ^{n}\times 
\mathcal{H}_{M}^{N}$, and where $d\boldsymbol{P}_{\Omega _{M}^{N}}$ is of
course a finite convex sum of unit point masses. We also then have%
\begin{equation*}
\boldsymbol{E}_{\Omega }^{\mathcal{D}}p\left( \mathcal{D}\right) =\int_{\Phi
_{M}^{N}}\left\{ \int_{\mathcal{H}_{\mathcal{D}_{\limfunc{fin}}}}p\left( 
\mathcal{D}\right) d\boldsymbol{P}_{\mathcal{H}_{\mathcal{D}_{\limfunc{fin}%
}}}\left( \mathcal{D}\right) \right\} d\boldsymbol{P}_{\Phi _{M}^{N}}\left( 
\mathcal{D}_{\limfunc{fin}}\right) .
\end{equation*}

Our main result in this subsubsection is the following lemma, which goes
back to Fefferman and Stein \cite[page 112]{FeSt} and also \cite[Lemma 2]%
{Saw3}. For any dyadic grid $\mathcal{D}\in \Omega $, we denote the
associated \emph{dyadic} maximal operator by 
\begin{equation*}
\mathcal{M}^{\mathcal{D}}f\equiv \sup_{Q\in \mathcal{D}}\frac{1}{\left\vert
Q\right\vert }\int_{Q}\left\vert f\right\vert .
\end{equation*}

\begin{lemma}
\label{domination}For $x\in \mathbb{R}^{n}$ and a positive Borel measure $%
f\geq 0$ on $\mathbb{R}^{n}$ we have%
\begin{equation}
\mathcal{M}f\left( x\right) \leq 2^{n+3}\boldsymbol{E}_{\Omega }^{\mathcal{D}%
}\mathcal{M}^{\mathcal{D}}f\left( x\right) .  \label{lim inf average}
\end{equation}
\end{lemma}

\begin{proof}
Fix $x\in \mathbb{R}^{n}$, and let $Q\in \mathcal{P}$ be such that $x\in Q$
and 
\begin{equation*}
\frac{1}{\left\vert Q\right\vert }\int_{Q}f>\frac{1}{2}\mathcal{M}f\left(
x\right) .
\end{equation*}%
Pick $N<0<1<M$ so that $2^{-M+100}\leq \ell \left( Q\right) \leq 2^{-N-100}$%
, which implies in particular that there exists $s^{\prime }\in \left[
0,2^{-M}\right) $ and $\mathcal{D}_{\limfunc{fin}}\in \Omega
_{M}^{N}+s^{\prime }$ such that $Q\in \mathcal{D}_{\limfunc{fin}}$. Thus the
truncated grid $\mathcal{D}_{\limfunc{fin}}\equiv \mathcal{D}_{\limfunc{fin}%
}-s^{\prime }$ belongs to the \emph{special} collection $\Omega _{M}^{N}$ of
truncated grids in Subsubsection \ref{Subsub parameter}, and $Q-s^{\prime }$
belongs to the corresponding special collection of cubes $\mathcal{S}%
_{M}^{N} $. Denote by $\Delta _{s^{\prime }}$ the set of indices $t\in
\Omega _{M}^{N} $ such that the translated grid $\left( \mathcal{D}%
_{0}\right) _{M}^{N}+t$ as given by Construction \#2 above has a cube $K$
with side length twice that of $Q$, and that contains $Q$. For such a cube $%
K $ we have $\frac{1}{\left\vert K\right\vert }\int_{K}f\geq \frac{1}{%
2^{n}\left\vert Q\right\vert }\int_{Q}f>\frac{1}{2^{n+1}}\mathcal{M}f(x)$.
Moreover, the set $\Delta _{s^{\prime }}$ has probability $\mu
_{M}^{N}\left( \Delta _{s^{\prime }}\right) \geq \frac{1}{2}$. Thus we have%
\begin{eqnarray*}
\int_{\Omega _{M}^{N}}\mathcal{M}^{\mathcal{D}_{\limfunc{fin},s^{\prime
}}+\beta }f\left( x\right) d\mu _{M}^{N}\left( \beta \right) &\geq
&\int_{\beta \in \Delta _{s^{\prime }}}\mathcal{M}^{\mathcal{D}_{\limfunc{fin%
},s^{\prime }}+\beta }f\left( x\right) d\mu _{M}^{N}\left( \beta \right) \\
&\geq &\int_{\beta \in \Delta _{s^{\prime }}}\frac{1}{2^{n+1}}\mathcal{M}%
f\left( x\right) d\mu _{M}^{N}\left( \beta \right) =\frac{\mu _{M}^{N}\left(
\Delta _{s^{\prime }}\right) }{2^{n+1}}\mathcal{M}f\left( x\right) \geq 
\frac{1}{2^{n+2}}\mathcal{M}f\left( x\right) .
\end{eqnarray*}%
Now using that $2^{-M+100}\leq \ell \left( Q\right) \leq 2^{-N-100}$, we
easily see from the geometry of the cubes and grids that for \emph{every} $%
s\in \left[ 0,2^{-M}\right) ^{n}$ and any $\mathcal{D}_{\limfunc{fin}}\in
\Omega _{M}^{N}+s$, we have%
\begin{equation*}
\int_{\Omega _{M}^{N}}\mathcal{M}^{\mathcal{D}_{\limfunc{fin}}+\beta
}f\left( x\right) d\mu _{M}^{N}\left( \beta \right) \geq \frac{\mu
_{M}^{N}\left( \Delta _{s}\right) }{2^{n+1}}\mathcal{M}f\left( x\right) \geq 
\frac{1}{2^{n+3}}\mathcal{M}f\left( x\right) ,
\end{equation*}%
upon using the crude estimate $\mu _{M}^{N}\left( \Delta _{s}\right) \geq
\mu _{M}^{N}\left( \Delta _{s^{\prime }}\right) -\frac{1}{4}\geq \frac{1}{4}$%
. Taking the average over $s$ in $\left[ 0,2^{-M}\right) ^{n}$ and using (%
\ref{slice decomp'})\ and (\ref{slice exp}) gives%
\begin{eqnarray*}
\int_{\Omega }\mathcal{M}^{\mathcal{D}}f\left( x\right) d\boldsymbol{P}%
_{\Omega }\left( \mathcal{D}\right) &=&\int_{\Omega _{M}^{N}}\left[ \int_{%
\left[ 0,2^{-M}\right) ^{n}}\left( \int_{\mathcal{H}}\mathcal{M}^{\mathcal{D}%
+t+s}f\left( x\right) d\boldsymbol{P}_{\mathcal{H}}\left( \mathcal{D}\right)
\right) \frac{ds}{2^{-Mn}}\right] d\boldsymbol{P}_{\Omega _{M}^{N}}\left(
t\right) \\
&=&\int_{\Omega _{M}^{N}}\left[ \int_{\left[ 0,2^{-M}\right) ^{n}}\left(
\int_{\mathcal{H}}\mathcal{M}^{\mathcal{D}+t+s}f\left( x\right) d\boldsymbol{%
P}_{\mathcal{H}}\left( \mathcal{D}\right) \right) \frac{ds}{2^{-Mn}}\right] d%
\boldsymbol{P}_{\Omega _{M}^{N}}\left( t\right) \\
&\geq &\int_{\left[ 0,2^{-M}\right) ^{n}}\left[ \int_{\Omega _{M}^{N}}\left(
\int_{\mathcal{H}}\mathcal{M}^{\mathcal{D}+t+s}f\left( x\right) d\boldsymbol{%
P}_{\mathcal{H}}\left( \mathcal{D}\right) \right) d\boldsymbol{P}_{\Omega
_{M}^{N}}\left( t\right) \right] \frac{ds}{2^{-Mn}} \\
&\geq &\int_{\left[ 0,2^{-M}\right) ^{n}}\frac{\mu _{M}^{N}\left( \Delta
_{s}\right) }{2^{n+1}}\mathcal{M}f\left( x\right) \frac{ds}{2^{-Mn}}\geq 
\frac{1}{2^{n+3}}\mathcal{M}f\left( x\right) ,
\end{eqnarray*}

which completes the proof of (\ref{lim inf average}).
\end{proof}

\subsection{Whitney decompositions}

Fix a finite measure $\nu $ with compact support on $\mathbb{R}^{n}$, and
for $k\in \mathbb{Z}$ let%
\begin{equation}
\Omega _{k}=\left\{ x\in \mathbb{R}^{n}:\mathcal{M}\nu (x)>2^{k}\right\} .
\label{e.OmegaK}
\end{equation}%
Note that $\Omega _{k}\neq \mathbb{R}^{n}$ is open for such $\nu $. Fix a
dyadic grid $\mathcal{D}\in \Omega $ and an integer $N\geq 5$ (not to be
confused with the \emph{different} integer $N$ in Subsubsection \ref{Subsub
parameter} above). We can choose $R_{W}\geq 3$ sufficiently large, depending
only on the dimension and $N$, such that there is a collection of $\mathcal{D%
}$-dyadic cubes $\left\{ Q_{j}^{k}\right\} _{j}$ which satisfy the following
properties for some positive constant $C_{W}$: 
\begin{equation}
\left\{ 
\begin{array}{ll}
\text{(disjoint cover)} & \Omega _{k}=\bigcup_{j}Q_{j}^{k}\text{ and }%
Q_{j}^{k}\cap Q_{i}^{k}=\emptyset \text{ if }i\neq j \\ 
\text{(Whitney condition)} & R_{W}Q_{j}^{k}\subset \Omega _{k}\text{ and }%
3R_{W}Q_{j}^{k}\cap \Omega _{k}^{c}\neq \emptyset \text{ for all }k,j \\ 
\text{(bounded overlap)} & \sum_{j}\chi _{NQ_{j}^{k}}\leq C_{W}\chi _{\Omega
_{k}}\text{ for all }k \\ 
\text{(crowd control)} & \#\left\{ Q_{s}^{k}:Q_{s}^{k}\cap NQ_{j}^{k}\neq
\emptyset \right\} \leq C_{W}\text{ for all }k,j \\ 
\text{(side length comparability)} & \frac{1}{2}\leq \frac{\ell \left(
Q_{j}^{k}\right) }{\ell \left( Q_{s}^{k}\right) }\leq 2\text{ if }%
3Q_{j}^{k}\cap 3Q_{s}^{k}\neq \emptyset \\ 
\text{(nested property)} & Q_{j}^{k}\varsubsetneqq Q_{i}^{\ell }\text{
implies }k>\ell%
\end{array}%
\right. .  \label{Whitney}
\end{equation}

Indeed, one can choose the $\left\{ Q_{j}^{k}\right\} _{j}$ from $\mathcal{D}
$ to satisfy an appropriate Whitney condition, and then show that the other
properties hold. This Whitney decomposition and its use below are derived
from work of C. Fefferman predating the two weight fractional integral
argument of Sawyer \cite{Saw2}. In particular, the properties above are as
in \cite{Saw2}, with the exception of the side length comparability, which
the reader can easily verify holds for $R_{W}$ chosen sufficiently large.

\section{Strong triple testing}

Now we begin the proof of Theorem \ref{maximal}, which starts along the
lines of the proof of the weaker result in \cite{LiSa}, but with the random
grids of Nazarov, Treil and Volberg \cite{NTV} used in place of the finite
collection of grids constructed in the so-called `one third trick' of Str%
\"{o}mberg.

Here is a brief description of the new features of the argument here as
compared to that in \cite{LiSa}. In \cite{LiSa}, we assumed the stronger
hypothesis of $D$-parental tripling, which meant that the testing condition
held for all cubes $Q$ satisfying the property that for \emph{at least one}
of the $2^{n}$ possible dyadic parents $P$ of $Q$, we had $\left\vert
P\right\vert _{\sigma }\leq D\left\vert Q\right\vert _{\sigma }$. Thus the
grids $\mathcal{Q}_{t,u}^{\limfunc{nontrip}}$ of nontripling cubes $Q$ in a
stopping cube $Q_{u}^{t}$ (as defined in \cite{LiSa}) were connected in the
Whitney grid $\mathcal{W}$, so that $\pi _{\mathcal{W}}Q\in \mathcal{Q}%
_{t,u}^{\limfunc{nontrip}}$ and $\left\vert Q\right\vert _{\sigma }<\frac{1}{%
D}\left\vert \pi Q\right\vert _{\sigma }$, which could then be iterated and
summed up to an acceptable Carleson estimate. In the analogous situation
here, the tripled cube $3Q$ can spill outside the stopping cube $Q_{u}^{t}$,
which is then difficult to control because the averages of $f$ outside the
stopping cube are no longer controlled by the average of $f$ over $Q_{u}^{t}$%
. This spilling out then requires control of the `$\limfunc{bad}$' cubes $%
Q\in \mathcal{W}$ whose triples are not contained in $Q_{u}^{t}$. This
control is effected by averaging over dyadic grids much as in \cite{NTV},
but is complicated by the fact that our cubes are contained in the subgrid
of \emph{Whitney} cubes, which necessitates some combinatoric arguments with
finite grids.

We wish to prove the following estimate with $\Gamma =3$ in the restricted
testing condition,%
\begin{equation}
\mathfrak{N}_{\mathcal{M}}\left( \sigma ,\omega \right) \lesssim \mathfrak{T}%
_{\mathcal{M}}\left( 3\right) \left( \sigma ,\omega \right) +\sqrt{%
A_{2}\left( \sigma ,\omega \right) }.  \label{it suffices}
\end{equation}%
Fix $f$ nonnegative and bounded with compact support, say $\mathop{\rm{supp}}%
f\subset Q(0,R)=\left[ -R,R\right] ^{n}$. Since $\mathcal{M}\left( f\sigma
\right) $ is lower semicontinuous, the set $\Omega _{k}\equiv \left\{ 
\mathcal{M}\left( f\sigma \right) >2^{k}\right\} $ is open and we can
consider the standard Whitney decomposition of the open set $\Omega _{k}$
into the union $\bigcup\limits_{j\in \mathbb{N}}Q_{j}^{k}$ of $\mathcal{D}%
^{\gamma }$-dyadic intervals $Q_{j}^{k}$ with bounded overlap and packing
properties as in (\ref{Whitney}). We denote the Whitney collection $\left\{
Q_{j}^{k}\right\} $ by $\mathcal{W}^{\gamma }$. We now use random grids to
obtain from Lemma \ref{domination} in Subsubsection \ref{Subsub domination}
that%
\begin{equation*}
\mathcal{M}\left( f\sigma \right) \left( x\right) \lesssim \boldsymbol{E}%
_{\Omega }^{\mathcal{D}^{\gamma }}\mathcal{M}^{\mathcal{D}^{\gamma }}f\left(
x\right) \ ,\ \ \ \ \ x\in \mathbb{R}^{n}.
\end{equation*}

Notice that if we replace $\omega $ by $\omega _{N}=\omega \mathbf{1}%
_{Q(0,N)}$ with $N>R$, we have 
\begin{equation*}
\int \mathcal{M}\left( f\sigma \right) ^{2}d\omega _{N}\leq \Vert f\Vert
_{L^{\infty }}^{2}\int_{Q(0,N)}\mathcal{M}(\mathbf{1}_{Q(0,N)}\sigma
)^{2}d\omega \leq \Vert f\Vert _{L^{\infty }}^{2}\mathfrak{T}_{\mathcal{M}%
}^{2}\left\vert 3Q(0,N)\right\vert _{\sigma }<\infty ,
\end{equation*}%
and therefore, without loss of generality, we can assume 
\begin{equation*}
\int \mathcal{M}(f\sigma )^{2}d\omega <\infty .
\end{equation*}

We now have 
\begin{align*}
\boldsymbol{E}_{\Omega }\int_{\mathbb{R}^{n}}\left[ \mathcal{M}^{\mathcal{D}%
^{\gamma }}\left( f\sigma \right) \left( x\right) \right] ^{2}d\omega \left(
x\right) & \leq \boldsymbol{E}_{\Omega }C_{n}\sum_{k\in \mathbb{Z}%
}2^{2(k+m)}\left\vert \left\{ \mathcal{M}^{\mathcal{D}^{\gamma }}\left(
f\sigma \right) >2^{k+m}\right\} \right\vert _{\omega } \\
& =\boldsymbol{E}_{\Omega }C_{n}\sum_{k\in \mathbb{Z},\ j\in \mathbb{N}%
}2^{2(k+m)}\left\vert Q_{j}^{k}\cap \Omega _{k+m}^{\gamma }\right\vert
_{\omega } \\
& \leq C_{n,m}\boldsymbol{E}_{\Omega }\sum_{k\in \mathbb{Z},\ j\in \mathbb{N}%
}2^{2k}\left\vert E_{j,\gamma }^{k}\right\vert _{\omega
}+3^{n}C_{n}2^{-2m_{0}}\int \left[ \mathcal{M}\left( f\sigma \right) \right]
^{2}d\omega \ ,
\end{align*}%
where%
\begin{equation*}
E_{j,\gamma }^{k}:=Q_{j}^{k}\cap \left( \Omega _{k+m}^{\gamma }\setminus
\Omega _{k+m+m_{0}}\right) ,\,\,\Omega _{k+m}^{\gamma }=\left\{ x:\mathcal{M}%
^{\mathcal{D}^{\gamma }}\left( f\sigma \right) >2^{k+m}\right\} ,
\end{equation*}%
and we shall choose $m_{0}$ to be sufficiently large so that the second term
can be absorbed (since it is finite). So the goal is to prove 
\begin{equation*}
\boldsymbol{E}_{\Omega }\sum_{k\in \mathbb{Z},\ j\in \mathbb{N}%
}2^{2k}\left\vert E_{j,\gamma }^{k}\right\vert _{\omega }\lesssim \left( 
\mathfrak{T}_{\mathcal{M}}\left( 3\right) \left( \sigma ,\omega \right)
^{2}+A_{2}\left( \sigma ,\omega \right) \right) \Vert f\Vert _{L^{2}(\sigma
)}^{2}.
\end{equation*}

Now fix $\gamma $ and we will abbreviate $E_{j,\gamma }^{k}$ by $E_{j}^{k}$.
As in \cite{LiSa} we claim the maximum principle,%
\begin{equation*}
2^{k+m-1}<\mathcal{M}^{\mathcal{D}^{\gamma }}\left( \mathbf{1}%
_{Q_{j}^{k}}f\sigma \right) \left( x\right) ,\ \ \ \ \ x\in E_{j}^{k}\ .
\end{equation*}%
Indeed, given $x\in E_{j}^{k}$, there is $Q\in \mathcal{D}^{\gamma }$ with $%
x\in Q$ and $Q\cap \left( Q_{j}^{k}\right) ^{c}\neq \emptyset $ (which
implies that $Q_{j}^{k}\subset Q$), and also $z\in \Omega _{k}^{c}$, such
that 
\begin{eqnarray*}
\mathcal{M}^{\mathcal{D}^{\gamma }}\left( \mathbf{1}_{\left(
Q_{j}^{k}\right) ^{c}}f\sigma \right) \left( x\right) &\leq &2\frac{1}{%
\left\vert Q\right\vert }\int_{Q\setminus Q_{j}^{k}}f\sigma \leq 2\frac{1}{%
\left\vert Q\right\vert }\int_{3R_{W}Q}f\sigma \\
&=&\frac{2(3R_{W})^{n}}{\left\vert 3R_{W}Q\right\vert }\int_{3R_{W}Q}f\sigma
\leq 2(3R_{W})^{n}\mathcal{M}\left( f\sigma \right) \left( z\right) \leq
2^{k+m-1}
\end{eqnarray*}%
if we choose $m>1$ large enough. Now we use $2^{k+m}<\mathcal{M}^{\mathcal{D}%
^{\gamma }}\left( f\sigma \right) \left( x\right) $ for $x\in E_{j}^{k}$ to
obtain%
\begin{equation*}
2^{k+m-1}<\mathcal{M}^{\mathcal{D}^{\gamma }}\left( f\sigma \right) \left(
x\right) -\mathcal{M}^{\mathcal{D}^{\gamma }}\left( \mathbf{1}_{\left(
Q_{j}^{k}\right) ^{c}}f\sigma \right) \left( x\right) \leq \mathcal{M}^{%
\mathcal{D}^{\gamma }}\left( \mathbf{1}_{Q_{j}^{k}}f\sigma \right) \left(
x\right) .
\end{equation*}

We now introduce some further notation which will play a crucial role below.
Let%
\begin{eqnarray*}
\mathcal{H}_{j}^{k}:= &&\left\{ \mathcal{M}^{\mathcal{D}^{\gamma }}\left( 
\mathbf{1}_{Q_{j}^{k}}f\sigma \right) >2^{k+m-1}\right\} , \\
\mathcal{H}_{j,\mathrm{in}}^{k}:= &&\left\{ \mathcal{M}^{\mathcal{D}^{\gamma
}}\left( \mathbf{1}_{Q_{j}^{k}\cap \Omega _{k+m+m_{0}}}f\sigma \right)
>2^{k+m-2}\right\} , \\
\mathcal{H}_{j,\mathrm{out}}^{k}:= &&\left\{ \mathcal{M}^{\mathcal{D}%
^{\gamma }}\left( \mathbf{1}_{Q_{j}^{k}\setminus \Omega _{k+m+m_{0}}}f\sigma
\right) >2^{k+m-2}\right\} ,
\end{eqnarray*}%
so that $\mathcal{H}_{j}^{k}\subset \mathcal{H}_{j,\mathrm{in}}^{k}\cup 
\mathcal{H}_{j,\mathrm{out}}^{k}$. We are here suppressing the dependence of 
$\mathcal{H}_{j}^{k}$ on $\gamma \in \Omega $.

We will now follow the main lines of the argument for fractional integrals
in \cite{Saw2}, but as in \cite{LiSa}, with two main changes:

\begin{enumerate}
\item \textbf{Sublinearizations}: Since $\mathcal{M}$ is not linear, the
duality arguments in \cite{Saw2} require that we construct symmetric
linearizations $L$ that are dominated by $\mathcal{M}$, and

\item \textbf{Tripling decompositions}: In order to exploit the triple
testing conditions we introduce Whitney grids, and construct stopping times
for tripling cubes, which entails some combinatorics. In particular, most of
our effort is spent on decomposing and controlling the analogue of term $IV$
from \cite{Saw2} using good and bad cubes.
\end{enumerate}

Now take $0<\beta <1$ to be chosen later, and consider the following three
exhaustive cases for $Q_{j}^{k}$ and $E_{j}^{k}$.

(\textbf{1}): $\left\vert E_{j}^{k}\right\vert _{\omega }<\beta \left\vert
3Q_{j}^{k}\right\vert _{\omega }$, in which case we say $(k,j)\in \Pi _{1}$,

(\textbf{2}): $\left\vert E_{j}^{k}\right\vert _{\omega }\geq \beta
\left\vert 3Q_{j}^{k}\right\vert _{\omega }$\ and $\left\vert E_{j}^{k}\cap 
\mathcal{H}_{j,\mathrm{out}}^{k}\right\vert _{\omega }\geq \frac{1}{2}%
\left\vert E_{j}^{k}\right\vert _{\omega }$, say $(k,j)\in \Pi _{2}$,

(\textbf{3}): $\left\vert E_{j}^{k}\right\vert _{\omega }\geq \beta
\left\vert 3Q_{j}^{k}\right\vert _{\omega }$\ and $\left\vert E_{j}^{k}\cap 
\mathcal{H}_{j,\mathrm{in}}^{k}\right\vert _{\omega }\geq \frac{1}{2}%
\left\vert E_{j}^{k}\right\vert _{\omega }$, say $(k,j)\in \Pi _{3}$.

Here is a brief, and somewhat imprecise, schematic diagram of the
decompositions, with bounds in $\fbox{}$, used in this proof:%
\begin{equation*}
\fbox{$%
\begin{array}{ccccc}
\int_{\mathbb{R}^{n}}\left[ \mathcal{M}\left( f\sigma \right) \left(
x\right) \right] ^{2}d\omega \left( x\right) &  &  &  &  \\ 
\downarrow &  &  &  &  \\ 
C_{n,m}\boldsymbol{E}_{\Omega }\sum_{k\in \mathbb{Z},\ j\in \mathbb{N}%
}2^{2k}\left\vert E_{j,\gamma }^{k}\right\vert _{\omega } & + & 
3^{n}C_{n}2^{-2m_{0}}\int_{\mathbb{R}^{n}}\left[ \mathcal{M}\left( f\sigma
\right) \right] ^{2}d\omega &  &  \\ 
\downarrow &  & \fbox{$m_{0}-$absorption} &  &  \\ 
\downarrow &  &  &  &  \\ 
\boldsymbol{E}_{\Omega }\sum_{\left( k,j\right) \in \Pi
_{3}}2^{2k}\left\vert E_{j,\gamma }^{k}\right\vert _{\omega } & + & 
\boldsymbol{\sup_{\Omega }}\sum_{\left( k,j\right) \in \Pi
_{2}}2^{2k}\left\vert E_{j,\gamma }^{k}\right\vert _{\omega } & + & 
\boldsymbol{\sup_{\Omega }}\sum_{\left( k,j\right) \in \Pi
_{1}}2^{2k}\left\vert E_{j,\gamma }^{k}\right\vert _{\omega } \\ 
\downarrow &  & \fbox{$A_{2}\left\Vert f\right\Vert _{L^{2}\left( \sigma
\right) }^{2}$} &  & \fbox{$\beta -$absorption} \\ 
\downarrow &  &  &  &  \\ 
\boldsymbol{E}_{\Omega }IV_{\mathbf{r}-\limfunc{good}} & + & \boldsymbol{%
\sup_{\Omega }}\sum_{\left( t,u\right) \in \Gamma }V\left( t,u\right) & + & 
\boldsymbol{E}_{\Omega }III_{\mathbf{r}-\limfunc{bad}}^{\ast } \\ 
\fbox{$\left( \mathfrak{T}_{\mathcal{M}}\left( 3\right) ^{2}+A_{2}\right)
\left\Vert f\right\Vert _{L^{2}\left( \sigma \right) }^{2}$} &  & \fbox{$%
A_{2}\left\Vert f\right\Vert _{L^{2}\left( \sigma \right) }^{2}$} &  & \fbox{%
$\mathbf{r}-$absorption}%
\end{array}%
$}
\end{equation*}%
where the notation is defined below. The expectation $\boldsymbol{E}_{\Omega
}$ is taken over dyadic grids $\mathcal{D}^{\gamma }$ in $\Omega $,
resulting in the absorption of the term $\boldsymbol{E}_{\Omega }III_{%
\mathbf{r}-\limfunc{bad}}^{\ast }$ in the diagram, provided $\mathbf{r}$ is
chosen sufficiently large. The term $\boldsymbol{\sup_{\Omega }}\sum_{\left(
k,j\right) \in \Pi _{1}}2^{2k}\left\vert E_{j,\gamma }^{k}\right\vert
_{\omega }$ is absorbed by taking the parameter $\beta >0$ sufficiently
small, and the term $3^{n}C_{n}2^{-2m_{0}}\int_{\mathbb{R}^{n}}\left[ 
\mathcal{M}\left( f\sigma \right) \right] ^{2}d\omega $ is absorbed by
taking the parameter $m_{0}\geq 1$ sufficiently large.

\subsection{The three cases}

The first case is trivially handled, the second case is easy, and the third
case consumes most of our effort.

\textbf{Case (1)}: The treatment of case (1) is easy by absorption. Indeed, 
\begin{equation}
\sum_{(k,j)\in \Pi _{1}}2^{2k}\left\vert E_{j}^{k}\right\vert _{\omega
}\lesssim \sum_{k\in \mathbb{Z},\ j\in \mathbb{N}}2^{2k}\beta \left\vert
3Q_{j}^{k}\right\vert _{\omega }\lesssim \beta \int \mathcal{M}\left(
f\sigma \right) ^{2}d\omega ,  \label{case 1 est}
\end{equation}%
and then it suffices to take $\beta $ sufficiently small at the end of the
proof.

\textbf{Case (2)}: In case (2) we have%
\begin{equation}
\sum_{(k,j)\in \Pi _{2}}2^{2k}\left\vert E_{j}^{k}\right\vert _{\omega
}\lesssim \sum_{(k,j)\in \Pi _{2}}2^{k}\int \mathbf{1}_{E_{j}^{k}}\mathcal{L}%
_{j}^{k}\left( \mathbf{1}_{Q_{j}^{k}\setminus \Omega _{k+m+m_{0}}}f\sigma
\right) d\omega .  \label{proceed}
\end{equation}%
Here the positive linear operator $\mathcal{L}_{j}^{k}$ given by 
\begin{equation*}
\mathcal{L}_{j}^{k}\left( h\sigma \right) \left( x\right) \equiv \sum_{\ell
=1}^{\infty }\frac{1}{|I_{j}^{k}(\ell )|}\int_{I_{j}^{k}(\ell )}hd\sigma 
\mathbf{1}_{I_{j}^{k}(\ell )}(x),
\end{equation*}%
where $I_{j}^{k}\left( \ell \right) \in \mathcal{D}^{\gamma }$ are the
maximal dyadic cubes in $\mathcal{H}_{j,\mathrm{out}}^{k}$, which implies
that $\mathcal{L}_{j}^{k}(\mathbf{1}_{Q_{j}^{k}\setminus \Omega
_{k+m+m_{0}}}f\sigma )\eqsim 2^{k}\mathbf{1}_{\mathcal{H}_{j,\mathrm{out}%
}^{k}}$. Now we can continue from (\ref{proceed}) as follows:%
\begin{align*}
& \sum_{(k,j)\in \Pi _{2}}2^{k}\int_{E_{j}^{k}}\mathcal{L}_{j}^{k}\left( 
\mathbf{1}_{Q_{j}^{k}\setminus \Omega _{k+m+m_{0}}}f\sigma \right) d\omega \\
& =\sum_{(k,j)\in \Pi _{2}}2^{k}\int_{Q_{j}^{k}\setminus \Omega _{k+m+m_{0}}}%
\mathcal{L}_{j}^{k}\left( \mathbf{1}_{E_{j}^{k}}\omega \right) fd\sigma \\
& \leq \sum_{(k,j)\in \Pi _{2}}2^{k}\left( \int_{Q_{j}^{k}\setminus \Omega
_{k+m+m_{0}}}\mathcal{L}_{j}^{k}\left( \mathbf{1}_{E_{j}^{k}}\omega \right)
^{2}d\sigma \right) ^{\frac{1}{2}}\left( \int_{Q_{j}^{k}\setminus \Omega
_{k+m+m_{0}}}f^{2}d\sigma \right) ^{\frac{1}{2}} \\
& \leq \Big(\sum_{(k,j)\in \Pi _{2}}2^{2k}\int_{Q_{j}^{k}\setminus \Omega
_{k+m+m_{0}}}\mathcal{L}_{j}^{k}\left( \mathbf{1}_{E_{j}^{k}}\omega \right)
^{2}d\sigma \Big)^{\frac{1}{2}}\Big(\sum_{(k,j)\in \Pi
_{2}}\int_{Q_{j}^{k}\setminus \Omega _{k+m+m_{0}}}f^{2}d\sigma \Big)^{\frac{1%
}{2}} \\
& \leq \Big(\sum_{(k,j)\in \Pi _{2}}2^{2k}\int_{Q_{j}^{k}}\mathcal{L}%
_{j}^{k}\left( \mathbf{1}_{Q_{j}^{k}}\omega \right) ^{2}d\sigma \Big)^{\frac{%
1}{2}}\Big(\sum_{k\in \mathbb{Z}}\int_{\Omega _{k}\setminus \Omega
_{k+m+m_{0}}}f^{2}d\sigma \Big)^{\frac{1}{2}} \\
& \leq C_{m,m_{0}}A_{2}^{\frac{1}{2}}\Big(\sum_{(k,j)\in \Pi
_{2}}2^{2k}\left\vert Q_{j}^{k}\right\vert _{\omega }\Big)^{\frac{1}{2}%
}\Vert f\Vert _{L^{2}(\sigma )} \\
& \leq \beta ^{-\frac{1}{2}}C_{m,m_{0}}A_{2}^{\frac{1}{2}}\Big(%
\sum_{(k,j)\in \Pi _{2}}2^{2k}\left\vert E_{j}^{k}\right\vert _{\omega }\Big)%
^{\frac{1}{2}}\Vert f\Vert _{L^{2}(\sigma )},
\end{align*}%
where we have used the following trivial estimate 
\begin{equation}
\int_{Q_{j}^{k}}\mathcal{L}_{j}^{k}\left( \mathbf{1}_{Q_{j}^{k}}\omega
\right) ^{2}d\sigma \leq \sum_{\ell =1}^{\infty }\frac{|I_{j}^{k}(\ell
)|_{\omega }|I_{j}^{k}(\ell )|_{\sigma }}{|I_{j}^{k}(\ell )|^{2}}%
|I_{j}^{k}(\ell )\cap Q_{j}^{k}|_{\omega }\leq A_{2}|Q_{j}^{k}|_{\omega }.
\label{eq:a2ljk}
\end{equation}%
Then immediately we get 
\begin{equation}
\sum_{(k,j)\in \Pi _{2}}2^{2k}\left\vert E_{j}^{k}\right\vert _{\omega }\leq
\beta ^{-1}C_{m+m_{0}}^{2}A_{2}\Vert f\Vert _{L^{2}(\sigma )}^{2}.
\label{case 2 est}
\end{equation}

\textbf{Case (3)}: For this case, we let $\{I_{j}^{k}(\ell )\}_{\ell }$ be
the collection of the maximal dyadic cubes in $\mathcal{H}_{j,\mathrm{in}%
}^{k}$ and define $\mathcal{L}_{j}^{k}$ similarly. Then likewise, $\mathcal{L%
}_{j}^{k}(\mathbf{1}_{Q_{j}^{k}\cap \Omega _{k+m+m_{0}}}f\sigma )\eqsim 2^{k}%
\mathbf{1}_{\mathcal{H}_{j,\mathrm{in}}^{k}}$ and therefore, 
\begin{eqnarray*}
\sum_{(k,j)\in \Pi _{3}}2^{2k}\left\vert E_{j}^{k}\right\vert _{\omega }
&\lesssim &\sum_{(k,j)\in \Pi _{3}}2^{k}\int_{E_{j}^{k}}\mathcal{L}%
_{j}^{k}\left( \mathbf{1}_{Q_{j}^{k}\cap \Omega _{k+m+m_{0}}}f\sigma \right)
d\omega \\
&=&\sum_{(k,j)\in \Pi _{3}}2^{k}\int_{Q_{j}^{k}\cap \Omega _{k+m+m_{0}}}%
\mathcal{L}_{j}^{k}\left( \mathbf{1}_{E_{j}^{k}}\omega \right) fd\sigma \\
&=&\sum_{(k,j)\in \Pi _{3}}2^{k}\sum_{i\in \mathbb{N}:\
Q_{i}^{k+m+m_{0}}\subset Q_{j}^{k}}\int_{Q_{i}^{k+m+m_{0}}}\mathcal{L}%
_{j}^{k}\left( \mathbf{1}_{E_{j}^{k}}\omega \right) fd\sigma .
\end{eqnarray*}

Before moving on, let us make some observations. Since we only need to
consider $I_{j}^{k}(\ell )$ such that $I_{j}^{k}(\ell )\cap E_{j}^{k}\neq
\emptyset $, we have $I_{j}^{k}(\ell )\not\subset \Omega _{k+m+m_{0}}$.
Therefore, if we fix $Q_{i}^{k+m+m_{0}}$, only those $I_{j}^{k}(\ell )$ such
that $Q_{i}^{k+m+m_{0}}\subset I_{j}^{k}(\ell )$ contribute to $\mathcal{L}%
_{j}^{k}$. In other words, $\mathcal{L}_{j}^{k}\left( \mathbf{1}%
_{E_{j}^{k}}\omega \right) $ is constant on $Q_{i}^{k+m+m_{0}}$. Set 
\begin{equation}
A_{j}^{k}=\frac{1}{\left\vert Q_{j}^{k}\right\vert _{\sigma }}%
\int_{Q_{j}^{k}}fd\sigma .  \label{not con}
\end{equation}%
We have 
\begin{align*}
& \sum_{(k,j)\in \Pi _{3}}2^{2k}\left\vert E_{j}^{k}\right\vert _{\omega } \\
& \lesssim \sum_{(k,j)\in \Pi _{3}}2^{k}\sum_{i\in \mathbb{N}:\
Q_{i}^{k+m+m_{0}}\subset Q_{j}^{k}}A_{i}^{k+m+m_{0}}\int_{Q_{i}^{k+m+m_{0}}}%
\mathcal{L}_{j}^{k}\left( \mathbf{1}_{E_{j}^{k}}\omega \right) \sigma \\
& =\lim_{N\rightarrow -\infty }\sum_{\substack{ k\in \mathbb{Z},k\geq N  \\ %
j\in \mathbb{N},(k,j)\in \Pi _{3}}}2^{k}\sum_{i\in \mathbb{N}:\
Q_{i}^{k+m+m_{0}}\subset Q_{j}^{k}}A_{i}^{k+m+m_{0}}\int_{Q_{i}^{k+m+m_{0}}}%
\mathcal{L}_{j}^{k}\left( \mathbf{1}_{E_{j}^{k}}\omega \right) \sigma .
\end{align*}%
We make a convention that the summation over $k$ is understood as $k\equiv
k_{0}\,\mathrm{mod}\,(m+m_{0})$ for some fixed $0\leq k_{0}\leq m+m_{0}-1$,
and since we are summing over products with factor $|E_{j}^{k}|_{\omega }$,
without loss of generality we only consider $Q_{j}^{k}$ for the largest $k$
if it is repeated, and define 
\begin{equation*}
\mathcal{W}^{\gamma }:=\{Q_{j}^{k}:k\equiv k_{0}\,\mathrm{mod}\,\left(
m+m_{0}\right) ,k\geq N\}.
\end{equation*}%
So in particular, there are no repeated cubes in $\mathcal{W}^{\gamma }$,
and $\mathcal{W}^{\gamma }$ is in one-to-one correspondence with the set $W_{%
\limfunc{dis}}^{\gamma }$ of distinguished pairs $\left( k,j\right) $ where $%
k$ is the largest $k$ among repeated cubes, and $k\equiv k_{0}\,\mathrm{mod}%
\,\left( m+m_{0}\right) $.

\begin{notation}
We say that the cube $Q_{j}^{k}$ belongs to a set $\Lambda \subset W_{%
\limfunc{dis}}^{\gamma }$ of distinguished pairs of indices when we have $%
\left( k,j\right) \in \Lambda $, i.e. we do not distinguish between the
distinguished index $\left( k,j\right) $ and the corresponding cube $%
Q_{j}^{k}$ for $\left( k,j\right) \in W_{\limfunc{dis}}^{\gamma }$. Thus if
we write $Q\in \Lambda $, this means that $Q=Q_{j}^{k}$ for $\left(
k,j\right) \in \Lambda $, and conversely we write $Q_{j}^{k}\in \Lambda $ if 
$\left( k,j\right) \in \Lambda $.
\end{notation}

We now drop the superscript $\gamma $ when it does not matter. We have,
using $|E_{j}^{k}|_{\omega }\approx |3Q_{j}^{k}|_{\omega }$ and $\mathcal{L}%
_{j}^{k}(\mathbf{1}_{Q_{j}^{k}\cap \Omega _{k+m+m_{0}}}f\sigma )\eqsim 2^{k}%
\mathbf{1}_{\mathcal{H}_{j,\mathrm{in}}^{k}}$ for $\left( k,j\right) \in \Pi
_{3}$ again, that%
\begin{align}
& \sum_{\substack{ k\in \mathbb{Z},k\geq N  \\ j\in \mathbb{N},(k,j)\in \Pi
_{3}}}2^{k}\sum_{i\in \mathbb{N}:\ Q_{i}^{k+m+m_{0}}\subset
Q_{j}^{k}}A_{i}^{k+m+m_{0}}\int_{Q_{i}^{k+m+m_{0}}}\mathcal{L}_{j}^{k}\left( 
\mathbf{1}_{E_{j}^{k}}\omega \right) \sigma  \label{last sum} \\
& \lesssim \sum_{\substack{ k\in \mathbb{Z},k\geq N  \\ j\in \mathbb{N}}}%
\frac{|E_{j}^{k}|_{\omega }}{|3Q_{j}^{k}|_{\omega }^{2}}\left[ \sum_{i\in 
\mathbb{N}:\ Q_{i}^{k+m+m_{0}}\subset Q_{j}^{k}}\hspace{-0.4cm}%
A_{i}^{k+m+m_{0}}\int_{Q_{i}^{k+m+m_{0}}}\mathcal{L}_{j}^{k}\left( \mathbf{1}%
_{E_{j}^{k}}\omega \right) \sigma \right] ^{2}=III^{\ast },  \notag
\end{align}%
where%
\begin{eqnarray*}
III^{\ast } &\equiv &\sum_{\substack{ k\in \mathbb{Z},k\geq N  \\ j\in 
\mathbb{N}}}III^{\ast }\left( Q_{j}^{k}\right) ; \\
III^{\ast }\left( Q_{j}^{k}\right) &\equiv &\frac{|E_{j}^{k}|_{\omega }}{%
|3Q_{j}^{k}|_{\omega }^{2}}\left[ \sum_{i\in \mathbb{N}:\
Q_{i}^{k+m+m_{0}}\subset Q_{j}^{k}}\hspace{-0.4cm}A_{i}^{k+m+m_{0}}%
\int_{Q_{i}^{k+m+m_{0}}}\mathcal{L}_{j}^{k}\left( \mathbf{1}%
_{E_{j}^{k}}\omega \right) \sigma \right] ^{2}.
\end{eqnarray*}

\subsection{Control of bad cubes}

Now we encounter the main new argument needed for proving Theorem \ref%
{maximal}. Given a cube $Q$\ in a grid $\mathcal{D}=\mathcal{D}^{\gamma }$
or $\mathcal{D}_{\limfunc{fin}}$, and a large positive integer $\mathbf{r}$,
we define $Q$ to be $\mathbf{r}$-$\limfunc{bad}$ in $\mathcal{D}$ if the
level $\mathbf{r}$ parent $\pi _{\mathcal{D}}^{\left( \mathbf{r}\right) }Q$
exists in the grid $\mathcal{D}$, and the boundary of $\pi _{\mathcal{D}%
}^{\left( \mathbf{r}\right) }Q$ intersects the boundary of the tripled cube $%
3Q$ (equivalently, either $\partial Q$ `touches' $\partial \pi _{\mathcal{D}%
}^{\left( \mathbf{r}\right) }Q$ or $\partial 3Q$ `touches' $\partial \pi _{%
\mathcal{D}}^{\left( \mathbf{r}\right) }Q$). We write $\mathcal{D}=\mathcal{D%
}_{\mathbf{r}-\limfunc{bad}}\dot{\cup}\mathcal{D}_{\mathbf{r}-\limfunc{good}%
} $ where\ $\mathcal{D}_{\mathbf{r}-\limfunc{bad}}=\left\{ Q\in \mathcal{D}:Q%
\text{ is }\mathbf{r}-\limfunc{bad}\text{ in }\mathcal{D}\right\} $. Note
that this definition of $\mathbf{r}-\limfunc{bad}$ is much more restrictive
than the usual definition in \cite{NTV}, in that it requires actual
`touching' of the boundary of $Q$ or $3Q$ to that of the $\mathbf{r}$%
-parent. With $\Omega $ and $\boldsymbol{P}_{\Omega }$\ as in the definition
(\ref{def infinite dyadic grid}) of untruncated dyadic grids, it is well
known that the set of grids $\mathcal{D}^{\gamma ^{\prime }}\in \Omega $ for
which $Q\in \mathcal{D}^{\gamma ^{\prime }}$ and $Q$ is $\mathbf{r}$-bad in $%
\mathcal{D}^{\gamma ^{\prime }}$ has conditional probability at most a
multiple of $2^{-\mathbf{r}}$, i.e.%
\begin{equation}
\boldsymbol{P}_{\Omega }\left\{ \mathcal{D}^{\gamma ^{\prime }}\in \Omega :Q%
\text{ is }\mathbf{r}-\limfunc{bad}\text{ in }\mathcal{D}^{\gamma ^{\prime }}%
\text{ conditioned on }Q\in \mathcal{D}^{\gamma ^{\prime }}\right\} \lesssim
2^{-\mathbf{r}}.  \label{prob est}
\end{equation}%
Indeed, this follows for example from the construction of $\Omega _{M}^{N}$
in Subsubsection \ref{Subsub parameter} upon noticing that, given a cube $%
Q\in \mathcal{S}_{M}^{N}$ with $2^{-N}\leq \ell \left( Q\right) <2^{M-%
\mathbf{r}}$, only $\left( 2^{\mathbf{r}}\right) ^{n}-\left( 2^{\mathbf{r}%
}-4\right) ^{n}\approx \left( 2^{\mathbf{r}}\right) ^{n-1}$ of the $2^{n%
\mathbf{r}}$ possible level $\mathbf{r}$ parents of the cube $Q$ have
boundary that intersects\ that of $3Q$. This shows that the proportion of
such $\mathbf{r}-\limfunc{bad}$ cubes is $\frac{\left( 2^{\mathbf{r}}\right)
^{n}-\left( 2^{\mathbf{r}}-4\right) ^{n}}{2^{n\mathbf{r}}}\approx \frac{%
\left( 2^{\mathbf{r}}\right) ^{n-1}}{2^{n\mathbf{r}}}=2^{-\mathbf{r}}$,
which yields (\ref{prob est}) after invoking the identities (\ref{slice
decomp}) and (\ref{slice exp}).

Now we observe that for $\left\vert E_{j}^{k}\right\vert _{\omega }\neq 0$,
the quantity%
\begin{equation*}
\frac{|E_{j}^{k}|_{\omega }}{|3Q_{j}^{k}|_{\omega }^{2}}\left[ \sum_{i\in 
\mathbb{N}:\ P(Q_{i}^{k+m+m_{0}})=P(Q_{j}^{k})}\hspace{-0.4cm}%
A_{i}^{k+m+m_{0}}\int_{Q_{i}^{k+m+m_{0}}}\mathcal{L}_{j}^{k}\left( \mathbf{1}%
_{E_{j}^{k}}\omega \right) \sigma \right] ^{2}=q\left( Q\right)
\end{equation*}%
depends only on the cube $Q=Q_{j}^{k}$ and not on the underlying grid $%
\mathcal{D}^{\gamma }$, since the operator $\mathcal{L}_{j}^{k}$ depends
only on the dyadic grid structure \emph{within} the cube $Q_{j}^{k}$. Before
further decomposing the last sum $III^{\ast }$ in (\ref{last sum}) above
into pieces $IV+V$, we will use probability to control the sum over $\mathbf{%
r}-\limfunc{bad}$ cubes in (\ref{last sum}),%
\begin{equation*}
III_{\mathbf{r}-\limfunc{bad}}^{\ast }\equiv \sum_{\substack{ k\in \mathbb{Z}%
,k\geq N  \\ j\in \mathbb{N}\text{ and }Q_{j}^{k}\in \mathcal{D}_{\mathbf{r}-%
\limfunc{bad}}^{\gamma }}}III^{\ast }\left( Q_{j}^{k}\right) =\sum_{Q\in 
\mathcal{W}^{\gamma }\cap \mathcal{D}_{\mathbf{r}-\limfunc{bad}}^{\gamma
}}III^{\ast }\left( Q\right) .
\end{equation*}%
At this point our grids $\mathcal{D}^{\gamma }$ are not truncated, and we
are not yet working with the finite collection of truncated grids $\Omega
_{M}^{N}$. When convenient, we also write $\mathcal{W}^{\mathcal{D}}$
instead of $\mathcal{W}^{\gamma }$ when $\mathcal{D}=\mathcal{D}^{\gamma }$,
so that if $\mathcal{D}$ is the underlying grid in the definition of $III_{%
\mathbf{r}-\limfunc{bad}}^{\ast }$, then 
\begin{equation*}
III_{\mathbf{r}-\limfunc{bad}}^{\ast }=\sum_{Q\in \mathcal{W}^{\mathcal{D}%
}\cap \mathcal{D}_{\mathbf{r}-\limfunc{bad}}}III^{\ast }\left( Q\right) .
\end{equation*}%
A key point in what follows - already noted above - is that the quantity%
\begin{equation*}
q\left( Q\right) \equiv III^{\ast }\left( Q_{j}^{k}\right) ,\ \ \ \ \ \text{%
if }Q=Q_{j}^{k}\in \mathcal{W}^{\gamma }\text{ for some }\left( k,j\right)
\in W_{\limfunc{dis}}\ ,
\end{equation*}%
which is defined for all $Q\in \widehat{\mathcal{W}}\equiv
\dbigcup\limits_{\gamma \in \Omega }\mathcal{W}^{\gamma }$, depends only on
the cube $Q$ and not on any of the untruncated grids $\mathcal{D}^{\gamma }$
for which $Q=Q_{j}^{k}\in \mathcal{W}^{\gamma }$, so that we have%
\begin{equation*}
q:\widehat{\mathcal{W}}\rightarrow \left[ 0,\infty \right) .
\end{equation*}

We would of course like to restrict matters to cubes with side length
between $2^{-M}$ and $2^{-N}$ and use the conditional probability estimate (%
\ref{in part}) by simply extending the definition of our function $q:%
\widehat{\mathcal{W}}\rightarrow \left[ 0,\infty \right) $ to all of $\left( 
\mathcal{P}^{n}\right) _{M}^{N}$ by setting $q\left( Q\right) =0$ if $Q\in
\left( \mathcal{P}^{n}\right) _{M}^{N}\setminus \widehat{\mathcal{W}}$.
However, a subtle point arises here that prevents such a simple application
of (\ref{in part}). If $Q\in \mathcal{W}^{\gamma _{1}}\cap \mathcal{D}%
^{\gamma _{2}}$ for some $\gamma _{1},\gamma _{2}\in \Omega $, it need 
\textbf{not} be the case that $Q\in \mathcal{W}^{\gamma _{2}}$. However, if $%
Q\not\in \mathcal{W}^{\gamma _{2}}$, then the $\gamma _{2}$-parent $\pi _{%
\mathcal{D}^{\gamma _{2}}}Q$ of $Q$ is in $\mathcal{W}^{\gamma _{2}}$, and
this will prove to be a suitable substitute. We state and prove this in the
following lemma.

\begin{lemma}
\label{diff grids}Suppose that $Q\in \mathcal{W}^{\gamma _{1}}\cap \mathcal{D%
}^{\gamma _{2}}$ for some $\gamma _{1},\gamma _{2}\in \Omega $. Then either $%
Q$ or $\pi _{\mathcal{D}^{\gamma _{2}}}Q$ belongs to $\mathcal{W}^{\gamma
_{2}}$.
\end{lemma}

\begin{proof}
For this proof we use the notation,%
\begin{equation*}
\mathfrak{C}^{\left( \ell \right) }\left( Q\right) \equiv \left\{ Q^{\prime
}:Q^{\prime }\text{ is a dyadic subcube of }Q\text{ with }\ell \left(
Q^{\prime }\right) =2^{-\ell }\ell \left( Q\right) \right\} ,
\end{equation*}%
and refer to a cube $Q^{\prime }\in \mathfrak{C}^{\left( \ell \right)
}\left( Q\right) $ as a level $\ell $ dyadic child of $Q$. Now pick a point $%
x\in Q$. Since $Q\in \mathcal{W}^{\gamma _{1}}$, there is an integer $k$
such that $Q=Q_{j}^{k}$ for some distinguished index $\left( k,j\right) \in
W_{\limfunc{dis}}$. Thus $x\in \Omega _{k}$, there is a unique cube $P\in 
\mathcal{W}^{\gamma _{2}}$ such that $P=P_{j^{\prime }}^{k}$ for some index $%
\left( k,j^{\prime }\right) $ (not necessarily distinguished in the grid $%
\mathcal{W}^{\gamma _{2}}$) and such that $P$ contains $x$. Clearly, $P$
cannot be a dyadic child of $Q$ at any level since then $P$ would be a
strict subcube of $Q$ and hence not maximal in $\mathcal{D}^{\gamma _{2}}$
with respect to the property that $R_{W}P\subset \Omega _{k}$. We now claim
that $P\neq \pi _{\mathcal{D}^{\gamma _{2}}}^{\left( \ell \right) }Q$ for
any $\ell \geq 2$. Indeed, if $P=\pi _{\mathcal{D}^{\gamma _{2}}}^{\left(
\ell \right) }Q$ for some $\ell \geq 2$, then $R_{W}P\subset \Omega _{k}$,
and we now claim that $R_{W}\pi _{\mathcal{D}^{\gamma _{1}}}Q\subset \Omega
_{k}$ as well. For this, consider the metric $d_{\infty }\left( x,y\right)
\equiv \max_{1\leq j\leq n}\left\vert x_{j}-y_{j}\right\vert $ in $\mathbb{R}%
^{n}$, so that the ball $B_{d_{\infty }}\left( x,r\right) $ is the open cube
centered $x$ with side length $2r$. Then if $c_{I}$ denotes the center of
the cube$\,I$ and $z\in R_{W}\pi _{\mathcal{D}^{\gamma _{1}}}Q$, we have for 
$R_{W}>\frac{3}{2}$ that 
\begin{eqnarray*}
d_{\infty }\left( z,c_{P}\right) &\leq &d_{\infty }\left( z,c_{\pi _{%
\mathcal{D}^{\gamma _{1}}}Q}\right) +d_{\infty }\left( c_{\pi _{\mathcal{D}%
^{\gamma _{1}}}Q},c_{P}\right) \\
&\leq &R_{W}\frac{\ell \left( \pi _{\mathcal{D}^{\gamma _{1}}}Q\right) }{2}%
+\left( 2^{\ell }-1\right) \ell \left( Q\right) \\
&=&\left( R_{W}+2^{\ell }-1\right) \ell \left( Q\right) \\
&<&R_{W}2^{\ell -1}\ell \left( Q\right) =R_{W}\frac{\ell \left( \pi _{%
\mathcal{D}^{\gamma _{2}}}^{\left( \ell \right) }Q\right) }{2},
\end{eqnarray*}%
which shows that $z\in R_{W}\pi _{\mathcal{D}^{\gamma _{2}}}^{\left( \ell
\right) }Q$. Here we have used that $R_{W}+2^{\ell }-1<R_{W}2^{\ell -1}$ if
and only if $2\frac{2^{\ell }-1}{2^{\ell }-2}<R_{W}$, which holds for $R_{W}>%
\frac{3}{2}$ and $\ell \geq 2$. Thus we have 
\begin{equation*}
R_{W}\pi _{\mathcal{D}^{\gamma _{1}}}Q\subset R_{W}\pi _{\mathcal{D}^{\gamma
_{2}}}^{\left( \ell \right) }Q\subset \Omega _{k},
\end{equation*}%
which contradicts the assumption that $Q$ is a maximal cube in $\mathcal{D}%
^{\gamma _{1}}$ with $R_{W}Q\subset \Omega _{k}$.
\end{proof}

We now set up some definitions to deal with the subtle point discussed
above. The quantity $q\left( Q_{j}^{k}\right) $ has the following upper
bound where $\mathcal{D}_{K}\equiv \left\{ Q\in \mathcal{D}:Q\subset
K\right\} $ is the grid of dyadic subcubes of $K$: 
\begin{eqnarray*}
q\left( Q_{j}^{k}\right) &=&\frac{|E_{j}^{k}|_{\omega }}{|3Q_{j}^{k}|_{%
\omega }^{2}}\left[ \sum_{i\in \mathbb{N}:\ Q_{i}^{k+m+m_{0}}\subset
Q_{j}^{k}}\hspace{-0.4cm}A_{i}^{k+m+m_{0}}\int_{Q_{i}^{k+m+m_{0}}}\mathcal{L}%
_{j}^{k}\left( \mathbf{1}_{E_{j}^{k}}\omega \right) \sigma \right] ^{2} \\
&=&\frac{|E_{j}^{k}|_{\omega }}{|3Q_{j}^{k}|_{\omega }^{2}}\left[ \sum_{i\in 
\mathbb{N}:\ Q_{i}^{k+m+m_{0}}\subset Q_{j}^{k}}\hspace{-0.4cm}%
A_{i}^{k+m+m_{0}}\int_{E_{j}^{k}}\mathcal{L}_{j}^{k}\left( \mathbf{1}%
_{Q_{i}^{k+m+m_{0}}}\sigma \right) \omega \right] ^{2} \\
&=&\frac{|E_{j}^{k}|_{\omega }}{|3Q_{j}^{k}|_{\omega }^{2}}\left[
\int_{E_{j}^{k}}\mathcal{L}_{j}^{k}\left( \sum_{i\in \mathbb{N}:\
Q_{i}^{k+m+m_{0}}\subset Q_{j}^{k}}\hspace{-0.4cm}A_{i}^{k+m+m_{0}}\mathbf{1}%
_{Q_{i}^{k+m+m_{0}}}\sigma \right) \omega \right] ^{2} \\
&\leq &\frac{|E_{j}^{k}|_{\omega }}{|3Q_{j}^{k}|_{\omega }^{2}}\left[
\int_{E_{j}^{k}}\mathcal{M}^{\mathcal{D}_{Q_{j}^{k}}}\left( \sum_{i\in 
\mathbb{N}:\ Q_{i}^{k+m+m_{0}}\subset Q_{j}^{k}}\hspace{-0.4cm}%
A_{i}^{k+m+m_{0}}\mathbf{1}_{Q_{i}^{k+m+m_{0}}}\sigma \right) \omega \right]
^{2} \\
&\leq &\frac{|E_{j}^{k}|_{\omega }}{|3Q_{j}^{k}|_{\omega }^{2}}\left[
\int_{E_{j}^{k}}\mathcal{M}^{\mathcal{D}_{Q_{j}^{k}}}\left( f\sigma \right)
\omega \right] ^{2}\leq \left\vert E_{j}^{k}\right\vert _{\omega }\left[ 
\frac{1}{\left\vert 3Q_{j}^{k}\right\vert _{\omega }}\int_{Q_{j}^{k}}%
\mathcal{M}\left( f\sigma \right) \omega \right] ^{2},
\end{eqnarray*}%
since for any cube $P\in \mathcal{D}_{Q_{j}^{k}}$, we have%
\begin{eqnarray*}
&&\frac{1}{\left\vert P\right\vert }\int_{P}\left( \sum_{i\in \mathbb{N}:\
Q_{i}^{k+m+m_{0}}\subset Q_{j}^{k}}\hspace{-0.4cm}A_{i}^{k+m+m_{0}}\mathbf{1}%
_{Q_{i}^{k+m+m_{0}}}\right) \sigma \\
&=&\frac{1}{\left\vert P\right\vert }\int_{P}\left( \sum_{i\in \mathbb{N}:\
Q_{i}^{k+m+m_{0}}\subset Q_{j}^{k}}\left( \frac{1}{\left\vert
Q_{i}^{k+m+m_{0}}\right\vert _{\sigma }}\int_{Q_{i}^{k+m+m_{0}}}\left\vert
f\right\vert d\sigma \right) \mathbf{1}_{Q_{i}^{k+m+m_{0}}}\right) \sigma \\
&=&\frac{1}{\left\vert P\right\vert }\sum_{i\in \mathbb{N}:\
Q_{i}^{k+m+m_{0}}\subset P}\int_{Q_{i}^{k+m+m_{0}}}\left\vert f\right\vert
d\sigma \leq \frac{1}{\left\vert P\right\vert }\int_{P}\left\vert
f\right\vert d\sigma .
\end{eqnarray*}

Suppose that $Q\in \widehat{\mathcal{W}}$. If $Q\in \mathcal{W}^{\mathcal{D}%
} $ for some grid $\mathcal{D}$, then $Q=Q_{j}^{k}$ for some distinguished
index $\left( k,j\right) \in \mathcal{W}^{\mathcal{D}}$. If we also have $%
Q\in \mathcal{W}^{\mathcal{D}^{\prime }}$ for some grid $\mathcal{D}^{\prime
}$, then $Q=Q_{j^{\prime }}^{k^{\prime }}$ for some distinguished index $%
\left( k^{\prime },j^{\prime }\right) \in \mathcal{W}^{\mathcal{D}^{\prime
}} $. It is easy to see that $k^{\prime }\geq k$, and then by symmetry that $%
k=k^{\prime }$. Thus there is a unique integer $\kappa \left( Q\right)
=k=k^{\prime }$ associated with $Q\in \widehat{\mathcal{W}}$ that we refer
to as the \emph{height} of $Q$. We now define%
\begin{eqnarray*}
q^{\ast }\left( Q\right) &\equiv &\left\vert E_{Q}\right\vert _{\omega } 
\left[ \frac{1}{\left\vert 3Q\right\vert _{\omega }}\int_{Q}\mathcal{M}%
\left( f\sigma \right) \omega \right] ^{2}; \\
E_{Q} &\equiv &Q\setminus \Omega _{k+m+m_{0}}\text{ where }k=\kappa \left(
Q\right) ,
\end{eqnarray*}%
for $Q\in \widehat{\mathcal{W}}$, so that we have%
\begin{equation*}
q\left( Q\right) \leq q^{\ast }\left( Q\right) ,\ \ \ \ \ \text{for all }%
Q\in \widehat{\mathcal{W}},
\end{equation*}%
and finally we define%
\begin{eqnarray}
q^{\ast \ast }\left( Q\right) &\equiv &q^{\ast }\left( Q\right)
+\sum_{Q^{\prime }\in \mathfrak{C}\left( Q\right) \cap \widehat{\mathcal{W}}%
}q^{\ast }\left( Q^{\prime }\right) ,\ \ \ \ \ \text{for all }Q\in \widehat{%
\mathcal{W}},  \label{def triple star} \\
q^{\ast \ast \ast }\left( Q\right) &\equiv &q^{\ast \ast }\left( Q\right)
+\sum_{Q^{\prime }\in \mathfrak{C}\left( Q\right) \cap \widehat{\mathcal{W}}%
}q^{\ast \ast }\left( Q^{\prime }\right) ,\ \ \ \ \ \text{for all }Q\in 
\widehat{\mathcal{W}},  \notag
\end{eqnarray}%
where if $\mathfrak{C}\left( Q\right) \cap \widehat{\mathcal{W}}=\emptyset $
in either line, the corresponding sum vanishes.

\subsubsection{Truncated grids}

Recall that we have already fixed a grid $\mathcal{D}\in \Omega $, and then
by (\ref{slice decomp}), the truncated grid $\mathcal{D}_{M}^{N}$ has the
form $\mathcal{D}_{\limfunc{fin}}+s$ for some $s\in \left[ 0,2^{-M}\right)
^{n}$ and some $\mathcal{D}_{\limfunc{fin}}\in \Omega _{M}^{N}$. We now also
restrict the cubes $Q=Q_{j}^{k}$ in our sums to belong to $\mathcal{P}%
_{M}^{N}$, i.e. to satisfy%
\begin{equation}
2^{-M}\leq \ell \left( Q\right) \leq 2^{-N}.  \label{side length}
\end{equation}%
Later, at the end of the proof, we will take the supremum of the estimates
obtained over all $N<0<M$.

Let $\mathfrak{P}\left( Q\right) $ denote the collection of $2^{n}$ dyadic
parents of the cube $Q$. For a given cube $I\in \widehat{\mathcal{W}}$
satisfying $2^{-M}\leq \ell \left( I\right) \leq 2^{-1-N}$, and a parent $%
J\in \mathfrak{P}\left( I\right) $, we decompose the grids $\mathcal{D}_{%
\limfunc{fin}}\in \left( \Omega _{M}^{N}\right) _{I}$ according to $J=\pi _{%
\mathcal{D}_{\limfunc{fin}}}I$ and also to whether or not $I\in \mathcal{W}^{%
\mathcal{D}_{\limfunc{fin}}}$:%
\begin{eqnarray*}
\left( \Omega _{M}^{N}\right) _{I} &=&\dbigcup\limits_{J\in \mathfrak{P}%
\left( I\right) }\left[ \left\{ \mathcal{D}_{\limfunc{fin}}\in \left( \Omega
_{M}^{N}\right) _{J}:I\in \mathcal{W}^{\mathcal{D}_{\limfunc{fin}}}\right\}
\cup \left\{ \mathcal{D}_{\limfunc{fin}}\in \left( \Omega _{M}^{N}\right)
_{J}:I\not\in \mathcal{W}^{\mathcal{D}_{\limfunc{fin}}}\right\} \right] \\
&\equiv &\dbigcup\limits_{J\in \mathfrak{P}\left( I\right) }\left\{ \left(
\Omega _{M}^{N}\right) _{J}^{I\in \mathcal{W}^{\mathcal{D}_{\limfunc{fin}%
}}}\cup \left( \Omega _{M}^{N}\right) _{J}^{I\not\in \mathcal{W}^{\mathcal{D}%
_{\limfunc{fin}}}}\right\} \ .
\end{eqnarray*}%
Thus the collection $\left( \Omega _{M}^{N}\right) _{J}^{I\in \mathcal{W}^{%
\mathcal{D}_{\limfunc{fin}}}}$ consists of all grids $\mathcal{D}_{\limfunc{%
fin}}$ in $\left( \Omega _{M}^{N}\right) _{J}$ with the property that $I\in 
\mathcal{W}^{\mathcal{D}_{\limfunc{fin}}}$, while $\left( \Omega
_{M}^{N}\right) _{J}^{I\not\in \mathcal{W}^{\mathcal{D}_{\limfunc{fin}}}}$
consists of all grids $\mathcal{D}_{\limfunc{fin}}$ in $\left( \Omega
_{M}^{N}\right) _{J}$ with the property that $I\not\in \mathcal{W}^{\mathcal{%
D}_{\limfunc{fin}}}$. However, by Lemma \ref{diff grids} above, it will be
the case that $J\in \mathcal{W}^{\mathcal{D}_{\limfunc{fin}}}$ if $I\not\in 
\mathcal{W}^{\mathcal{D}_{\limfunc{fin}}}$. Thus for every grid $\mathcal{D}%
_{\limfunc{fin}}$ in $\left( \Omega _{M}^{N}\right) _{J}$ and every dyadic
child $I$ of $J$ which satisfies $I\in \widehat{\mathcal{W}}$, it is the
case that either $J$ or its child $I$ belongs to $\mathcal{W}^{\mathcal{D}_{%
\limfunc{fin}}}$. As we will see below, this is the reason we defined the
quantities $q^{\ast \ast }\left( Q\right) $ and $q^{\ast \ast \ast }\left(
Q\right) $ above. Finally we note that%
\begin{equation*}
\#\left( \Omega _{M}^{N}\right) _{I}=\sum_{J\in \mathfrak{P}\left( I\right)
}\left( \#\left( \Omega _{M}^{N}\right) _{J}^{I\in \mathcal{W}^{\mathcal{D}_{%
\limfunc{fin}}}}+\#\left( \Omega _{M}^{N}\right) _{J}^{I\not\in \mathcal{W}^{%
\mathcal{D}_{\limfunc{fin}}}}\right) .
\end{equation*}

Now let $I_{\mathbf{r}-\limfunc{bad}}$ denote the collection of grids $%
\mathcal{D}_{\limfunc{fin}}\in \left( \Omega _{M}^{N}\right) _{I}$ such that 
$I$ is $\mathbf{r}-\limfunc{bad}$ in $\mathcal{D}_{\limfunc{fin}}$,%
\begin{equation*}
\text{i.e. }I_{\mathbf{r}-\limfunc{bad}}\equiv \left\{ \mathcal{D}_{\limfunc{%
fin}}\in \left( \Omega _{M}^{N}\right) _{I}:I\text{ is }\mathbf{r}-\limfunc{%
bad}\text{ in }\mathcal{D}_{\limfunc{fin}}\right\} ,
\end{equation*}%
so that 
\begin{equation*}
I_{\mathbf{r}-\limfunc{bad}}\subset \left( \Omega _{M}^{N}\right) _{I}\text{
and }\#I_{\mathbf{r}-\limfunc{bad}}\leq C2^{-\mathbf{r}}\ \#\left( \Omega
_{M}^{N}\right) _{I}\ ,\ \ \ \ \ 2^{-M}\leq \ell \left( I\right) <2^{-%
\mathbf{r}-N}.
\end{equation*}%
We restrict the side length of $I$ to $\ell \left( I\right) <2^{-\mathbf{r}%
-N}$ in order that the level $\mathbf{r}$ parent $\pi ^{\left( \mathbf{r}%
\right) }I$ of $I$ in the grid $\mathcal{D}_{\limfunc{fin}}$ belongs to $%
\left( \Omega _{M}^{N}\right) _{I}$. Recall also that one should think of $M$
near $\infty $ and $N$ near $-\infty $. Define the collection $\widehat{%
\left( \Omega _{M}^{N}\right) }_{I}\equiv \left\{ \mathcal{D}_{\limfunc{fin}%
}\in \left( \Omega _{M}^{N}\right) _{I}:I\in \mathcal{W}^{\mathcal{D}_{%
\limfunc{fin}}}\right\} $ to consist of those grids $\mathcal{D}_{\limfunc{%
fin}}$ that not only contain $I$, but satisfy $I\in \mathcal{W}^{\mathcal{D}%
_{\limfunc{fin}}}$. Fix $I\in \widehat{\mathcal{W}}$ and let $\mathcal{D}_{%
\limfunc{fin}}\in \left( \Omega _{M}^{N}\right) _{I}$. Then $\mathcal{D}_{%
\limfunc{fin}}\in \left( \Omega _{M}^{N}\right) _{J}$ for some $J\in 
\mathfrak{P}\left( I\right) $,\ and Lemma \ref{diff grids} implies that
either $\mathcal{D}_{\limfunc{fin}}\in \widehat{\left( \Omega
_{M}^{N}\right) }_{I}$ or $\mathcal{D}_{\limfunc{fin}}\in \widehat{\left(
\Omega _{M}^{N}\right) }_{J}$. Thus we have $\left( \Omega _{M}^{N}\right)
_{I}=\widehat{\left( \Omega _{M}^{N}\right) }_{I}\cup \dbigcup\limits_{J\in 
\mathfrak{P}\left( I\right) }\widehat{\left( \Omega _{M}^{N}\right) }_{J}$
and so%
\begin{equation*}
I_{\mathbf{r}-\limfunc{bad}}\subset \left[ \widehat{\left( \Omega
_{M}^{N}\right) }_{I}\cap I_{\mathbf{r}-\limfunc{bad}}\right] \cup
\dbigcup\limits_{J\in \mathfrak{P}\left( I\right) }\left[ \widehat{\left(
\Omega _{M}^{N}\right) }_{J}\cap J_{\left( \mathbf{r}-1\right) -\limfunc{bad}%
}\right] \ ,
\end{equation*}%
since $\mathcal{D}_{\limfunc{fin}}\in \widehat{\left( \Omega _{M}^{N}\right) 
}_{J}\cap I_{\mathbf{r}-\limfunc{bad}}$ for a dyadic child $I\in \mathfrak{C}%
\left( J\right) $ implies that $J$ itself is $\left( \mathbf{r}-1\right) -%
\limfunc{bad}$ in $\mathcal{D}_{\limfunc{fin}}$.

We have%
\begin{eqnarray*}
\#I_{\mathbf{r}-\limfunc{bad}} &\leq &\#\left[ \widehat{\left( \Omega
_{M}^{N}\right) }_{I}\cap I_{\mathbf{r}-\limfunc{bad}}\right] +\sum_{J\in 
\mathfrak{P}\left( I\right) }\#\left[ \widehat{\left( \Omega _{M}^{N}\right) 
}_{J}\cap J_{\left( \mathbf{r}-1\right) -\limfunc{bad}}\right] \\
&\leq &\sum_{J\in \mathfrak{P}\left( I\right) }\left( \#\left[ \left( \Omega
_{M}^{N}\right) _{J}\cap \widehat{\left( \Omega _{M}^{N}\right) }_{I}\cap I_{%
\mathbf{r}-\limfunc{bad}}\right] +\#\left[ \widehat{\left( \Omega
_{M}^{N}\right) }_{J}\cap J_{\left( \mathbf{r}-1\right) -\limfunc{bad}}%
\right] \right) \ ,
\end{eqnarray*}%
since $\left( \Omega _{M}^{N}\right) _{I}$ is the pairwise disjoint union of
the sets $\left\{ \left( \Omega _{M}^{N}\right) _{J}\right\} _{J\in 
\mathfrak{P}\left( I\right) }$. We now deviate slightly from the treatment
of conditional probability in (\ref{in part}) above by setting

\begin{eqnarray*}
\Theta _{M}^{N} &\equiv &\left\{ \left( I,\mathcal{D}_{\limfunc{fin}}\right)
\in \mathcal{P}_{M}^{N}\times \Omega _{M}^{N}:I\in \mathcal{W}^{\mathcal{D}_{%
\limfunc{fin}}}\right\} , \\
B_{M}^{N} &\equiv &\left\{ \left( I,\mathcal{D}_{\limfunc{fin}}\right) \in
\Theta _{M}^{N}:I\text{ is }\mathbf{r}-\limfunc{bad}\text{ in }\mathcal{D}_{%
\limfunc{fin}}\right\} .
\end{eqnarray*}%
By (\ref{def triple star}) we have $q\left( I\right) \leq q^{\ast }\left(
I\right) \leq q^{\ast \ast }\left( I\right) $ for all cubes $I\in \widehat{%
\mathcal{W}}$. We now denote by $III_{\mathbf{r}-\limfunc{bad}}^{\ast
}\left( M,N+\mathbf{r}+1\right) $ the term $III_{\mathbf{r}-\limfunc{bad}%
}^{\ast }$ but with cubes $Q$ restricted to satisfying%
\begin{equation}
2^{-M}\leq \ell \left( Q\right) \leq 2^{-\left( N+\mathbf{r}+1\right) }.
\label{side length restriction}
\end{equation}

Now fix $s\in \left[ 0,2^{-M}\right) ^{n}$. We have%
\begin{eqnarray}
&&\boldsymbol{E}_{\Omega _{M}^{N}+s}^{\mathcal{D}_{\limfunc{fin}}}\left(
\sum_{I\in \mathcal{W}^{\mathcal{D}_{\limfunc{fin}}}\cap \mathcal{P}_{M}^{N+%
\mathbf{r}+1}}q^{\ast }\left( I\right) \mathbf{1}_{B_{M}^{N}}\left( I,%
\mathcal{D}_{\limfunc{fin}}\right) \right)  \label{comp} \\
&=&\frac{1}{\#\Omega _{M}^{N}}\sum_{\mathcal{D}\in \Omega }\sum_{I\in 
\mathcal{W}^{\mathcal{D}}\cap \mathcal{P}_{M}^{N+\mathbf{r}+1}:\ \left( I,%
\mathcal{D}\right) \in B_{M}^{N}}q^{\ast }\left( I\right) \leq \frac{1}{%
\#\Omega _{M}^{N}}\sum_{I\in \widehat{\mathcal{W}}\cap \mathcal{P}_{M}^{N+%
\mathbf{r}+1}}q^{\ast }\left( I\right) \ \#I_{\mathbf{r}-\limfunc{bad}} 
\notag \\
&=&\frac{1}{\#\Omega _{M}^{N}}\sum_{I\in \widehat{\mathcal{W}}\cap \mathcal{P%
}_{M}^{N+\mathbf{r}+1}}q^{\ast }\left( I\right) \left( \#\left[ \widehat{%
\left( \Omega _{M}^{N}\right) }_{I}\cap I_{\mathbf{r}-\limfunc{bad}}\right]
+\sum_{J\in \mathfrak{P}\left( I\right) \cap \widehat{\mathcal{W}}}\#\left[ 
\widehat{\left( \Omega _{M}^{N}\right) }_{J}\cap J_{\left( \mathbf{r}%
-1\right) -\limfunc{bad}}\right] \right)  \notag \\
&\leq &\frac{1}{\#\Omega _{M}^{N}}\sum_{I\in \widehat{\mathcal{W}}\cap 
\mathcal{P}_{M}^{N+\mathbf{r}+1}}q^{\ast }\left( I\right) \#\left[ \widehat{%
\left( \Omega _{M}^{N}\right) }_{I}\cap I_{\mathbf{r}-\limfunc{bad}}\right] +%
\frac{2^{n}}{\#\Omega }\sum_{J\in \widehat{\mathcal{W}}\cap \mathcal{P}%
_{M-1}^{N+\mathbf{r}}}q^{\ast \ast }\left( J\right) \#\left[ \widehat{\left(
\Omega _{M}^{N}\right) }_{J}\cap J_{\left( \mathbf{r}-1\right) -\limfunc{bad}%
}\right]  \notag \\
&\leq &C\frac{1}{\#\Omega _{M}^{N}}\sum_{J\in \widehat{\mathcal{W}}\cap 
\mathcal{P}_{M-1}^{N+\mathbf{r}}}q^{\ast \ast }\left( J\right) \#\left[
\left( \Omega _{M}^{N}\right) _{J}\cap J_{\left( \mathbf{r}-1\right) -%
\limfunc{bad}}\right] \leq C2^{-\mathbf{r}}\frac{1}{\#\Omega _{M}^{N}}%
\sum_{J\in \widehat{\mathcal{W}}\cap \mathcal{P}_{M-1}^{N+\mathbf{r}%
}}q^{\ast \ast }\left( J\right) \#\left( \Omega _{M}^{N}\right) _{J}\ , 
\notag
\end{eqnarray}%
by the probability estimate (\ref{prob est}) with $\mathbf{r}-1$ in place of 
$\mathbf{r}$. Now for $J\in \widehat{\mathcal{W}}$ we have $\left( \Omega
_{M}^{N}\right) _{J}=\widehat{\left( \Omega _{M}^{N}\right) }_{J}\cup
\dbigcup\limits_{K\in \mathfrak{P}\left( J\right) }\widehat{\left( \Omega
_{M}^{N}\right) }_{K}$ by Lemma \ref{diff grids}, and so we can continue (%
\ref{comp}) with%
\begin{eqnarray*}
&&C2^{-\mathbf{r}}\frac{1}{\#\Omega _{M}^{N}}\sum_{J\in \widehat{\mathcal{W}}%
\cap \mathcal{P}_{M-1}^{N+\mathbf{r}}}q^{\ast \ast }\left( J\right) \#\left(
\Omega _{M}^{N}\right) _{J} \\
&\leq &C2^{-\mathbf{r}}\frac{1}{\#\Omega _{M}^{N}}\sum_{J\in \widehat{%
\mathcal{W}}\cap \mathcal{P}_{M-1}^{N+\mathbf{r}}}q^{\ast \ast }\left(
J\right) \#\left[ \widehat{\left( \Omega _{M}^{N}\right) }_{J}\cup
\dbigcup\limits_{K\in \mathfrak{P}\left( J\right) }\widehat{\left( \Omega
_{M}^{N}\right) }_{K}\right] \\
&\leq &C2^{-\mathbf{r}}\frac{1}{\#\Omega _{M}^{N}}\sum_{J\in \widehat{%
\mathcal{W}}\cap \mathcal{P}_{M-1}^{N+\mathbf{r}}}q^{\ast \ast }\left(
J\right) \#\widehat{\left( \Omega _{M}^{N}\right) }_{J}+C2^{-\mathbf{r}}%
\frac{1}{\#\Omega }\sum_{K\in \widehat{\mathcal{W}}\cap \mathcal{P}_{M-2}^{N+%
\mathbf{r}-1}}q^{\ast \ast \ast }\left( K\right) \#\widehat{\left( \Omega
_{M}^{N}\right) }_{K} \\
&\leq &C2^{-\mathbf{r}}\frac{1}{\#\Omega _{M}^{N}}\sum_{K\in \widehat{%
\mathcal{W}}\cap \mathcal{P}_{M-2}^{N+\mathbf{r}-1}}q^{\ast \ast \ast
}\left( K\right) \#\widehat{\left( \Omega _{M}^{N}\right) }_{K} \\
&=&C2^{-\mathbf{r}}\frac{1}{\#\Omega _{M}^{N}}\sum_{\mathcal{D}_{\limfunc{fin%
}}\in \Omega _{M}^{N}+s}\sum_{K\in \mathcal{W}^{\mathcal{D}_{\limfunc{fin}%
}}\cap \mathcal{P}_{M-2}^{N+\mathbf{r}-1}}q^{\ast \ast \ast }\left( K\right)
\leq C2^{-\mathbf{r}}\boldsymbol{E}_{\Omega _{M}^{N}+s}^{\mathcal{D}_{%
\limfunc{fin}}}\left( \sum_{K\in \mathcal{W}^{\mathcal{D}_{\limfunc{fin}%
}}}q^{\ast \ast \ast }\left( K\right) \right) \ ,
\end{eqnarray*}%
where we have used (\ref{def triple star}) above once more. Now take an
average over $s\in \left[ 0,2^{-M}\right) ^{n}$ in the above inequality%
\begin{equation*}
\boldsymbol{E}_{\Omega _{M}^{N}+s}^{\mathcal{D}_{\limfunc{fin}}}\left(
\sum_{I\in \mathcal{W}^{\mathcal{D}_{\limfunc{fin}}}\cap \mathcal{P}_{M}^{N+%
\mathbf{r}+1}}q^{\ast }\left( I\right) \mathbf{1}_{B_{M}^{N}}\left( I,%
\mathcal{D}_{\limfunc{fin}}\right) \right) \leq C2^{-\mathbf{r}}\boldsymbol{E%
}_{\Omega _{M}^{N}+s}^{\mathcal{D}_{\limfunc{fin}}}\left( \sum_{K\in 
\mathcal{W}^{\mathcal{D}_{\limfunc{fin}}}}q^{\ast \ast \ast }\left( K\right)
\right) ,
\end{equation*}%
to obtain%
\begin{equation*}
\boldsymbol{E}_{\Phi _{M}^{N}}^{\mathcal{D}_{\limfunc{fin}}}\left( III_{%
\mathbf{r}-\limfunc{bad}}^{\ast }\left( M,N+\mathbf{r}+1\right) \right) \leq
C2^{-\mathbf{r}}\boldsymbol{E}_{\Phi _{M}^{N}}^{\mathcal{D}_{\limfunc{fin}%
}}\left( \sum_{K\in \mathcal{W}^{\mathcal{D}_{\limfunc{fin}}}\cap \mathcal{P}%
_{M}^{N}}q^{\ast \ast \ast }\left( K\right) \right) \ ,
\end{equation*}%
where $\Phi _{M}^{N}$ is defined in Subsubsection \ref{Subsub domination} as
a union of translates of the grids in $\Omega _{M}^{N}$, and we remind the
reader that the sum in $III_{\mathbf{r}-\limfunc{bad}}^{\ast }\left( M-2,N+%
\mathbf{r}\right) $ on the left hand side satisfies (\ref{side length
restriction}) while the sum on the right hand side satisfies (\ref{side
length}).

Now we estimate the sums $\sum_{K\in \mathcal{W}^{\mathcal{D}_{\limfunc{fin}%
}}}q^{\ast \ast \ast }\left( K\right) $ uniformly over grids $\mathcal{D}_{%
\limfunc{fin}}$ as follows. Fix a grid $\mathcal{D}_{\limfunc{fin}}$ for the
moment. If $\mathcal{M}_{\omega }^{\mathcal{D}_{\limfunc{fin}}}$ denotes the 
$\mathcal{D}_{\limfunc{fin}}$-dyadic maximal operator with respect to the
measure $\omega $, i.e. 
\begin{equation*}
\mathcal{M}_{\omega }^{\mathcal{D}_{\limfunc{fin}}}h\left( x\right) \equiv
\sup_{Q\in \mathcal{D}_{\limfunc{fin}}:\ x\in Q}\frac{1}{\left\vert
Q\right\vert _{\omega }}\int_{Q}\left\vert h\right\vert d\omega ,
\end{equation*}%
then for $Q_{j}^{k}\in \mathcal{W}^{\mathcal{D}_{\limfunc{fin}}}$, we have 
\begin{eqnarray*}
q^{\ast \ast \ast }\left( Q_{j}^{k}\right) &=&\left\vert
E_{j}^{k}\right\vert _{\omega }\left[ \frac{1}{\left\vert
3Q_{j}^{k}\right\vert _{\omega }}\int_{Q_{j}^{k}}\mathcal{M}\left( f\sigma
\right) \omega \right] ^{2}+\sum_{Q^{\prime }\in \mathfrak{C}\left(
Q_{j}^{k}\right) \cap \widehat{\mathcal{W}}}\left\vert E_{Q^{\prime
}}\right\vert _{\omega }\left[ \frac{1}{\left\vert 3Q^{\prime }\right\vert
_{\omega }}\int_{Q^{\prime }}\mathcal{M}\left( f\sigma \right) \omega \right]
^{2} \\
&&+\sum_{Q^{\prime \prime }\in \mathfrak{C}^{\left( 2\right) }\left(
Q_{j}^{k}\right) \cap \widehat{\mathcal{W}}}\left\vert E_{Q^{\prime \prime
}}\right\vert _{\omega }\left[ \frac{1}{\left\vert 3Q^{\prime \prime
}\right\vert _{\omega }}\int_{Q^{\prime \prime }}\mathcal{M}\left( f\sigma
\right) \omega \right] ^{2} \\
&\leq &\int_{E_{j}^{k}}\mathcal{M}_{\omega }^{\mathcal{D}_{\limfunc{fin}%
}}\left( \mathcal{M}\left( f\sigma \right) \right) ^{2}d\omega
+\sum_{Q^{\prime }\in \mathfrak{C}\left( Q_{j}^{k}\right) \cap \widehat{%
\mathcal{W}}}\int_{E\left( Q^{\prime }\right) }\mathcal{M}_{\omega }^{%
\mathcal{D}_{\limfunc{fin}}}\left( \mathcal{M}\left( f\sigma \right) \right)
^{2}d\omega \\
&&+\sum_{Q^{\prime \prime }\in \mathfrak{C}^{\left( 2\right) }\left(
Q_{j}^{k}\right) \cap \widehat{\mathcal{W}}}\int_{E\left( Q^{\prime \prime
}\right) }\mathcal{M}_{\omega }^{\mathcal{D}_{\limfunc{fin}}}\left( \mathcal{%
M}\left( f\sigma \right) \right) ^{2}d\omega \\
&\equiv &\mathcal{I}_{j}^{k}+\mathcal{II}_{j}^{k}+\mathcal{III}_{j}^{k}\ ,
\end{eqnarray*}%
where $E\left( Q^{\prime }\right) =Q^{\prime }\setminus \Omega _{\kappa
\left( Q^{\prime }\right) +m+m_{0}}$ and $E\left( Q^{\prime \prime }\right)
=Q^{\prime \prime }\setminus \Omega _{\kappa \left( Q^{\prime \prime
}\right) +m+m_{0}}$. Now we trivially have%
\begin{eqnarray*}
\sum_{k,j:\ Q_{j}^{k}\in \mathcal{W}^{\mathcal{D}_{\limfunc{fin}}}}\mathcal{I%
}_{j}^{k} &\leq &\sum_{k,j:\ Q_{j}^{k}\in \mathcal{W}^{\mathcal{D}_{\limfunc{%
fin}}}}\int_{E_{j}^{k}}\mathcal{M}_{\omega }^{\mathcal{D}_{\limfunc{fin}%
}}\left( \mathcal{M}\left( f\sigma \right) \right) ^{2}d\omega \\
&\leq &\int_{\mathbb{R}^{n}}\mathcal{M}_{\omega }^{\mathcal{D}_{\limfunc{fin}%
}}\left( \mathcal{M}\left( f\sigma \right) \right) ^{2}d\omega \leq C\int_{%
\mathbb{R}^{n}}\left( \mathcal{M}\left( f\sigma \right) \right) ^{2}d\omega ,
\end{eqnarray*}%
since the collection of sets $\left\{ E_{j}^{k}\right\} _{\left( k,j\right)
\in \mathcal{W}_{\limfunc{dis}}^{\mathcal{D}_{\limfunc{fin}}}}$ is pairwise
disjoint in $\mathbb{R}^{n}$.

To estimate the sum of the terms $\mathcal{II}_{j}^{k}$ we will require a
bounded overlap constant for the collection of sets $\left\{ E\left(
Q^{\prime }\right) :Q^{\prime }\in \mathfrak{C}\left( Q_{j}^{k}\right) \cap 
\widehat{\mathcal{W}}\right\} _{\left( k,j\right) \in \mathcal{W}_{\limfunc{%
dis}}^{\mathcal{D}_{\limfunc{fin}}}}$. Recall that the cubes $Q_{j}^{k}$ all
belong to the Whitney grid $\mathcal{W}^{\mathcal{D}_{\limfunc{fin}}}$. To
obtain such a bounded overlap constant, suppose that $\mathcal{T}\equiv
\left\{ Q_{\ell }\right\} _{\ell =1}^{L}$ is a strictly increasing \emph{%
consective} tower of cubes $Q_{\ell }\subsetneqq Q_{\ell +1}$ with $Q_{\ell
}\in \widehat{\mathcal{W}}$ (by \emph{consecutive} we mean that every cube $%
Q $ in $\widehat{\mathcal{W}}$ that satisfies $Q_{1}\subset Q\subset Q_{L}$
is included in the tower $\mathcal{T}$). By Lemma \ref{diff grids}, we see
that if $Q_{\ell }\notin \mathcal{W}^{\mathcal{D}_{\limfunc{fin}}}$, then $%
\pi Q_{\ell }\in \mathcal{W}^{\mathcal{D}_{\limfunc{fin}}}\subset \widehat{%
\mathcal{W}}$ which shows that $Q_{\ell +1}=\pi Q_{\ell }\in \mathcal{W}^{%
\mathcal{D}_{\limfunc{fin}}}$. Thus we see that at least half of the cubes
in the tower belong to $\mathcal{W}^{\mathcal{D}_{\limfunc{fin}}}$. Now
focus attention on the subtower $\mathcal{S}=\left\{ Q_{\ell _{i}}\right\}
_{i=1}^{I}\equiv \left\{ Q_{\ell _{i}}\in \mathcal{T}:Q_{\ell _{i}}\in 
\mathcal{W}^{\mathcal{D}_{\limfunc{fin}}}\right\} $ of cubes which belong to 
$\mathcal{W}^{\mathcal{D}_{\limfunc{fin}}}$. It thus suffices to establish a
bounded overlap constant for the subtower $\mathcal{S}$. However, there are
clearly at most $m+m_{0}$ cubes in the tower $\mathcal{S}$ since $E\left(
Q_{\ell _{i}}\right) =Q_{\ell _{i}}\setminus \Omega _{\kappa \left( Q_{\ell
_{i}}\right) +m+m_{0}}$ where $\kappa \left( Q_{\ell _{i}}\right) $ is
strictly decreasing in $i$ because all the cubes $Q_{\ell _{i}}$ belong to a
common grid, namely $\mathcal{D}_{\limfunc{fin}}$. Thus $\#\mathcal{S}=I\leq
m+m_{0}$ and $\#\mathcal{T}\leq 2m+2m_{0}$. It follows in particular that
the collection of sets $\left\{ E\left( Q^{\prime }\right) :Q^{\prime }\in 
\mathfrak{C}\left( Q_{j}^{k}\right) \cap \widehat{\mathcal{W}}\right\}
_{\left( k,j\right) \in \mathcal{W}_{\limfunc{dis}}^{\mathcal{D}_{\limfunc{%
fin}}}}$ has bounded overlap\ at most $2m+2m_{0}$, and we conclude that%
\begin{eqnarray*}
\sum_{k,j:\ Q_{j}^{k}\in \mathcal{W}^{\mathcal{D}_{\limfunc{fin}}}}\mathcal{%
II}_{j}^{k} &\leq &\sum_{k,j:\ Q_{j}^{k}\in \mathcal{W}^{\mathcal{D}_{%
\limfunc{fin}}}}\sum_{Q^{\prime }\in \mathfrak{C}\left( Q_{j}^{k}\right)
\cap \widehat{\mathcal{W}}}\int_{E\left( Q^{\prime }\right) }\mathcal{M}%
_{\omega }^{\mathcal{D}_{\limfunc{fin}}}\left( \mathcal{M}\left( f\sigma
\right) \right) ^{2}d\omega \\
&\leq &\left( 2m+2m_{0}\right) \int_{\mathbb{R}^{n}}\mathcal{M}_{\omega }^{%
\mathcal{D}_{\limfunc{fin}}}\left( \mathcal{M}\left( f\sigma \right) \right)
^{2}d\omega \leq C\int_{\mathbb{R}^{n}}\left( \mathcal{M}\left( f\sigma
\right) \right) ^{2}d\omega .
\end{eqnarray*}

The sum of the terms $\mathcal{III}_{j}^{k}$ satisfies a similar estimate.
Indeed, we have already shown above that the tower $\mathcal{T}\equiv
\left\{ Q_{\ell }\right\} _{\ell =1}^{L}$ satisfies $\#\mathcal{T}\leq
2m+2m_{0}$, and it follows in particular that $\left\{ E\left( Q^{\prime
\prime }\right) :Q^{\prime \prime }\in \mathfrak{C}^{\left( 2\right) }\left(
Q_{j}^{k}\right) \cap \widehat{\mathcal{W}}\right\} _{\left( k,j\right) \in 
\mathcal{W}_{\limfunc{dis}}^{\mathcal{D}_{\limfunc{fin}}}}$ also has bounded
overlap\ at most $2m+2m_{0}$. Altogether then we have%
\begin{eqnarray*}
\sum_{K\in \mathcal{W}^{\mathcal{D}_{\limfunc{fin}}}}q^{\ast \ast \ast
}\left( K\right) &=&\sum_{k,j:\ Q_{j}^{k}\in \mathcal{W}^{\mathcal{D}_{%
\limfunc{fin}}}}q^{\ast \ast \ast }\left( Q_{j}^{k}\right) \leq 3\sum_{k,j:\
Q_{j}^{k}\in \mathcal{W}^{\mathcal{D}_{\limfunc{fin}}}}\left( \mathcal{I}%
_{j}^{k}+\mathcal{II}_{j}^{k}+\mathcal{III}_{j}^{k}\right) \\
&\leq &C\int_{\mathbb{R}^{n}}\mathcal{M}_{\omega }^{\mathcal{D}_{\limfunc{fin%
}}}\left( \mathcal{M}\left( f\sigma \right) \right) ^{2}d\omega \leq C\int_{%
\mathbb{R}^{n}}\left( \mathcal{M}\left( f\sigma \right) \right) ^{2}d\omega ,
\end{eqnarray*}%
for all $\mathcal{D}_{\limfunc{fin}}\in \mathcal{G}_{M}^{N}$. Finally, we
average over $\mathcal{H}_{\mathcal{D}_{\limfunc{fin}}}$ for each $\mathcal{D%
}_{\limfunc{fin}}\in \Phi _{M}^{N}$ and use (\ref{slice exp}) to obtain 
\begin{eqnarray}
\boldsymbol{E}_{\Omega }^{\mathcal{D}}\left( III_{\mathbf{r}-\limfunc{bad}%
}^{\ast }\left( M,N+\mathbf{r}+1\right) \right) &\leq &C2^{-\mathbf{r}}%
\boldsymbol{E}_{\Phi _{M}^{N}}^{\mathcal{D}_{\limfunc{fin}}}\boldsymbol{E}_{%
\mathcal{H}_{\mathcal{D}_{\limfunc{fin}}}}^{\mathcal{D}}\left( \sum_{K\in 
\mathcal{W}^{\mathcal{D}}\cap \mathcal{P}_{M}^{N}}q^{\ast \ast \ast }\left(
K\right) \right)  \label{alto} \\
&\leq &C2^{-\mathbf{r}}\boldsymbol{E}_{\Omega }^{\mathcal{D}}\left( C\int_{%
\mathbb{R}^{n}}\left( \mathcal{M}\left( f\sigma \right) \right) ^{2}d\omega
\right) =C2^{-\mathbf{r}}\int \left\vert \mathcal{M}f\sigma \right\vert
^{2}d\omega ,  \notag
\end{eqnarray}%
where the sum over cubes in $III_{\mathbf{r}-\limfunc{bad}}^{\ast }\left(
M,N+\mathbf{r}+1\right) $ on the left hand side satisfies (\ref{side length
restriction}). This estimate will be applied at the end of the proof in
order to estimate $III_{\mathbf{r}-\limfunc{bad}}^{\ast }$ by taking a
supremum over cubes $Q$ satisfying (\ref{side length restriction}), i.e. $%
2^{-M}\leq \ell \left( Q\right) \leq 2^{-\left( N+\mathbf{r}+1\right) }$.

\subsection{Principal cube decomposition}

Recall our convention regarding distinguished index pairs $\left( k,j\right) 
$: namely that $k\equiv k_{0}\ \func{mod}\left( m+m_{0}\right) $ for some
fixed $0\leq k_{0}\leq m+m_{0}-1$, and that $k$ is maximal among equal cubes 
$Q_{j}^{k}$. Fix an integer $L\in \mathbb{Z}$ (thought of as near $-\infty $%
) such that $L\equiv k_{0}\ \func{mod}\left( m+m_{0}\right) $, and let $%
G_{0} $ consist of the $\mathcal{D}^{\gamma }$-maximal cubes in $\Omega _{L}$%
. With the grid $\mathcal{W}=\mathcal{W}^{\gamma }$ in hand, we now
introduce principal cubes as in \cite[page 804]{MuWh} (note that we are
suppressing the dependence of $\mathcal{W}$ on $\gamma $ for reduction of
notation). If $G_{n}$ has been defined, let $G_{n+1}$ consist \ of those
indices $\left( k,j\right) $ for which $Q_{j}^{k}\in \mathcal{W}$, there is
an index $\left( t,u\right) \in G_{n}$ with $k\geq t$ and $Q_{j}^{k}\subset
Q_{u}^{t}$, and

(\textbf{i}) $A_{j}^{k}>\eta A_{u}^{t}\ $,

(\textbf{ii}) $A_{i}^{\ell }\leq \eta A_{u}^{t}$ whenever $%
Q_{j}^{k}\subsetneqq Q_{i}^{\ell }\subset Q_{u}^{t}$ .

Here $\eta $ is any constant larger than $1$, for example $\eta =4$ works
fine. Now define $\Gamma \equiv \bigcup\limits_{n=0}^{\infty }G_{n}$ and for
each index $\left( k,j\right) $ define $P\left( Q_{j}^{k}\right) $ to be the
smallest dyadic cube $Q_{u}^{t}$ containing $Q_{j}^{k}$ and with $\left(
t,u\right) \in \Gamma $. Then we have%
\begin{eqnarray}
&&\ \text{(\textbf{i}) }P\left( Q_{j}^{k}\right) =Q_{u}^{t}\Longrightarrow
A_{j}^{k}\leq \eta A_{u}^{t}\ ,  \label{main have} \\
&&\ \text{(\textbf{ii}) }Q_{j}^{k}\subsetneqq Q_{u}^{t}\text{ with }\left(
k,j\right) ,\left( t,u\right) \in \Gamma \Longrightarrow A_{j}^{k}>\eta
A_{u}^{t}\ .  \notag
\end{eqnarray}

Now we return to the sum of $III^{\ast }\left( Q_{j}^{k}\right) $ over $%
\left( k,j\right) $ satisfying $k\geq L$, and decompose the sum over $i$
inside $III^{\ast }\left( Q_{j}^{k}\right) $ according to whether $%
(k+m+m_{0},i)\in \Gamma $ or the predecessor $P\left(
Q_{i}^{k+m+m_{0}}\right) $ of $Q_{i}^{k+m+m_{0}}$ in the grid $\Gamma $
coincides with the predecessor $P\left( Q_{j}^{k}\right) $ of $Q_{j}^{k}$:

\begin{align*}
\sum_{\substack{ k\in \mathbb{Z},k\geq L  \\ j\in \mathbb{N}}}III^{\ast
}\left( Q_{j}^{k}\right) & \lesssim \sum_{\substack{ k\in \mathbb{Z},k\geq L 
\\ j\in \mathbb{N}}}\frac{|E_{j}^{k}|_{\omega }}{|3Q_{j}^{k}|_{\omega }^{2}}%
\left[ \sum_{i\in \mathbb{N}:\ P\left( Q_{i}^{k+m+m_{0}}\right) =P\left(
Q_{j}^{k}\right) }\hspace{-0.4cm}A_{i}^{k+m+m_{0}}\int_{Q_{i}^{k+m+m_{0}}}%
\mathcal{L}_{j}^{k}\left( \mathbf{1}_{E_{j}^{k}}\omega \right) \sigma \right]
^{2} \\
& +\sum_{\substack{ k\in \mathbb{Z},k\geq L  \\ j\in \mathbb{N}}}\frac{%
|E_{j}^{k}|_{\omega }}{|3Q_{j}^{k}|_{\omega }^{2}}\left[ \sum_{\substack{ %
i\in \mathbb{N}:\ (k+m+m_{0},i)\in \Gamma  \\ Q_{i}^{k+m+m_{0}}\subset
Q_{j}^{k}}}A_{i}^{k+m+m_{0}}\int_{Q_{i}^{k+m+m_{0}}}\mathcal{L}%
_{j}^{k}\left( \mathbf{1}_{E_{j}^{k}}\omega \right) \sigma \right] ^{2} \\
& =:IV+V.
\end{align*}%
It is relatively easy to estimate term $V$ by the Cauchy-Schwarz inequality
and \eqref{eq:a2ljk}, 
\begin{align}
V& =\sum_{\substack{ k\in \mathbb{Z},k\geq L  \\ j\in \mathbb{N}}}\frac{%
|E_{j}^{k}|_{\omega }}{|3Q_{j}^{k}|_{\omega }^{2}}\Big[\sum_{\substack{ i\in 
\mathbb{N}:\ (k+m+m_{0},i)\in \Gamma  \\ Q_{i}^{k+m+m_{0}}\subset Q_{j}^{k}}}%
A_{i}^{k+m+m_{0}}\int_{Q_{i}^{k+m+m_{0}}}\mathcal{L}_{j}^{k}\big(\mathbf{1}%
_{E_{j}^{k}}\omega \big)\sigma \Big]^{2}  \label{V est} \\
& \leq \sum_{\substack{ k\in \mathbb{Z},k\geq L  \\ j\in \mathbb{N}}}\frac{%
\left\vert E_{j}^{k}\right\vert _{\omega }}{\left\vert 3Q_{j}^{k}\right\vert
_{\omega }^{2}}\Big[\sum_{\substack{ i\in \mathbb{N}:\ (k+m+m_{0},i)\in
\Gamma  \\ Q_{i}^{k+m+m_{0}}\subset Q_{j}^{k}}}|Q_{i}^{k+m+m_{0}}|_{\sigma }%
\big(A_{i}^{k+m+m_{0}}\big)^{2}\Big]  \notag \\
& \times \Big[\sum_{\substack{ i\in \mathbb{N}:\ (k+m+m_{0},i)\in \Gamma  \\ %
Q_{i}^{k+m+m_{0}}\subset Q_{j}^{k}}}\Big(\int_{Q_{i}^{k+m+m_{0}}}\mathcal{L}%
_{j}^{k}\big(\mathbf{1}_{E_{j}^{k}}\omega \big)d\sigma \Big)%
^{2}|Q_{i}^{k+m+m_{0}}|_{\sigma }^{-1}\Big]  \notag \\
& \leq \sum_{\substack{ k\in \mathbb{Z},k\geq L  \\ j\in \mathbb{N}}}\frac{%
\left\vert E_{j}^{k}\right\vert _{\omega }}{\left\vert 3Q_{j}^{k}\right\vert
_{\omega }^{2}}\Big[\sum_{\substack{ i\in \mathbb{N}:\ (k+m+m_{0},i)\in
\Gamma  \\ Q_{i}^{k+m+m_{0}}\subset Q_{j}^{k}}}|Q_{i}^{k+m+m_{0}}|_{\sigma }%
\big(A_{i}^{k+m+m_{0}}\big)^{2}\Big]\int_{Q_{j}^{k}}\big[\mathcal{L}_{j}^{k}%
\big(\mathbf{1}_{Q_{j}^{k}}\omega \big)\big]^{2}d\sigma  \notag \\
& \lesssim A_{2}\sum_{(t,u)\in \Gamma }(A_{u}^{t})^{2}|Q_{u}^{t}|_{\sigma
}\lesssim A_{2}\Vert f\Vert _{L^{2}(\sigma )}^{2}.  \notag
\end{align}

Thus we are left to estimate term $IV$ which we decompose as%
\begin{eqnarray*}
IV &=&\sum_{\substack{ k\in \mathbb{Z},k\geq L  \\ j\in \mathbb{N}\text{ and 
}Q_{j}^{k}\in \mathcal{D}_{\mathbf{r}-\limfunc{good}}}}\frac{%
|E_{j}^{k}|_{\omega }}{|3Q_{j}^{k}|_{\omega }^{2}}\left[ \sum_{i\in \mathbb{N%
}:\ P\left( Q_{i}^{k+m+m_{0}}\right) =P\left( Q_{j}^{k}\right) }\hspace{%
-0.4cm}A_{i}^{k+m+m_{0}}\int_{Q_{i}^{k+m+m_{0}}}\mathcal{L}_{j}^{k}\left( 
\mathbf{1}_{E_{j}^{k}}\omega \right) \sigma \right] ^{2} \\
&&+\sum_{\substack{ k\in \mathbb{Z},k\geq L  \\ j\in \mathbb{N}\text{ and }%
Q_{j}^{k}\in \mathcal{D}_{\mathbf{r}-\limfunc{bad}}}}\frac{%
|E_{j}^{k}|_{\omega }}{|3Q_{j}^{k}|_{\omega }^{2}}\left[ \sum_{i\in \mathbb{N%
}:\ P\left( Q_{i}^{k+m+m_{0}}\right) =P\left( Q_{j}^{k}\right) }\hspace{%
-0.4cm}A_{i}^{k+m+m_{0}}\int_{Q_{i}^{k+m+m_{0}}}\mathcal{L}_{j}^{k}\left( 
\mathbf{1}_{E_{j}^{k}}\omega \right) \sigma \right] ^{2} \\
&\equiv &IV_{\mathbf{r}-\limfunc{good}}+IV_{\mathbf{r}-\limfunc{bad}}.
\end{eqnarray*}%
Fix $(t,u)$, and consider the sum%
\begin{eqnarray*}
IV_{\mathbf{r}-\limfunc{good}}^{t,u} &\equiv &\sum_{Q_{j}^{k}\in \mathcal{W}%
\cap \mathcal{D}_{\mathbf{r}-\limfunc{good}}:\ Q_{j}^{k}\subset Q_{u}^{t}}%
\frac{|E_{j}^{k}|_{\omega }}{|3Q_{j}^{k}|_{\omega }^{2}}\left[ \sum_{i\in 
\mathbb{N}:\ P(Q_{i}^{k+m+m_{0}})=P(Q_{j}^{k})}\hspace{-0.4cm}%
A_{i}^{k+m+m_{0}}\int_{Q_{i}^{k+m+m_{0}}}\mathcal{L}_{j}^{k}\left( \mathbf{1}%
_{E_{j}^{k}}\omega \right) \sigma \right] ^{2} \\
&\leq &C\sum_{Q_{j}^{k}\in \mathcal{W}\cap \mathcal{D}_{\mathbf{r}-\limfunc{%
good}}:\ Q_{j}^{k}\subset Q_{u}^{t}}\frac{|E_{j}^{k}|_{\omega }}{%
|3Q_{j}^{k}|_{\omega }^{2}}\left[ \sum_{i\in \mathbb{N}:\
P(Q_{i}^{k+m+m_{0}})=P(Q_{j}^{k})}\hspace{-0.4cm}\int_{Q_{i}^{k+m+m_{0}}}%
\mathcal{L}_{j}^{k}\left( \mathbf{1}_{E_{j}^{k}}\omega \right) \sigma \right]
^{2}\left( A_{u}^{t}\right) ^{2} \\
&\leq &C\left( \sum_{Q\in \mathcal{W}\cap \mathcal{D}_{\mathbf{r}-\limfunc{%
good}}:\ Q\subset Q_{u}^{t}}q\left( Q\right) \right) \left( A_{u}^{t}\right)
^{2}=C\mathcal{S}_{\mathbf{r}-\limfunc{good}}^{t,u}\left( A_{u}^{t}\right)
^{2},
\end{eqnarray*}%
where 
\begin{equation*}
\mathcal{S}_{\mathbf{r}-\limfunc{good}}^{t,u}\equiv \sum_{Q\in \mathcal{W}%
\cap \mathcal{D}_{\mathbf{r}-\limfunc{good}}:\ Q\subset Q_{u}^{t}}q\left(
Q\right) .
\end{equation*}%
It is here in estimating $\mathcal{S}_{\mathbf{r}-\limfunc{good}}^{t,u}$,
that the only quantitative use of the triple testing condition occurs.

\begin{lemma}
\label{Ki}We claim that%
\begin{equation}
\mathcal{S}_{\mathbf{r}-\limfunc{good}}^{t,u}\leq C\left( \left( \mathfrak{T}%
_{\mathcal{M}}\left( 3\right) \right) ^{2}+A_{2}\right) \left\vert
Q_{u}^{t}\right\vert _{\sigma }\ .  \label{nontrip good bound}
\end{equation}
\end{lemma}

\begin{proof}
Let $\left\{ K_{i}\right\} _{i\in \mathcal{I}}$ be the collection of maximal 
$\mathcal{D}$-cubes $K_{i}$ satisfying $5K_{i}\subset Q_{u}^{t}$. Then for
all cubes $K_{i}$ we have%
\begin{eqnarray*}
\sum_{Q\in \mathcal{W}:\ Q\subset K_{i}}q\left( Q\right) &\leq
&\sum_{Q_{j}^{k}\in \mathcal{W},Q_{j}^{k}\subset K_{i}}\left\vert
E_{j}^{k}\right\vert _{\omega }\left[ \frac{1}{\left\vert
Q_{j}^{k}\right\vert _{\omega }}\int_{Q_{j}^{k}}\mathbf{1}_{K_{i}}\mathcal{M}%
\left( \mathbf{1}_{K_{i}}\sigma \right) d\omega \right] ^{2} \\
&\leq &C\int \left[ \mathcal{M}_{\omega }^{\mathcal{D}^{\gamma }}\big(%
\mathbf{1}_{K_{i}}\mathcal{M}_{M}\left( \mathbf{1}_{K_{i}}\sigma \right) %
\big)\right] ^{2}d\omega \\
&\leq &C\int_{K_{i}}\mathcal{M}\left( \mathbf{1}_{K_{i}}\sigma \right)
^{2}d\omega \lesssim \left( \mathfrak{T}_{\mathcal{M}}\left( 3\right)
\right) ^{2}|3K_{i}|_{\sigma }\ .
\end{eqnarray*}%
Thus we have%
\begin{equation*}
\sum_{i\in \mathcal{I}}\sum_{Q\in \mathcal{W}:\ Q\subset K_{i}}q\left(
Q\right) \leq \sum_{i\in \mathcal{I}}\left( \mathfrak{T}_{\mathcal{M}}\left(
3\right) \right) ^{2}\left\vert 3K_{i}\right\vert _{\sigma }\leq C_{\limfunc{%
bound}}\left( \mathfrak{T}_{\mathcal{M}}\left( 3\right) \right)
^{2}\left\vert Q_{u}^{t}\right\vert _{\sigma }\ ,
\end{equation*}%
where $C_{\limfunc{bound}}$ is a constant such that $\sum_{i\in \mathcal{I}}%
\mathbf{1}_{3K_{i}}\leq C_{\limfunc{bound}}\mathbf{1}_{Q_{u}^{t}}$. We also
have%
\begin{equation*}
\sum_{Q\in \mathcal{W}:\ Q\subset Q_{u}^{t}\text{ and }\ell \left( Q\right)
\geq 2^{-\mathbf{r}}\ell \left( Q_{u}^{t}\right) }q\left( Q\right) \leq
C\sum_{Q\in \mathcal{W}:\ Q\subset Q_{u}^{t}\text{ and }\ell \left( Q\right)
\geq 2^{-\mathbf{r}}\ell \left( Q_{u}^{t}\right) }CA_{2}\left\vert
Q\right\vert _{\sigma }\leq C2^{n\mathbf{r}}A_{2}\left\vert
Q_{u}^{t}\right\vert _{\sigma }\ .
\end{equation*}%
Finally, we note that if a cube $Q\in \mathcal{W}$ is contained in $%
Q_{u}^{t} $ and satisfies $\ell \left( Q\right) <2^{-\mathbf{r}}\ell \left(
Q_{u}^{t}\right) $, but is \textbf{not} contained in any $K_{i}$, then $Q$
is $\mathbf{r}$-$\limfunc{bad}$. Indeed, if we consider the tiling of $%
Q_{u}^{t}$ by dyadic subcubes $Q$ of side length $\ell \left( Q\right)
=2^{-m}\ell \left( Q_{u}^{t}\right) $ for some fixed $m>\mathbf{r}$, then
the only cubes $Q$ in this tiling that do not satisfy $5Q\subset Q_{u}^{t}$
are those for which $\overline{3Q}\cap \overline{\partial Q_{u}^{t}}\neq
\emptyset $. This completes the proof of (\ref{nontrip good bound}) and
hence that of Lemma \ref{Ki}.
\end{proof}

Then summing over $\left( t,u\right) \in \Gamma $ we obtain%
\begin{eqnarray}
IV &=&IV_{\mathbf{r}-\limfunc{good}}+IV_{\mathbf{r}-\limfunc{bad}%
}=\sum_{\left( t,u\right) \in \Gamma }\left( A_{u}^{t}\right) ^{2}IV_{%
\mathbf{r}-\limfunc{good}}^{t,u}+IV_{\mathbf{r}-\limfunc{bad}}
\label{second collection} \\
&\lesssim &\left( \left( \mathfrak{T}_{\mathcal{M}}\left( 3\right) \right)
^{2}+A_{2}\right) \sum_{\left( t,u\right) \in \Gamma }|Q_{u}^{t}|_{\sigma
}\left( A_{u}^{t}\right) ^{2}+III_{\mathbf{r}-\limfunc{bad}}^{\ast }  \notag
\\
&\lesssim &\left( \left( \mathfrak{T}_{\mathcal{M}}\left( 3\right) \right)
^{2}+A_{2}\right) \left\Vert f\right\Vert _{L^{2}\left( \sigma \right)
}^{2}+III_{\mathbf{r}-\limfunc{bad}}^{\ast }\ ,  \notag
\end{eqnarray}%
which combined with (\ref{V est}) gives%
\begin{equation}
\sum_{(k,j)\in \Pi _{3}}2^{2k}\left\vert E_{j}^{k}\right\vert _{\omega }\leq
III\leq \left( \left( \mathfrak{T}_{\mathcal{M}}\left( 3\right) \right)
^{2}+A_{2}\right) \left\Vert f\right\Vert _{L^{2}\left( \sigma \right)
}^{2}+III_{\mathbf{r}-\limfunc{bad}}^{\ast }\ .  \label{case 3 est}
\end{equation}

\subsubsection{Wrapup of the proof}

Now letting the integer $L\rightarrow -\infty $ in the construction of
principal cubes, and summing over $0\leq k_{0}\leq m+m_{0}-1$ in our
convention regarding distinguished index pairs, we obtain from (\ref{lim inf
average}) that 
\begin{eqnarray}
&&\int_{\mathbb{R}^{n}}\left[ \mathcal{M}\left( f\sigma \right) \left(
x\right) \right] ^{2}d\omega \left( x\right) \lesssim \boldsymbol{E}_{\Omega
}^{\mathcal{D}}\int_{\mathbb{R}^{n}}\left[ \mathcal{M}^{\mathcal{D}}\left(
f\sigma \right) \left( x\right) \right] ^{2}d\omega \left( x\right)
\label{wrap} \\
&\lesssim &\boldsymbol{E}_{\Omega }^{\mathcal{D}}\left( \sum_{\text{all }%
(k,j)}2^{2k}\right) +3^{n}C_{n}2^{-2m_{0}}\int \left[ \mathcal{M}\left(
f\sigma \right) \right] ^{2}d\omega  \notag \\
&\lesssim &\boldsymbol{E}_{\Omega }^{\mathcal{D}}\left( \sum_{(k,j)\in \Pi
_{1}}2^{2k}+\sum_{(k,j)\in \Pi _{2}}2^{2k}+\sum_{(k,j)\in \Pi
_{3}}2^{2k}\right) +3^{n}C_{n}2^{-2m_{0}}\int \left[ \mathcal{M}\left(
f\sigma \right) \right] ^{2}d\omega ,  \notag
\end{eqnarray}%
which by the estimates (\ref{case 1 est}), (\ref{case 2 est}) and (\ref{case
3 est}), together with (\ref{alto}), then gives%
\begin{eqnarray}
&&\int_{\mathbb{R}^{n}}\left[ \mathcal{M}\left( f\sigma \right) \left(
x\right) \right] ^{2}d\omega \left( x\right)  \label{wrap'} \\
&\lesssim &\left( \beta +2^{-2m_{0}}\right) \int \mathcal{M}\left( f\sigma
\right) ^{2}d\omega +\beta ^{-1}C_{m+m_{0}}^{2}A_{2}\Vert f\Vert
_{L^{2}(\sigma )}^{2}+\left( \left( \mathfrak{T}_{\mathcal{M}}\left(
3\right) \right) ^{2}+A_{2}\right) \Vert f\Vert _{L^{2}(\sigma )}^{2}  \notag
\\
&&\ \ \ \ \ \ \ \ \ \ \ \ \ \ \ \ \ \ \ \ \ \ \ \ \ \ \ \ \ \ +\sup_{N<0<M}%
\boldsymbol{E}_{\Omega }^{\mathcal{D}}\left( III_{\mathbf{r}-\limfunc{bad}%
}^{\ast }\left( M,N+\mathbf{r}+1\right) \right)  \notag \\
&\lesssim &\left( \beta +2^{-2m_{0}}+2^{-\mathbf{r}}\right) \int \mathcal{M}%
\left( f\sigma \right) ^{2}d\omega +\beta ^{-1}C_{m+m_{0}}^{2}A_{2}\Vert
f\Vert _{L^{2}(\sigma )}^{2}+\left( \left( \mathfrak{T}_{\mathcal{M}}\left(
3\right) \right) ^{2}+A_{2}\right) \Vert f\Vert _{L^{2}(\sigma )}^{2}. 
\notag
\end{eqnarray}%
Now we can absorb the first term on the right hand side by choosing $\beta
>0 $ sufficiently small and $m_{0}$ and $\mathbf{r}$ sufficiently large
since the integral $\int \mathcal{M}\left( f\sigma \right) ^{2}d\omega $ is
finite. Then we take the supremum over $f\in L^{2}\left( \sigma \right) $
with $\left\Vert f\right\Vert _{L^{2}(\sigma )}=1$ to obtain 
\begin{equation*}
\mathfrak{N}_{\mathcal{M}}\leq C\left( \mathfrak{T}_{\mathcal{M}}\left(
3\right) +\sqrt{A_{2}}\right) \ .
\end{equation*}%
As the opposite inequality is trivial, this completes the proof of Theorem %
\ref{maximal}.

\section{Weak triple testing\label{Sec weak}}

Now we adapt the previous arguments to prove our main result, Theorem \ref%
{weak}. Recall that given a pair $\left( \sigma ,\omega \right) $ of weights
(i.e. locally finite positive Borel measures) in $\mathbb{R}^{n}$ and $%
D,\Gamma >1$, we say that $\left( \sigma ,\omega \right) $ satisfies the $D$%
\emph{-}$\Gamma $\emph{-testing condition} for the maximal function $%
\mathcal{M}$ if there is a constant $\mathfrak{T}_{\mathcal{M}}^{D}\left(
\Gamma \right) \left( \sigma ,\omega \right) $ such that 
\begin{equation*}
\int_{Q}\left\vert \mathcal{M}\mathbf{1}_{Q}\sigma \right\vert ^{2}d\omega
\leq \mathfrak{T}_{\mathcal{M}}^{D}\left( \Gamma \right) \left( \sigma
,\omega \right) ^{2}\left\vert Q\right\vert _{\sigma }\ ,\ \ \ \ \ \text{for
all }Q\in \mathcal{P}^{n}\text{ with }\left\vert \Gamma Q\right\vert
_{\sigma }\leq D\left\vert Q\right\vert _{\sigma }\ ,
\end{equation*}%
and if so we denote by $\mathfrak{T}_{\mathcal{M}}^{D}\left( \Gamma \right)
\left( \sigma ,\omega \right) $ the least such constant. Here again is the
main result of this paper.

\begin{theorem}
Let $\Gamma >1$. Then there is $D>1$ depending only on $\Gamma $ and the
dimension $n$ such that%
\begin{equation}
\mathfrak{N}_{\mathcal{M}}\left( \sigma ,\omega \right) \approx \mathfrak{T}%
_{\mathcal{M}}^{D}\left( \Gamma \right) \left( \sigma ,\omega \right) +\sqrt{%
A_{2}\left( \sigma ,\omega \right) },  \label{final}
\end{equation}%
for all locally finite positive Borel measures $\sigma $ and $\omega $ on $%
\mathbb{R}^{n}$.
\end{theorem}

To begin the proof, we point out the well known fact that for locally finite
positive Borel measures $\sigma $ and $\omega $,%
\begin{equation}
\boldsymbol{P}_{\Omega }\left( \left\{ \mathcal{D}\in \Omega :\left\vert
\partial Q\right\vert _{\sigma }+\left\vert \partial Q\right\vert _{\omega
}>0\text{ for some }Q\in \mathcal{D}\right\} \right) =0.
\label{boundary prob}
\end{equation}%
Indeed, for $0\leq k\leq n-1$, there are at most countably many $k$-planes
parallel to the coordinate $k$-planes that are charged by $\sigma +\omega $.
Now note that with probability zero, a random grid $\mathcal{D}\in \Omega $
includes a cube $Q\in \mathcal{D}$ whose boundary $\partial Q$ contains one
of these countably many $k$-planes. More precisely, consider the subcase of
hyperplanes ($k=n-1$) parallel to the hyperplane 
\begin{equation*}
P_{n}\equiv \left\{ \left( x_{1},...,x_{n-1},0\right) :\left(
x_{1},...,x_{n-1}\right) \in \mathbb{R}^{n-1}\right\}
\end{equation*}%
that passes through the origin and is perpendicular to the $x_{n}$-axis. Let 
$\digamma \equiv \left\{ z\in \mathbb{R}:\left\vert P_{n}+\left(
0,0,...,0,z\right) \right\vert _{\sigma +\omega }>0\right\} $. If $B\left(
0,j\right) =\left\{ x\in \mathbb{R}^{n}:\left\vert x\right\vert <j\right\} $
is the ball of radius $j$, then the sets 
\begin{equation*}
\digamma _{j}\equiv \left\{ z\in \mathbb{R}:\left\vert B\left( 0,j\right)
\cap \left( P_{n}+\left( 0,0,...,0,z\right) \right) \right\vert _{\sigma
+\omega }>\frac{1}{j}\right\}
\end{equation*}%
are clearly finite for each $j$ since the measure $\sigma +\omega $ is
locally finite, i.e. $\left\vert B\left( 0,j\right) \right\vert _{\sigma
+\omega }<\infty $, and it follows that $\digamma $ is at most countable.
Now if $\mathcal{D}\in \Omega $ is any grid, and if $\partial _{n}\mathcal{D}
$ denotes the collection of all hyperplanes $P$ that are parallel to $P_{n}$
and contain a face of a dyadic cube from $\mathcal{D}$, then $\partial _{n}%
\mathcal{D}$ is countable. Thus with $\mathcal{D}_{0}$ equal to the standard
dyadic grid on $\mathbb{R}^{n}$, this shows that the set of $t\in \mathbb{R}%
^{n}$ such that $\partial _{n}\left( \mathcal{D}_{0}+t\right) \cap \digamma
\neq \emptyset $ has Lebesgue measure zero, and thus that%
\begin{equation*}
\boldsymbol{P}_{\Omega }\left( \left\{ \mathcal{D}\in \Omega :\partial _{n}%
\mathcal{D}\cap \digamma \neq \emptyset \right\} \right) =0.
\end{equation*}%
Now we repeat this calculation for hyperplanes parallel to $P_{i}$, where $%
P_{i}$ is the hyperplane perpendicular to the $x_{i}$-axis. And then we
perform similar calculations for $k$-planes with $0\leq k\leq n-2$. The case 
$k=0$ is particularly easy since the set of points in $\mathbb{R}^{n}$ that
are charged by $\sigma +\omega $ is clearly countable.

We say that a random grid $\mathcal{D}\in \Omega $ has \emph{null boundaries}
if $\left\vert \partial Q\right\vert _{\sigma +\omega }=0$ for all cubes $%
Q\in \mathcal{D}$, and set $\Omega ^{\mathrm{null}}\equiv \left\{ \mathcal{D}%
\in \Omega :\mathcal{D}\text{ has }\emph{null\ boundaries}\right\} $, $%
\mathcal{P}^{\mathrm{null}}\equiv \bigcup\limits_{\mathcal{D}\in \Omega ^{%
\mathrm{null}}}\mathcal{D}$ and for any positive Borel measure $f$ on $%
\mathbb{R}^{n}$, 
\begin{equation*}
\mathcal{M}^{\mathrm{null}}f\left( x\right) \equiv \sup_{Q\in \mathcal{P}^{%
\mathrm{null}}:\ x\in Q}\frac{1}{\left\vert Q\right\vert }\int_{Q}f.
\end{equation*}%
Then, using that $\boldsymbol{P}_{\Omega }\left( \Omega \setminus \Omega ^{%
\mathrm{null}}\right) =0$, equivalently $\boldsymbol{P}_{\Omega }\left(
\Omega ^{\mathrm{null}}\right) =1$, together with (\ref{lim inf average}) in
Lemma \ref{domination}, we have%
\begin{eqnarray*}
\mathcal{M}f\left( x\right) &\leq &2^{n+3}\boldsymbol{E}_{\Omega }^{\mathcal{%
D}}\mathcal{M}^{\mathcal{D}}f\left( x\right) =2^{n+3}\boldsymbol{E}_{\Omega
^{\mathrm{null}}}^{\mathcal{D}}\mathcal{M}^{\mathcal{D}}f\left( x\right) \\
&\leq &2^{n+3}\sup_{\mathcal{D}\in \Omega ^{\mathrm{null}}}\mathcal{M}^{%
\mathcal{D}}f\left( x\right) \leq 2^{n+3}\sup_{Q\in \mathcal{P}^{\mathrm{null%
}}:x\in Q}\frac{1}{\left\vert Q\right\vert }\int_{Q}f=\mathcal{M}^{\mathrm{%
null}}f\left( x\right) \leq 2^{n+3}\mathcal{M}f\left( x\right)
\end{eqnarray*}%
for all positive Borel measures $f$ on $\mathbb{R}^{n}$. Thus we have the
pointwise equivalence 
\begin{equation}
\mathcal{M}^{\mathrm{null}}f\left( x\right) \approx \sup_{\mathcal{D}\in
\Omega ^{\mathrm{null}}}\mathcal{M}^{\mathcal{D}}f\left( x\right) \approx 
\mathcal{M}f\left( x\right) ,  \label{point equiv}
\end{equation}%
and in particular, we conclude that (\ref{final}) is equivalent to%
\begin{equation}
\mathfrak{N}_{\mathcal{M}^{\mathrm{null}}}\left( \sigma ,\omega \right)
\approx \mathfrak{T}_{\mathcal{M}}^{D}\left( \Gamma \right) \left( \sigma
,\omega \right) +\sqrt{A_{2}\left( \sigma ,\omega \right) }.  \label{final'}
\end{equation}

We complete the proof of (\ref{final'}), and hence of Theorem \ref{weak}, by
modifying the proof of Theorem \ref{maximal} in the following seven steps.

\medskip

\textbf{Step 1}: There were only two places in the proof of Theorem \ref%
{maximal} where the hypothesis of triple testing was used:

\begin{enumerate}
\item qualitatively, at the beginning of the argument, in order to assume
without loss of generality that $\int M\left( f\sigma \right) ^{2}d\omega
<\infty $,

\item and quantitatively, near the end of the argument, in the proof of
Lemma \ref{Ki}.
\end{enumerate}

The qualitative use of the triple testing condition is easily handled using $%
D$-triple testing as follows. If we replace $\omega $ by $\omega _{N}=\omega 
\mathbf{1}_{B(0,N)}$ with $N>R$ where $\omega $ is supported in $B\left(
0,R\right) $, then the $D$-triple testing condition and $A_{2}$ condition
still hold, and with constants no larger than before. Moreover, the \emph{%
testing condition} for the cube $Q_{m}=\left[ -3^{m}N,3^{m}N\right] $ must
hold for some $m\geq 0$, since otherwise iteration of the inequality $%
\left\vert Q_{m}\right\vert _{\sigma }\leq \frac{1}{D}\left\vert
Q_{m+1}\right\vert _{\sigma }$ eventually violates the $A_{2}$ condition, 
\begin{equation*}
A_{2}\left( \sigma ,\omega \right) \geq \frac{\left\vert Q_{m}\right\vert
_{\sigma }\left\vert Q_{m}\right\vert _{\omega }}{\left\vert
Q_{m}\right\vert ^{2}}\geq \frac{D^{m}\left\vert Q_{0}\right\vert _{\sigma
}\left\vert Q_{0}\right\vert _{\omega }}{2^{2mn}\left\vert Q_{0}\right\vert
^{2}}=\left( \frac{D}{2^{2n}}\right) ^{m}\frac{\left\vert Q_{0}\right\vert
_{\sigma }\left\vert Q_{0}\right\vert _{\omega }}{\left\vert
Q_{0}\right\vert ^{2}},
\end{equation*}%
if $D$ is chosen greater than $2^{2n+1}$. Thus if the \emph{testing condition%
} holds for the cube $Q_{m}$ we have 
\begin{equation*}
\int \mathcal{M}\left( f\sigma \right) ^{2}d\omega _{N}\leq \Vert f\Vert
_{L^{\infty }}\int_{B(0,N)}\mathcal{M}(\mathbf{1}_{Q_{m}}\sigma )^{2}d\omega
<\infty ,
\end{equation*}%
and therefore, without loss of generality, we can assume $\int \mathcal{M}%
(f\sigma )^{2}d\omega <\infty $.

For the quantitative use of the triple testing condition, recall that Lemma %
\ref{Ki} asserted%
\begin{equation*}
\mathcal{S}_{\mathbf{r}-\mathrm{good}}^{t,u}=\sum_{Q\in \mathcal{W}\cap 
\mathcal{D}_{\mathbf{r}-\mathrm{good}}:\ Q\subset Q_{u}^{t}}q\left( Q\right)
\leq C\left( \left( \mathfrak{T}_{\mathcal{M}}\left( 3\right) \right)
^{2}+A_{2}\right) \left\vert Q_{u}^{t}\right\vert _{\sigma }\ ,
\end{equation*}%
where $\mathcal{W}$ was the Whitney grid associated with a given dyadic grid 
$\mathcal{D}\in \Omega $, and $\mathcal{D}_{\mathbf{r}-\mathrm{good}}$ was
the associated subgrid of $\mathbf{r}-\mathrm{good}$ cubes. The triple
testing condition was used only in the inequality%
\begin{equation}
\sum_{Q\in \mathcal{W}:\ Q\subset K_{i}}q\left( Q\right) \leq C\int_{K_{i}}%
\mathcal{M}\left( \mathbf{1}_{K_{i}}\sigma \right) ^{2}d\omega \leq C\left( 
\mathfrak{T}_{\mathcal{M}}\left( 3\right) \right) ^{2}|3K_{i}|_{\sigma }\ .
\label{ineq used}
\end{equation}%
However, if we only have the testing condition over $D$-tripling cubes, then
we have%
\begin{equation*}
\int_{K_{i}}\mathcal{M}\left( \mathbf{1}_{K_{i}}\sigma \right) ^{2}d\omega
\leq \left( \mathfrak{T}_{\mathcal{M}}^{D}\left( 3\right) \right)
^{2}\left\vert K_{i}\right\vert _{\sigma }\leq \left( \mathfrak{T}_{\mathcal{%
M}}^{D}\left( 3;\mathcal{D}\right) \right) ^{2}\left\vert 3K_{i}\right\vert
_{\sigma },\ \ \ \ \ \text{if }\left\vert 3K_{i}\right\vert _{\sigma }\leq
D\left\vert K_{i}\right\vert _{\sigma }\ ,
\end{equation*}%
where in the testing condition%
\begin{equation*}
\mathfrak{T}_{\mathcal{M}}^{D}\left( 3;\mathcal{D}\right) \equiv \sup 
_{\substack{ K\in \mathcal{D}  \\ \left\vert 3K\right\vert _{\sigma }\leq
D\left\vert K\right\vert _{\sigma }}}\sqrt{\frac{1}{\left\vert K\right\vert
_{\sigma }}\int_{K}\mathcal{M}\left( \mathbf{1}_{K}\sigma \right)
^{2}d\omega },
\end{equation*}%
the supremum is taken over only $\mathcal{D}$-dyadic cubes $K$ satisfying $%
\left\vert 3K\right\vert _{\sigma }\leq D\left\vert K\right\vert _{\sigma }$%
. On the other hand, for the $D$-nontripling cubes $K_{i}$, we can only use
the inequality%
\begin{equation*}
\int_{K_{i}}\mathcal{M}\left( \mathbf{1}_{K_{i}}\sigma \right) ^{2}d\omega
\leq \left( \mathfrak{T}_{\mathcal{M}}\right) ^{2}\left\vert
K_{i}\right\vert _{\sigma }\leq \frac{1}{D}\left( \mathfrak{T}_{\mathcal{M}%
}\right) ^{2}\left\vert 3K_{i}\right\vert _{\sigma }\ ,
\end{equation*}%
where%
\begin{equation*}
\mathfrak{T}_{\mathcal{M}}\equiv \sup_{K\in \mathcal{P}}\sqrt{\frac{1}{%
\left\vert K\right\vert _{\sigma }}\int_{K}\mathcal{M}\left( \mathbf{1}%
_{K}\sigma \right) ^{2}d\omega },
\end{equation*}%
and this gives%
\begin{equation*}
\int_{K_{i}}\mathcal{M}\left( \mathbf{1}_{K_{i}}\sigma \right) ^{2}d\omega
\leq \frac{1}{D}\left( \mathfrak{T}_{\mathcal{M}}\right) ^{2}\left\vert
3K_{i}\right\vert _{\sigma },\ \ \ \ \ \text{if }\left\vert
3K_{i}\right\vert _{\sigma }>D\left\vert K_{i}\right\vert _{\sigma }\ .
\end{equation*}%
Altogether then we obtain%
\begin{equation*}
\sum_{Q\in \mathcal{W}:\ Q\subset K_{i}}q\left( Q\right) \leq C\left[ \left( 
\mathfrak{T}_{\mathcal{M}}^{D}\left( 3;\mathcal{D}\right) \right) ^{2}+\frac{%
1}{D}\left( \mathfrak{T}_{\mathcal{M}}\right) ^{2}\right] |3K_{i}|_{\sigma
}\ ,
\end{equation*}%
and the only difference from (\ref{ineq used}) is that $\left( \mathfrak{T}_{%
\mathcal{M}}\left( 3\right) \right) ^{2}$ has been replaced with $\left( 
\mathfrak{T}_{\mathcal{M}}^{D}\left( 3;\mathcal{D}\right) \right) ^{2}+\frac{%
1}{D}\left( \mathfrak{T}_{\mathcal{M}}\right) ^{2}$. As a consequence, the
previous inequalities (\ref{wrap}) and (\ref{wrap'}), together with the fact
that $\boldsymbol{P}_{\Omega }\left( \Omega \setminus \Omega ^{\mathrm{null}%
}\right) =0$, can be modified to yield the inequalities (where $\Omega $
gets replaced by $\Omega ^{\mathrm{null}}$),%
\begin{eqnarray*}
&&\int_{\mathbb{R}^{n}}\left[ \mathcal{M}\left( f\sigma \right) \left(
x\right) \right] ^{2}d\omega \left( x\right) \lesssim \boldsymbol{E}_{\Omega
^{\mathrm{null}}}^{\mathcal{D}}\int_{\mathbb{R}^{n}}\left[ \mathcal{M}^{%
\mathcal{D}}\left( f\sigma \right) \left( x\right) \right] ^{2}d\omega
\left( x\right) \\
&\lesssim &\boldsymbol{E}_{\Omega ^{\mathrm{null}}}^{\mathcal{D}}\left(
\sum_{\text{all }(k,j)}2^{2k}\right) +3^{n}C_{n}2^{-2m_{0}}\int \left[ 
\mathcal{M}\left( f\sigma \right) \right] ^{2}d\omega \\
&\lesssim &\boldsymbol{E}_{\Omega ^{\mathrm{null}}}^{\mathcal{D}}\left(
\sum_{(k,j)\in \Pi _{1}}2^{2k}+\sum_{(k,j)\in \Pi _{2}}2^{2k}+\sum_{(k,j)\in
\Pi _{3}}2^{2k}\right) +3^{n}C_{n}2^{-2m_{0}}\int \left[ \mathcal{M}\left(
f\sigma \right) \right] ^{2}d\omega ,
\end{eqnarray*}%
and%
\begin{eqnarray*}
&&\int_{\mathbb{R}^{n}}\left[ \mathcal{M}\left( f\sigma \right) \left(
x\right) \right] ^{2}d\omega \left( x\right) \\
&\lesssim &\left( \beta +2^{-2m_{0}}\right) \int \mathcal{M}\left( f\sigma
\right) ^{2}d\omega +\beta ^{-1}C_{m+m_{0}}^{2}A_{2}\Vert f\Vert
_{L^{2}(\sigma )}^{2}+\left( \left( \mathfrak{T}_{\mathcal{M}}\left( 3;%
\mathcal{P}^{\mathrm{null}}\right) \right) ^{2}+A_{2}\right) \Vert f\Vert
_{L^{2}(\sigma )}^{2} \\
&&\ \ \ \ \ \ \ \ \ \ \ \ \ \ \ \ \ \ \ \ \ \ \ \ \ \ \ \ \ \ +\sup_{N<0<M}%
\boldsymbol{E}_{\Omega ^{\mathrm{null}}}^{\mathcal{D}}\left( III_{\mathbf{r}-%
\mathrm{bad}}^{\ast }\left( M,N+\mathbf{r}+1\right) \right) \\
&\lesssim &\left( \beta +2^{-2m_{0}}+2^{-\mathbf{r}}\right) \int \mathcal{M}%
\left( f\sigma \right) ^{2}d\omega +\beta ^{-1}C_{m+m_{0}}^{2}A_{2}\Vert
f\Vert _{L^{2}(\sigma )}^{2} \\
&&+\left( \left( \mathfrak{T}_{\mathcal{M}}\left( 3;\mathcal{P}^{\mathrm{null%
}}\right) \right) ^{2}+A_{2}+\frac{1}{\sqrt{D}}\left( \mathfrak{T}_{\mathcal{%
M}}\right) ^{2}\right) \Vert f\Vert _{L^{2}(\sigma )}^{2},
\end{eqnarray*}%
which, after absorption of the first term on the right hand side, give the
conclusion that%
\begin{equation*}
\mathfrak{N}_{\mathcal{M}}\left( \sigma ,\omega \right) \leq C\left( 
\mathfrak{T}_{\mathcal{M}}^{D}\left( 3;\mathcal{P}^{\mathrm{null}}\right)
\left( \sigma ,\omega \right) +\sqrt{A_{2}\left( \sigma ,\omega \right) }+%
\frac{1}{\sqrt{D}}\mathfrak{T}_{\mathcal{M}}\left( \sigma ,\omega \right)
\right) \ ,
\end{equation*}%
where%
\begin{equation*}
\mathfrak{T}_{\mathcal{M}}^{D}\left( 3;\mathcal{P}^{\mathrm{null}}\right)
\equiv \sup_{\substack{ Q\in \mathcal{P}^{\mathrm{null}}  \\ \left\vert
3Q\right\vert _{\sigma }\leq D\left\vert Q\right\vert _{\sigma }}}\sqrt{%
\frac{1}{\left\vert Q\right\vert _{\sigma }}\int_{Q}\mathcal{M}\left( 
\mathbf{1}_{K}\sigma \right) ^{2}d\omega }\approx \sup_{\substack{ Q\in 
\mathcal{P}^{\mathrm{null}}  \\ \left\vert 3Q\right\vert _{\sigma }\leq
D\left\vert Q\right\vert _{\sigma }}}\sqrt{\frac{1}{\left\vert Q\right\vert
_{\sigma }}\int_{Q}\mathcal{M}^{\mathrm{null}}\left( \mathbf{1}_{K}\sigma
\right) ^{2}d\omega }\ .
\end{equation*}%
Thus in the testing constant $\mathfrak{T}_{\mathcal{M}}^{D}\left( 3;%
\mathcal{P}^{\mathrm{null}}\right) $, the supremum over cubes $Q$ is
restricted to those cubes $Q$ satisfying both $\left\vert \partial
Q\right\vert _{\sigma }+\left\vert \partial Q\right\vert _{\omega }=0$ and $%
\left\vert 3Q\right\vert _{\sigma }\leq D\left\vert Q\right\vert _{\sigma }$.

\medskip

\textbf{Step 2}: If $\mathfrak{T}_{\mathcal{M}}\left( \sigma ,\omega \right)
<\infty $, then Step 1 shows that $\mathfrak{N}_{\mathcal{M}}\left( \sigma
,\omega \right) <\infty $, and since we trivially have $\mathfrak{T}_{%
\mathcal{M}}\left( \sigma ,\omega \right) \leq \mathfrak{N}_{\mathcal{M}%
}\left( \sigma ,\omega \right) $, we can then absorb the term $\frac{1}{%
\sqrt{D}}\mathfrak{T}_{\mathcal{M}}\left( \sigma ,\omega \right) $ into the
left hand side of the inequality to obtain the \emph{apriori} inequality%
\begin{equation*}
\mathfrak{N}_{\mathcal{M}}\left( \sigma ,\omega \right) \leq C\left( 
\mathfrak{T}_{\mathcal{M}}^{D}\left( 3;\mathcal{P}^{\mathrm{null}}\right)
\left( \sigma ,\omega \right) +\sqrt{A_{2}\left( \sigma ,\omega \right) }%
\right) ,\ \ \ \ \ \text{whenever }\mathfrak{T}_{\mathcal{M}}\left( \sigma
,\omega \right) <\infty .
\end{equation*}%
However, noting that all of the above holds with $\Gamma ^{\prime }>1$ in
place of $3$ and a corresponding $D^{\prime }=D^{\prime }\left( \Gamma
^{\prime },n\right) >1$ in place of $D$, we obtain%
\begin{equation*}
\mathfrak{N}_{\mathcal{M}}\left( \sigma ,\omega \right) \leq C\left( 
\mathfrak{T}_{\mathcal{M}}^{D^{\prime }}\left( \Gamma ^{\prime };\mathcal{P}%
^{\mathrm{null}}\right) \left( \sigma ,\omega \right) +\sqrt{A_{2}\left(
\sigma ,\omega \right) }\right) ,
\end{equation*}%
for a fixed $D^{\prime }=D^{\prime }\left( \Gamma ^{\prime },n\right) >1$
depending only on $\Gamma ^{\prime }$ and dimension $n$, and where the cubes 
$Q$ used to define the testing condition $\mathfrak{T}_{\mathcal{M}%
}^{D^{\prime }}\left( \Gamma ^{\prime };\mathcal{P}^{\mathrm{null}}\right)
\left( \sigma ,\omega \right) $ are restricted to those $Q$ satisfying both $%
\left\vert \partial Q\right\vert _{\sigma }+\left\vert \partial Q\right\vert
_{\omega }=0$ and $\left\vert \Gamma ^{\prime }Q\right\vert _{\sigma }\leq
D^{\prime }\left\vert Q\right\vert _{\sigma }$.

\subsection{Approximation by mollified weights}

It remains to appropriately approximate the measure pair $\left( \sigma
,\omega \right) $ by a family of measure pairs $\left( \sigma _{\varepsilon
},\omega _{\varepsilon ^{\prime }}\right) $ for which $\mathfrak{N}_{%
\mathcal{M}}\left( \sigma _{\varepsilon },\omega _{\varepsilon ^{\prime
}}\right) <\infty $. A standard mollification will serve this purpose.

\medskip

\textbf{Step 3}: Suppose that $\omega $ is supported in the compact cube $%
K=Q\left( 0,R\right) \equiv \left[ -R,R\right] ^{n}$. Fix $\varphi :\left(
-1,1\right) ^{n}\rightarrow \left[ 0,1\right] $ smooth and compactly
supported in $\left( -1,1\right) ^{n}$ with $\varphi \geq 2^{-n}$ on $\left(
-\frac{5}{8},\frac{5}{8}\right) ^{n}$ and $\int \varphi =1$. For $%
0<\varepsilon <1$ define $\varphi _{\varepsilon }\left( x\right)
=\varepsilon ^{-n}\varphi \left( \frac{x}{\varepsilon }\right) $ and%
\begin{equation*}
\sigma _{\varepsilon }\equiv \sigma \ast \varphi _{\varepsilon }\text{ and }%
\omega _{\varepsilon ^{\prime }}\equiv \omega \ast \varphi _{\varepsilon
^{\prime }}\ ,\ \ \ \ \ 0<\varepsilon ,\varepsilon ^{\prime }<1.
\end{equation*}%
We claim that 
\begin{equation*}
\mathfrak{T}_{\mathcal{M}}\left( \sigma _{\varepsilon },\omega _{\varepsilon
^{\prime }}\right) <\infty ,\ \ \ \ \ \text{for }0<\varepsilon ,\varepsilon
^{\prime }<\frac{1}{4}.
\end{equation*}%
Indeed, $d\sigma _{\varepsilon }\left( x\right) =s_{\varepsilon }\left(
x\right) dx$ and $d\omega _{\varepsilon ^{\prime }}\left( x\right)
=w_{\varepsilon ^{\prime }}\left( x\right) dx$ where $s_{\varepsilon }$ and $%
w_{\varepsilon ^{\prime }}$ are smooth functions, and thus if $Q\in \mathcal{%
P}$ we have 
\begin{eqnarray*}
\int_{Q}\mathcal{M}\left( \mathbf{1}_{Q}\sigma _{\varepsilon }\right)
^{2}d\omega _{\varepsilon ^{\prime }} &\leq &\left\Vert w_{\varepsilon
^{\prime }}\right\Vert _{\infty }\int_{Q}\mathcal{M}\left( \mathbf{1}%
_{Q}\sigma _{\varepsilon }\right) \left( x\right) ^{2}dx \\
&\leq &C_{\mathrm{class}}^{2}\left\Vert w_{\varepsilon ^{\prime
}}\right\Vert _{\infty }\int_{Q}\sigma _{\varepsilon }\left( x\right) ^{2}dx
\\
&\leq &C_{\mathrm{class}}^{2}\left\Vert w_{\varepsilon ^{\prime
}}\right\Vert _{\infty }\left\Vert \mathbf{1}_{Q}\sigma _{\varepsilon
}\right\Vert _{\infty }\int_{Q}\sigma _{\varepsilon }\left( x\right) dx,
\end{eqnarray*}%
where $C_{\mathrm{class}}$ is the classical bound for $M$ on (unweighted) $%
L^{2}$. Now $\left\Vert w_{\varepsilon ^{\prime }}\right\Vert _{\infty
}<\infty $ since $\omega $ is compactly supported, and by the same token, $%
\sup_{Q\subset 3K}\left\Vert \mathbf{1}_{Q}\sigma _{\varepsilon }\right\Vert
_{\infty }\leq \left\Vert \mathbf{1}_{3K}\sigma _{\varepsilon }\right\Vert
_{\infty }<\infty $, yielding%
\begin{equation*}
\int_{Q}\mathcal{M}\left( \mathbf{1}_{Q}\sigma _{\varepsilon }\right)
^{2}d\omega _{\varepsilon ^{\prime }}\leq CC_{\mathrm{class}%
}^{2}\int_{Q}\sigma _{\varepsilon }\left( x\right) dx,\ \ \ \ \ \text{for }%
Q\subset 3K.
\end{equation*}%
So it only remains to consider a cube $Q$ with $Q\cap B_{K}(\varepsilon
^{\prime })\neq \emptyset $ and $Q\cap \left[ \mathbb{R}^{n}\setminus 3K%
\right] \neq \emptyset $, where $B_{K}(\varepsilon ^{\prime })=\cup _{x\in
K}Q(x,\varepsilon ^{\prime })$. We may assume $K$ is big enough, e.g., $\ell
(K)=2R>100$. Then $B_{K}(\varepsilon ^{\prime })\subset \frac{51}{50}K$ and
we can write%
\begin{equation*}
\int_{Q}\mathcal{M}\left( \mathbf{1}_{Q}\sigma _{\varepsilon }\right)
^{2}d\omega _{\varepsilon ^{\prime }}\lesssim \int_{Q\cap B_{K}(\varepsilon
^{\prime })}\mathcal{M}\left( \mathbf{1}_{Q\cap 3K}\sigma _{\varepsilon
}\right) ^{2}d\omega _{\varepsilon ^{\prime }}+\int_{Q\cap B_{K}(\varepsilon
^{\prime })}\mathcal{M}\left( \mathbf{1}_{Q\setminus 3K}\sigma _{\varepsilon
}\right) ^{2}d\omega _{\varepsilon ^{\prime }}
\end{equation*}%
where the first term is handled using the estimates in the above, i.e.%
\begin{equation*}
\int_{Q\cap B_{K}(\varepsilon ^{\prime })}\mathcal{M}\left( \mathbf{1}%
_{Q\cap 3K}\sigma _{\varepsilon }\right) ^{2}d\omega _{\varepsilon ^{\prime
}}\leq C_{\mathrm{class}}^{2}\left\Vert w_{\varepsilon ^{\prime
}}\right\Vert _{\infty }\left\Vert \mathbf{1}_{3K}\sigma _{\varepsilon
}\right\Vert _{\infty }\int_{Q}\sigma _{\varepsilon }\left( x\right) dx\leq
CC_{\mathrm{class}}^{2}\int_{Q}\sigma _{\varepsilon }\left( x\right) dx\ ,
\end{equation*}%
and the second term satisfies%
\begin{eqnarray*}
\int_{Q\cap B_{K}(\varepsilon ^{\prime })}\mathcal{M}\left( \mathbf{1}%
_{Q\setminus 3K}\sigma _{\varepsilon }\right) ^{2}d\omega _{\varepsilon
^{\prime }} &\leq &C\left\vert \frac{51}{50}K\right\vert _{\omega
_{\varepsilon ^{\prime }}}\mathcal{M}\left( \mathbf{1}_{Q\setminus 3K}\sigma
_{\varepsilon }\right) \left( c_{K}\right) ^{2} \\
&\leq &C\left\vert 3K\right\vert _{\omega _{\varepsilon ^{\prime }}}\left[
\sup_{\ell \geq 2}\frac{1}{3^{\ell n}\left\vert K\right\vert }\int_{Q\cap
\left( 3^{\ell }K\setminus 3K\right) }d\sigma _{\varepsilon }\left( x\right) %
\right] ^{2} \\
&\leq &C\left[ \sup_{\ell \geq 2}\frac{\left\vert 3^{\ell }K\right\vert
_{\omega _{\varepsilon ^{\prime }}}}{\left( 3^{\ell n}\left\vert
K\right\vert \right) ^{2}}\int_{3^{\ell }K\setminus 3K}d\sigma _{\varepsilon
}\left( x\right) \right] \left\vert Q\right\vert _{\sigma _{\varepsilon }} \\
&\leq &CA_{2}(\sigma ,\omega )\left\vert Q\right\vert _{\sigma _{\varepsilon
}}<\infty ,
\end{eqnarray*}%
where we have used the fact that $|3^{\ell }K|_{\sigma _{\varepsilon
}}\lesssim |B(3^{\ell }K,\varepsilon )|_{\sigma }\leq |\frac{51}{50}3^{\ell
}K|_{\sigma }$ and similar estimates for $|3^{\ell }K|_{\omega _{\varepsilon
^{\prime }}}$.

\medskip

\textbf{Step 4}: Combining Steps 2 and 3, we obtain%
\begin{equation}
\mathfrak{N}_{\mathcal{M}}\left( \sigma _{\varepsilon },\omega _{\varepsilon
^{\prime }}\right) \leq C\left( \mathfrak{T}_{\mathcal{M}}^{D^{\prime
}}\left( \Gamma ^{\prime };\mathcal{P}^{\mathrm{null}}\right) \left( \sigma
_{\varepsilon },\omega _{\varepsilon ^{\prime }}\right) +\sqrt{A_{2}\left(
\sigma _{\varepsilon },\omega _{\varepsilon ^{\prime }}\right) }\right) ,\ \
\ \ \ \text{for }0<\varepsilon ,\varepsilon ^{\prime }<1,  \label{def D'}
\end{equation}%
for a fixed $D^{\prime }=D^{\prime }\left( \Gamma ^{\prime },n\right) >1$
depending only on $\Gamma ^{\prime }$ and dimension $n$, and where the cubes 
$Q$ used to define the testing condition $\mathfrak{T}_{\mathcal{M}%
}^{D^{\prime }}\left( \Gamma ^{\prime };\mathcal{P}^{\mathrm{null}}\right)
\left( \sigma _{\varepsilon },\omega _{\varepsilon ^{\prime }}\right) $ are
restricted to those $Q$ satisfying both $\left\vert \partial Q\right\vert
_{\sigma }+\left\vert \partial Q\right\vert _{\omega }=0$ and $\left\vert
\Gamma ^{\prime }Q\right\vert _{\sigma }\leq D^{\prime }\left\vert
Q\right\vert _{\sigma }$.

We will now prove the general statement in Theorem \ref{weak} for $\Gamma >1$%
, namely that given $\Gamma >1$, there is $D=D\left( \Gamma ,n\right) >1$
such that (\ref{final'}) holds. We will do this by choosing any fixed $%
\Gamma ^{\prime }>\Gamma $, e.g. $\Gamma ^{\prime }=\Gamma +1$ works just
fine, and then proving that (\ref{final'}) holds\ with the constant $D$
given by 
\begin{equation*}
D\equiv 2^{n}D^{\prime }\left( \Gamma ^{\prime },n\right) ,
\end{equation*}%
where $D^{\prime }\left( \Gamma ^{\prime },n\right) $ is the constant in (%
\ref{def D'}).

From this point\ on we will consider only pairs $\left( \varepsilon
,\varepsilon ^{\prime }\right) $ with $\frac{\varepsilon }{\varepsilon
^{\prime }}=8$, and so we will replace the pair $\left( \varepsilon
,\varepsilon ^{\prime }\right) $ with $\left( 8\varepsilon ,\varepsilon
\right) $. We claim that for $\Gamma ^{\prime }$ and $D^{\prime }$ chosen as
above, namely $\Gamma ^{\prime }>\Gamma $ and $D^{\prime }>1$ so that (\ref%
{def D'}) holds, then we have 
\begin{eqnarray}
\mathfrak{N}_{\mathcal{M}^{\mathrm{null}}}\left( \sigma ,\omega \right)
&\leq &\liminf_{\varepsilon \searrow 0}\mathfrak{N}_{\mathcal{M}}\left(
\sigma _{8\varepsilon },\omega _{\varepsilon }\right)  \label{claim D} \\
&\lesssim &\liminf_{\varepsilon \searrow 0}\mathfrak{T}_{\mathcal{M}%
}^{D^{\prime }}\left( \Gamma ^{\prime }\right) \left( \sigma _{8\varepsilon
},\omega _{\varepsilon }\right) +\liminf_{\varepsilon \searrow 0}\sqrt{%
A_{2}\left( \sigma _{8\varepsilon },\omega _{\varepsilon }\right) }  \notag
\\
&\lesssim &\mathfrak{T}_{\mathcal{M}}^{D}\left( \Gamma \right) \left( \sigma
,\omega \right) +\sqrt{A_{2}\left( \sigma ,\omega \right) },  \notag
\end{eqnarray}%
which, once established, will complete the proof of (\ref{final'}), and
hence that of Theorem \ref{weak}. We will prove (\ref{claim D}) by proving
three assertions, namely%
\begin{eqnarray}
\mathfrak{N}_{\mathcal{M}^{\mathrm{null}}}\left( \sigma ,\omega \right)
&\lesssim &\liminf_{\varepsilon \searrow 0}\mathfrak{N}_{\mathcal{M}}\left(
\sigma _{8\varepsilon },\omega _{\varepsilon }\right) ,  \label{claim D'} \\
\sup_{0<\varepsilon <\frac{1}{8}}\sqrt{A_{2}\left( \sigma _{8\varepsilon
},\omega _{\varepsilon }\right) } &\lesssim &\sqrt{A_{2}\left( \sigma
,\omega \right) },  \notag \\
\sup_{0<\varepsilon <\frac{1}{8}}\mathfrak{T}_{\mathcal{M}}^{D^{\prime
}}\left( \Gamma ^{\prime }\right) \left( \sigma _{8\varepsilon },\omega
_{\varepsilon }\right) &\lesssim &\mathfrak{T}_{\mathcal{M}}^{D}\left(
\Gamma \right) \left( \sigma ,\omega \right) +\sqrt{A_{2}\left( \sigma
,\omega \right) }.  \notag
\end{eqnarray}

\medskip

\textbf{Step 5}: We begin with the first line in (\ref{claim D'}), and prove
that%
\begin{equation}
\mathfrak{N}_{\mathcal{M}^{\mathrm{null}}}\left( \sigma ,\omega \right)
\lesssim \mathfrak{T}_{\mathcal{M}^{\mathrm{null}}}\left( \mathcal{P}^{%
\mathrm{null}}\right) \left( \sigma ,\omega \right) \leq
\liminf_{\varepsilon \searrow 0}\mathfrak{N}_{\mathcal{M}}\left( \sigma
_{8\varepsilon },\omega _{\varepsilon }\right) ,  \label{null testing}
\end{equation}%
where 
\begin{equation*}
\mathfrak{T}_{\mathcal{M}^{\mathrm{null}}}\left( \mathcal{P}^{\mathrm{null}%
}\right) \left( \sigma ,\omega \right) \equiv \sup_{Q\in \mathcal{P}^{%
\mathrm{null}}}\sqrt{\frac{1}{\left\vert Q\right\vert _{\sigma }}\int_{Q}%
\mathcal{M}^{\mathrm{null}}\left( \mathbf{1}_{Q}\sigma \right) ^{2}d\omega }.
\end{equation*}%
The first inequality in (\ref{null testing}) follows from (\ref{point equiv}%
) and $\mathfrak{N}_{\mathcal{M}}\left( \sigma ,\omega \right) \lesssim 
\mathfrak{T}_{\mathcal{M}}\left( \mathcal{P}^{\mathrm{null}}\right) \left(
\sigma ,\omega \right) $, which in turn follows from the observation that,
with probability one, the grids $\mathcal{D}\in \Omega $ used in (\ref{lim
inf average}) of Lemma \ref{domination} above have \emph{null boundaries},
and so contain only cubes belonging to the collection $\mathcal{P}^{\mathrm{%
null}}$. Thus Theorem \ref{maximal} yields the even stronger conclusion that%
\footnote{%
One can also obtain the first inequality in (\ref{null testing}) by
observing that, with probability one, the grids used in Lemma 2 of \cite%
{Saw3} contain only cubes in $\mathcal{P}^{\mathrm{null}}$.} 
\begin{eqnarray*}
\mathfrak{N}_{\mathcal{M}}\left( \sigma ,\omega \right) &\lesssim &\mathfrak{%
T}_{\mathcal{M}}\left( 3;\mathcal{P}^{\mathrm{null}}\right) \left( \sigma
,\omega \right) ; \\
\text{where }\mathfrak{T}_{\mathcal{M}}\left( 3;\mathcal{P}^{\mathrm{null}%
}\right) \left( \sigma ,\omega \right) &\equiv &\sup_{Q\in \mathcal{P}^{%
\mathrm{null}}}\sqrt{\frac{1}{\left\vert 3Q\right\vert _{\sigma }}\int_{Q}%
\mathcal{M}\left( \mathbf{1}_{Q}\sigma \right) ^{2}d\omega }.
\end{eqnarray*}

Now we turn to proving the second inequality in (\ref{null testing}). Fix $%
Q\in \mathcal{P}^{\mathrm{null}}$. We begin by noting that 
\begin{equation}  \label{eq:epsequiv}
\left\vert \left( 1-\frac{2\varepsilon}{\ell(Q)} \right) Q\right\vert
_{\sigma }\leq \left\vert Q\right\vert _{\sigma _{\varepsilon }}\leq
|B_Q(\varepsilon)| \le \left\vert \left( 1+\frac{2\varepsilon} {\ell(Q)}%
\right) Q\right\vert _{\sigma }\ ,
\end{equation}%
so that by the regularity of locally finite positive Borel measures on $%
\mathbb{R}^{n}$, together with $\left\vert \partial Q\right\vert _{\sigma
}=0 $, we have 
\begin{eqnarray*}
\limsup_{\varepsilon \searrow 0}\left\vert Q\right\vert _{\sigma
_{\varepsilon }} \le \left\vert Q\right\vert _{\sigma } &\le &
\liminf_{\varepsilon \searrow 0}\left\vert Q\right\vert _{\sigma
_{\varepsilon }}\ , \\
&\Longrightarrow &\left\vert Q\right\vert _{\sigma }=\lim_{\varepsilon
\searrow 0}\left\vert Q\right\vert _{\sigma _{\varepsilon }}.
\end{eqnarray*}%
A similar argument shows that $\left\vert R\right\vert _{\sigma
}=\lim_{\varepsilon \searrow 0}\left\vert R\right\vert _{\sigma
_{\varepsilon }}$ for any rectangle $R=Q\cap K$ with $Q,K\in \mathcal{P}^{%
\mathrm{null}}$. Moreover, we also have $\left\vert R\right\vert _{\omega
}=\lim_{\varepsilon \searrow 0}\left\vert R\right\vert _{\omega
_{\varepsilon }}$ for any rectangle $R=Q\cap K$ with $Q,K\in \mathcal{P}^{%
\mathrm{null}}$.

Next, still supposing that $Q\in \mathcal{P}^{\mathrm{null}}$, we claim that%
\begin{equation}
\mathcal{M}^{\mathrm{null}}\left( \mathbf{1}_{Q}\sigma \right) \left(
x\right) \leq \liminf_{\varepsilon \searrow 0}\mathcal{M}^{\mathrm{null}%
}\left( \mathbf{1}_{Q}\sigma _{\varepsilon }\right) \left( x\right) ,\ \ \ \
\ x\in \mathbb{R}^{n}.  \label{lim inf dom}
\end{equation}%
Indeed, given $\delta >0$, there is a cube $K\in \mathcal{P}^{\mathrm{null}} 
$ such that $x\in K$ and%
\begin{equation*}
\mathcal{M}^{\mathrm{null}}\left( \mathbf{1}_{Q}\sigma \right) \left(
x\right) -\delta <\frac{\left\vert Q\cap K\right\vert _{\sigma }}{\left\vert
K\right\vert }.
\end{equation*}%
Then since $\left\vert Q\cap K\right\vert _{\sigma }=\lim_{\varepsilon
\searrow 0}\left\vert Q\cap K\right\vert _{\sigma _{\varepsilon }}$ for the
rectangle $R=Q\cap K$, we have%
\begin{equation*}
\mathcal{M}^{\mathrm{null}}\left( \mathbf{1}_{Q}\sigma \right) \left(
x\right) -\delta <\frac{\left\vert Q\cap K\right\vert _{\sigma }}{\left\vert
K\right\vert }=\lim_{\varepsilon \searrow 0}\frac{\left\vert Q\cap
K\right\vert _{\sigma _{\varepsilon }}}{\left\vert K\right\vert }\leq
\liminf_{\varepsilon \searrow 0}\mathcal{M}\left( \mathbf{1}_{Q}\sigma
_{\varepsilon }\right) \left( x\right) ,
\end{equation*}%
which proves (\ref{lim inf dom}) upon letting $\delta \searrow 0$. An
application of Fatou's lemma then gives for $Q\in \mathcal{P}^{\mathrm{null}%
} $ that%
\begin{equation*}
\frac{1}{\left\vert Q\right\vert _{\sigma }}\int_{Q}\mathcal{M}^{\mathrm{null%
}}\left( \mathbf{1}_{Q}\sigma \right) ^{2}d\omega \leq \frac{1}{\left\vert
Q\right\vert _{\sigma }}\int_{Q}\liminf_{\varepsilon \searrow 0}\mathcal{M}%
\left( \mathbf{1}_{Q}\sigma _{\varepsilon }\right) ^{2}d\omega \leq \frac{1}{%
\left\vert Q\right\vert _{\sigma }}\liminf_{\varepsilon \searrow 0}\int_{Q}%
\mathcal{M}\left( \mathbf{1}_{Q}\sigma _{\varepsilon }\right) ^{2}d\omega .
\end{equation*}

We next observe that the following oscillation inequality holds for any cube 
$Q\in $ $\mathcal{P}$:%
\begin{equation}
\mathcal{M}\left( \mathbf{1}_{Q}\sigma _{\varepsilon }\right) \left(
x\right) \leq C\mathcal{M}\left( \mathbf{1}_{Q}\sigma _{8\varepsilon
}\right) \left( x+h\right) ,\ \ \ \ \ x,h\in \mathbb{R}^{n}\text{ with }%
\left\vert h\right\vert <\varepsilon <\frac{1}{8}.  \label{osc ineq}
\end{equation}%
Indeed, we have%
\begin{equation*}
\mathcal{M}\left( \mathbf{1}_{Q}\sigma _{\varepsilon }\right) \left(
x\right) \leq \sup_{\substack{ K\in \mathcal{P}:\ x\in K  \\ \ell \left(
K\right) \geq \varepsilon }}\frac{\left\vert Q\cap K\right\vert _{\sigma
_{\varepsilon }}}{\left\vert K\right\vert }+\sup_{\substack{ K\in \mathcal{P}%
:\ x\in K  \\ \ell \left( K\right) <\varepsilon }}\frac{\left\vert Q\cap
K\right\vert _{\sigma _{\varepsilon }}}{\left\vert K\right\vert },
\end{equation*}%
and using the inequality $\varphi _{\varepsilon }\left( z\right) \leq
C\varphi _{8\varepsilon }\left( z+h\right) $ for $\left\vert h\right\vert
<\varepsilon $, we obtain that for any $\left\vert h\right\vert <\varepsilon 
$, 
\begin{eqnarray*}
\sup_{\substack{ K\in \mathcal{P}:\ x\in K  \\ \ell \left( K\right) \geq
\varepsilon }}\frac{\left\vert Q\cap K\right\vert _{\sigma _{\varepsilon }}}{%
\left\vert K\right\vert } &=&\sup_{\substack{ K\in \mathcal{P}:\ x\in K  \\ %
\ell \left( K\right) \geq \varepsilon }}\frac{1}{\left\vert K\right\vert }%
\int_{Q\cap K}\left( \int \varphi _{\varepsilon }\left( x-y\right) d\sigma
\left( y\right) \right) dx \\
&\leq &\sup_{\substack{ K\in \mathcal{P}:\ x\in K  \\ \ell \left( K\right)
\geq \varepsilon }}\frac{1}{\left\vert K\right\vert }\int_{Q\cap K}\left(
\int C\varphi _{8\varepsilon }\left( x+h-y\right) d\sigma \left( y\right)
\right) dx \\
&=&C\sup_{\substack{ K\in \mathcal{P}:\ x\in K  \\ \ell \left( K\right) \geq
\varepsilon }}\frac{1}{\left\vert K\right\vert }\int_{Q\cap K}\sigma
_{8\varepsilon }\left( x+h\right) dx \\
&\leq &C\sup_{K\in \mathcal{P}:\ x+h\in B_K(\varepsilon )}\frac{1}{%
\left\vert B_K(\varepsilon )\right\vert }\int_{Q\cap \left( B_K(\varepsilon
)\right) }\sigma _{8\varepsilon }\left( x\right) dx \\
&\leq &C\mathcal{M}\left( \mathbf{1}_{Q}\sigma _{8\varepsilon }\right)
\left( x+h\right) .
\end{eqnarray*}%
We also have 
\begin{eqnarray*}
\sup_{\substack{ K\in \mathcal{P}:\ x\in K  \\ \ell \left( K\right)
<\varepsilon }}\frac{\left\vert Q\cap K\right\vert _{\sigma _{\varepsilon }}%
}{\left\vert K\right\vert } &\leq &\sup_{\substack{ K\in \mathcal{P}:\ x\in
K  \\ \ell \left( K\right) <\varepsilon }}\left\Vert \mathbf{1}_{Q\cap
K}\sigma _{\varepsilon }\right\Vert _{\infty }=\sup_{\substack{ K\in 
\mathcal{P}:\ x\in K  \\ \ell \left( K\right) <\varepsilon }}\sup_{z\in
Q\cap K}\int \varphi _{\varepsilon }\left( z-y\right) d\sigma \left( y\right)
\\
&\leq &\sup_{\substack{ K\in \mathcal{P}:\ x\in K  \\ \ell \left( K\right)
<\varepsilon }}\sup_{z\in Q\cap K}\frac{1}{\left\vert B_{\varepsilon
}\right\vert }\int_{B_{\varepsilon }\left( z\right) }d\sigma \left( y\right)
\\
&\leq &\frac{1}{\left\vert B_{\varepsilon }\right\vert }\int_{B_{2%
\varepsilon }\left( x\right) }d\sigma \left( y\right) \leq C\sigma
_{4\varepsilon }\left( x\right) \leq C\mathcal{M}\left( \mathbf{1}_{Q}\sigma
_{8\varepsilon }\right) \left( x+h\right) ,
\end{eqnarray*}%
for $\left\vert h\right\vert <\varepsilon $, where $B_{\delta }(x)$ denote
the \emph{cube} of side length $2\delta $ centered at $x$ (if $x$ is the
origin then we simply denote it by $B_{\delta }$). This completes the proof
of the oscillation inequality (\ref{osc ineq}). With (\ref{osc ineq}), it
follows immediately that for any cube $Q\in \mathcal{P}$ that 
\begin{eqnarray*}
\int_{Q}\mathcal{M}\left( \mathbf{1}_{Q}\sigma _{\varepsilon }\right)
^{2}d\omega &\leq &\int_{Q}\int \varphi _{\varepsilon }\left( h\right)
\left\{ C\mathcal{M}\left( \mathbf{1}_{Q}\sigma _{8\varepsilon }\right)
\left( x+h\right) \right\} ^{2}dhd\omega \left( x\right) \\
&=&\int_{Q}\left\{ C\mathcal{M}\left( \mathbf{1}_{Q}\sigma _{8\varepsilon
}\right) \left( x\right) \right\} ^{2}\left\{ \int \varphi _{\varepsilon
}\left( h\right) d\omega \left( x-h\right) dh\right\} \\
&=&C^{2}\int_{Q}\mathcal{M}\left( \mathbf{1}_{Q}\sigma _{8\varepsilon
}\right) ^{2}d\omega _{\varepsilon }\ .
\end{eqnarray*}%
Now, restricting to cubes $Q\in \mathcal{P}^{\mathrm{null}}$, and using $%
\left\vert Q\right\vert _{\sigma }=\lim_{\varepsilon \searrow 0}\left\vert
Q\right\vert _{\sigma _{8\varepsilon }}$, we have 
\begin{eqnarray}
\frac{1}{\left\vert Q\right\vert _{\sigma }}\int_{Q}\mathcal{M}^{\mathrm{null%
}}\left( \mathbf{1}_{Q}\sigma \right) ^{2}d\omega &\leq
&\liminf_{\varepsilon \searrow 0}\frac{1}{\left\vert Q\right\vert _{\sigma }}%
\int_{Q}\mathcal{M}\left( \mathbf{1}_{Q}\sigma _{\varepsilon }\right)
^{2}d\omega  \label{test con} \\
&\leq &C^{2}\liminf_{\varepsilon \searrow 0}\frac{1}{\left\vert Q\right\vert
_{\sigma }}\int_{Q}\mathcal{M}\left( \mathbf{1}_{Q}\sigma _{8\varepsilon
}\right) ^{2}d\omega _{\varepsilon }  \notag \\
&=&C^{2}\liminf_{\varepsilon \searrow 0}\frac{1}{\left\vert Q\right\vert
_{\sigma _{8\varepsilon }}}\int_{Q}\mathcal{M}\left( \mathbf{1}_{Q}\sigma
_{8\varepsilon }\right) ^{2}d\omega _{\varepsilon }  \notag \\
&\leq &C^{2}\liminf_{\varepsilon \searrow 0}\mathfrak{N}_{\mathcal{M}}\left(
\sigma _{8\varepsilon },\omega _{\varepsilon }\right) \ ,  \notag
\end{eqnarray}%
which is a bound independent of the cube $Q$. If we now take the supremum
over all cubes $Q\in \mathcal{P}^{\mathrm{null}}$ we obtain 
\begin{eqnarray*}
\mathfrak{T}_{\mathcal{M}^{\mathrm{null}}}\left( \mathcal{P}^{\mathrm{null}%
}\right) \left( \sigma ,\omega \right) &=&\sup_{Q\in \mathcal{P}^{\mathrm{%
null}}}\frac{1}{\left\vert Q\right\vert _{\sigma }}\int_{Q}\mathcal{M}^{%
\mathrm{null}}\left( \mathbf{1}_{Q}\sigma \right) ^{2}d\omega \\
&\leq &\sup_{Q\in \mathcal{P}^{\mathrm{null}}}C^{2}\liminf_{\varepsilon
\searrow 0}\mathfrak{N}_{\mathcal{M}}\left( \sigma _{8\varepsilon },\omega
_{\varepsilon }\right) =C^{2}\liminf_{\varepsilon \searrow 0}\mathfrak{N}_{%
\mathcal{M}}\left( \sigma _{8\varepsilon },\omega _{\varepsilon }\right) ,
\end{eqnarray*}%
which completes the proof of the second line in (\ref{null testing}), and
hence the first line in (\ref{claim D'}).

\medskip

\textbf{Step 6}: Now we turn to the second line in (\ref{claim D'}), and
prove that%
\begin{equation}
\sup_{0<\varepsilon <\frac{1}{8}}A_{2}\left( \sigma _{8\varepsilon },\omega
_{\varepsilon }\right) \leq \sup_{0<\varepsilon <\frac{1}{8}}\sup_{Q\in 
\mathcal{P}:\ \ell \left( Q\right) \geq \varepsilon }\frac{\left\vert
Q\right\vert _{\sigma _{8\varepsilon }}\left\vert Q\right\vert _{\omega
_{\varepsilon }}}{\left\vert Q\right\vert ^{2}}+\sup_{0<\varepsilon <\frac{1%
}{8}}\sup_{Q\in \mathcal{P}:\ \ell \left( Q\right) <\varepsilon }\frac{%
\left\vert Q\right\vert _{\sigma _{8\varepsilon }}\left\vert Q\right\vert
_{\omega _{\varepsilon }}}{\left\vert Q\right\vert ^{2}}\leq CA_{2}\left(
\sigma ,\omega \right) .  \label{A2 eps bound}
\end{equation}%
Indeed, to see this, we bound the first summand in (\ref{A2 eps bound}) by%
\begin{eqnarray*}
&&\sup_{0<\varepsilon <\frac{1}{8}}\sup_{Q\in \mathcal{P}:\ \ell \left(
Q\right) \geq \varepsilon }\frac{\left\vert Q\right\vert _{\sigma
_{8\varepsilon }}\left\vert Q\right\vert _{\omega _{\varepsilon }}}{%
\left\vert Q\right\vert ^{2}} \\
&\leq &C\sup_{0<\varepsilon <\frac{1}{8}}\sup_{Q\in \mathcal{P}:\ \ell
\left( Q\right) \geq \varepsilon }\left( 1+\frac{16\varepsilon }{\ell \left(
Q\right) }\right) ^{n}\left( 1+\frac{2\varepsilon }{\ell \left( Q\right) }%
\right) ^{n}\frac{\left\vert \left( 1+\frac{16\varepsilon }{\ell \left(
Q\right) }\right) Q\right\vert _{\sigma }\left\vert \left( 1+\frac{%
2\varepsilon }{\ell \left( Q\right) }\right) Q\right\vert _{\omega }}{%
\left\vert \left( 1+\frac{16\varepsilon }{\ell \left( Q\right) }\right)
Q\right\vert \left\vert \left( 1+\frac{2\varepsilon }{\ell \left( Q\right) }%
\right) Q\right\vert } \\
&\leq &CA_{2}\left( \sigma ,\omega \right) .
\end{eqnarray*}
Then using 
\begin{equation*}
\left\vert Q\right\vert _{\sigma _{8\varepsilon }}=\int_{Q}\left\{ \int
\varphi _{8\varepsilon }\left( x-y\right) d\sigma \left( y\right) \right\}
dx=\int \left\{ \int_{Q}\varphi _{8\varepsilon }\left( x-y\right) dx\right\}
d\sigma \left( y\right) \leq C\frac{1}{\left\vert B_{8\varepsilon
}\right\vert }\int_{B_Q(8\varepsilon )}d\sigma
\end{equation*}%
and similarly $\left\vert Q\right\vert _{\omega _{\varepsilon }}\leq C\frac{1%
}{\left\vert B_{\varepsilon }\right\vert }\int_{B_Q(\varepsilon )}d\omega $,
we bound the second summand in (\ref{A2 eps bound}) by 
\begin{eqnarray}
&&\sup_{0<\varepsilon <\frac{1}{8}}\sup_{Q\in \mathcal{P}:\ \ell \left(
Q\right) <\varepsilon }\frac{\left\vert Q\right\vert _{\sigma _{8\varepsilon
}}\left\vert Q\right\vert _{\omega _{\varepsilon }}}{\left\vert Q\right\vert
^{2}}\leq C\sup_{0<\varepsilon <\frac{1}{8}}\sup_{Q\in \mathcal{P}:\ \ell
\left( Q\right) <\varepsilon }\left( \frac{1}{\left\vert B_{8\varepsilon
}\right\vert }\int_{B_Q(8\varepsilon)}d\sigma \right) \left( \frac{1}{%
\left\vert B_{\varepsilon }\right\vert }\int_{B_Q(\varepsilon )}d\omega
\right)  \label{also have} \\
&=&C\sup_{0<\varepsilon <\frac{1}{8}}\sup_{Q\in \mathcal{P}:\ \ell \left(
Q\right) <\varepsilon }\frac{\left\vert B_Q(\varepsilon )\right\vert }{%
\left\vert B_{8\varepsilon }\right\vert }\frac{\left\vert B_Q(\varepsilon
)\right\vert }{\left\vert B_{\varepsilon }\right\vert }\left( \frac{1}{%
\left\vert B_Q(\varepsilon )\right\vert }\int_{B_Q(\varepsilon )}d\sigma
\right) \left( \frac{1}{\left\vert B_Q(\varepsilon )\right\vert }%
\int_{B_Q(\varepsilon )}d\omega \right)  \notag \\
&\leq &CA_{2}\left( \sigma ,\omega \right) .  \notag
\end{eqnarray}%
This completes the proof of (\ref{A2 eps bound}).

\medskip

\textbf{Step 7}: In order to complete the proof of (\ref{claim D}), it
remains to prove the third line in (\ref{claim D'}), namely that for $\Gamma
^{\prime }>\Gamma $ and $D= D^{\prime }$ where $D^{\prime }=D^{\prime
}\left( \Gamma ^{\prime },n\right) $ is such that (\ref{def D'}) holds, we
have%
\begin{equation}
\sup_{0<\varepsilon <\frac{1}{8}}\mathfrak{T}_{\mathcal{M}}^{D^{\prime
}}\left( \Gamma ^{\prime }\right) \left( \sigma _{8\varepsilon },\omega
_{\varepsilon }\right) \leq C\left[ \mathfrak{T}_{\mathcal{M}}^{D}\left(
\Gamma \right) \left( \sigma ,\omega \right) +\sqrt{A_{2}\left( \sigma
,\omega \right) }\right] .  \label{prove both}
\end{equation}

Suppose first that $Q\in \mathcal{P}$ satisfies $\left\vert \Gamma ^{\prime
}Q\right\vert _{\sigma _{8\varepsilon }}\leq D^{\prime }\left\vert
Q\right\vert _{\sigma _{8\varepsilon }}$ and $\varepsilon \leq \ell \left(
Q\right) $. Recall that $B_{\delta }$ is the cube of side length $\delta $
centered at the origin. Fix $x\in Q$ and $\delta >0$, and choose $K\in 
\mathcal{P}$ such that $x\in K$ and%
\begin{eqnarray*}
\mathcal{M}\left( \mathbf{1}_{Q}\sigma _{8\varepsilon }\right) \left(
x\right) -\delta &<&\frac{1}{\left\vert K\right\vert }\int_{K\cap Q}\sigma
_{8\varepsilon }\left( z\right) dz=\frac{1}{\left\vert K\right\vert }%
\int_{K\cap Q}\left\{ \int_{\mathbb{R}^{n}}\varphi _{8\varepsilon }\left(
z-y\right) d\sigma \left( y\right) \right\} dz \\
&=&\frac{1}{\left\vert K\right\vert }\int_{B_{K\cap Q}(\varepsilon )}\left\{
\int_{K\cap Q}\varphi _{8\varepsilon }\left( z-y\right) dz\right\} d\sigma
\left( y\right) \\
&\leq &C\frac{1}{\left\vert K\right\vert }\int_{B_{K\cap Q}(\varepsilon
)}\left\{ \frac{1}{\left\vert B_{8\varepsilon }\right\vert }\int_{K\cap Q}%
\mathbf{1}_{B_{8\varepsilon }}\left( z-y\right) dz\right\} d\sigma \left(
y\right) \\
&=&C\int_{B_{K\cap Q}(\varepsilon )}\left\{ \frac{\left\vert K\cap Q\cap
\left( B_{8\varepsilon }+y\right) \right\vert }{\left\vert K\right\vert
\left\vert B_{8\varepsilon }\right\vert }\right\} d\sigma \left( y\right) .
\end{eqnarray*}%
There are now two subcases, $\varepsilon \leq \ell \left( K\right) $ and $%
\ell \left( K\right) <\varepsilon $. In the first case $\varepsilon \leq
\ell \left( K\right) $ we continue with

\begin{eqnarray*}
&&\int_{B_{K\cap Q}(\varepsilon )}\left\{ \frac{\left\vert K\cap Q\cap
\left( B_{8\varepsilon }+y\right) \right\vert }{\left\vert K\right\vert
\left\vert B_{8\varepsilon }\right\vert }\right\} d\sigma \left( y\right) \\
&\leq &C\frac{1}{\left\vert B_K(8\varepsilon )\right\vert }%
\int_{B_K(8\varepsilon )}\mathbf{1}_{B_Q(8\varepsilon )}\left( y\right)
d\sigma \left( y\right) \leq C\mathcal{M}\left( \mathbf{1}_{B_Q(8\varepsilon
)}\sigma \right) \left( x\right) ,
\end{eqnarray*}%
while in the second case $\ell \left( K\right) <\varepsilon $ we continue
with%
\begin{eqnarray*}
&&\int_{B_{K\cap Q}(\varepsilon )}\left\{ \frac{\left\vert K\cap Q\cap
\left( B_{8\varepsilon }+y\right) \right\vert }{\left\vert K\right\vert
\left\vert B_{8\varepsilon }\right\vert }\right\} d\sigma \left( y\right)
\leq C\frac{1}{\left\vert B_{8\varepsilon }\right\vert }\int_{B_{K\cap
Q}(\varepsilon )}d\sigma \left( y\right) \\
&\leq &C\frac{1}{\left\vert B_K(8\varepsilon )\right\vert }%
\int_{B_K(8\varepsilon )}\mathbf{1}_{B_Q(8\varepsilon )}\left( y\right)
d\sigma \left( y\right) \leq C\mathcal{M}\left( \mathbf{1}_{B_Q(8\varepsilon
)}\sigma \right) \left( x\right) .
\end{eqnarray*}%
Thus altogether we have $\mathcal{M}\left( \mathbf{1}_{Q}\sigma
_{8\varepsilon }\right) \left( x\right) -\delta <C\mathcal{M}\left( \mathbf{1%
}_{B_Q(8\varepsilon )}\sigma \right) \left( x\right) $ for all $\delta >0$,
which yields 
\begin{equation*}
\mathcal{M}\left( \mathbf{1}_{Q}\sigma _{8\varepsilon }\right) \left(
x\right) \leq C\mathcal{M}\left( \mathbf{1}_{B_Q(8\varepsilon )}\sigma
\right) \left( x\right) ,\ \ \ \ \ \text{when }Q\in \mathcal{P}\text{ and }%
\varepsilon \leq \ell \left( Q\right) .
\end{equation*}

Hence, using \eqref{osc ineq} and the restricted testing constant $\mathfrak{%
T}_{\mathcal{M}}^{D}\left( \Gamma \right) $, we will have%
\begin{eqnarray}
&&\int_{Q}\mathcal{M}\left( \mathbf{1}_{Q}\sigma _{8\varepsilon }\right)
^{2}d\omega _{\varepsilon }= \int_{Q}\mathcal{M}\left( \mathbf{1}%
_{Q}\sigma_{8\varepsilon} \right) \left( x\right) ^{2}\left\{\int
\varphi_\varepsilon(x-z)d\omega(z)\right\}dx  \label{rest} \\
&= &\int\int_Q \mathcal{M}\left( \mathbf{1}_{Q}\sigma _{8\varepsilon
}\right)(x) ^{2}\varphi_\varepsilon(x-z)dx d\omega(z)=\int\int_{Q-z} 
\mathcal{M}\left( \mathbf{1}_{Q}\sigma _{8\varepsilon }\right)(z+h)
^{2}\varphi_\varepsilon(h)dh d\omega(z)  \notag \\
&\leq &C^{2} \int\int_{Q-z} \mathcal{M}\left( \mathbf{1}_{Q}\sigma
_{64\varepsilon }\right)(z) ^{2}\varphi_\varepsilon(h)dh d\omega(z) \le
C^{2} \int_{B_Q(\varepsilon)}\int \mathcal{M}\left( \mathbf{1}_{Q}\sigma
_{64\varepsilon }\right)(z) ^{2}\varphi_\varepsilon(h)dh d\omega(z)  \notag
\\
&=& C^{2} \int_{B_Q(\varepsilon)} \mathcal{M}\left( \mathbf{1}_{Q}\sigma
_{64\varepsilon }\right)(z) ^{2} d\omega(z)\le C^{2} \int_{B_Q(\varepsilon)} 
\mathcal{M}\left( \mathbf{1}_{B_Q(64\varepsilon)}\sigma\right)(z) ^{2}
d\omega(z)  \notag \\
&\leq &C^{2}\left( \mathfrak{T}_{\mathcal{M}}^{D}\left( \Gamma \right)
\left( \sigma ,\omega \right) \right) ^{2}\left\vert B_Q(64\varepsilon
)\right\vert _{\sigma }\ ,  \notag
\end{eqnarray}%
provided that $\left\vert \Gamma B_Q(64\varepsilon ) \right\vert _{\sigma
}\leq D\left\vert B_Q(64\varepsilon )\right\vert _{\sigma }$. But we claim
this latter inequality will hold for%
\begin{equation*}
0<\varepsilon \leq \min \left\{ \frac{\frac{\Gamma^{\prime }}{\Gamma}-1}{288}%
,\frac{1}{32}\right\} \ell \left( Q\right) =\alpha \ell \left( Q\right) ,\ \
\ \ \ \text{where }\alpha =\min \left\{ \frac{\frac{\Gamma^{\prime }}{\Gamma}%
-1}{288},\frac{1}{32}\right\}>0.
\end{equation*}%
Indeed, since $B_Q(64\varepsilon )=\left( 1+128\frac{\varepsilon }{\ell
\left( Q\right) }\right) Q$, we have%
\begin{equation}
\left\vert \Gamma B_Q(64\varepsilon ) \right\vert _{\sigma }\leq \left\vert
\Gamma \left(1+\frac{128\varepsilon}{\ell(Q)}\right) \frac1{1-\frac {%
16\varepsilon}{\ell(Q)}}Q\right\vert _{\sigma_{8\varepsilon} }\leq
\left\vert \Gamma ^{\prime }Q\right\vert _{\sigma _{8\varepsilon }}\leq
D^{\prime }\left\vert Q\right\vert _{\sigma _{8\varepsilon }}\leq D^{\prime
}\left\vert B_Q(8\varepsilon )\right\vert _{\sigma }\le D\left\vert
B_Q(64\varepsilon )\right\vert _{\sigma }  \label{Gamma cube}
\end{equation}%
The first inequality is by \eqref{eq:epsequiv}. Then the second inequality
is by using $\varepsilon \le \alpha$. The third inequality in (\ref{Gamma
cube}) follows from our starting assumption that $\left\vert \Gamma ^{\prime
}Q\right\vert _{\sigma _{8\varepsilon }}\leq D^{\prime }\left\vert
Q\right\vert _{\sigma _{8\varepsilon }}$, and the fourth inequality is
trivial. Thus (\ref{Gamma cube}) shows that (\ref{rest}) holds for $%
0<\varepsilon \leq \alpha \ell \left( Q\right) $. Now we note an additional
consequence of (\ref{Gamma cube}), namely that for $0<\varepsilon \leq
\alpha \ell \left( Q\right) $ we have 
\begin{equation}
\left\vert B_Q(64\varepsilon )\right\vert _{\sigma }\leq \left\vert \Gamma
B_Q(64\varepsilon ) \right\vert _{\sigma }\leq D^{\prime }\left\vert
Q\right\vert _{\sigma _{8\varepsilon }}\ .  \label{add con}
\end{equation}%
Thus when $Q\in \mathcal{P}$ is a $D^{\prime }$-$\Gamma ^{\prime }$-$\sigma
_{8\varepsilon }$-cube and $0<\varepsilon \leq \alpha \ell \left( Q\right) $%
, we have from (\ref{rest}) and (\ref{add con}) that%
\begin{eqnarray*}
\int_{Q}\mathcal{M}\left( \mathbf{1}_{Q}\sigma _{8\varepsilon }\right)
^{2}d\omega _{\varepsilon } &\leq &C^{2}\left( \mathfrak{T}_{\mathcal{M}%
}^{D}\left( \Gamma \right) \left( \sigma ,\omega \right) \right)
^{2}\left\vert B_Q(64\varepsilon )\right\vert _{\sigma } \\
&\leq &C^{2}\left( \mathfrak{T}_{\mathcal{M}}^{D}\left( \Gamma \right)
\left( \sigma ,\omega \right) \right) ^{2} D^{\prime }\left\vert
Q\right\vert _{\sigma _{8\varepsilon }} \\
&=& C^{2} D^{\prime }\left( \mathfrak{T}_{\mathcal{M}}^{D}\left( \Gamma
\right) \left( \sigma ,\omega \right) \right) ^{2}\left\vert Q\right\vert
_{\sigma _{8\varepsilon }}\ ,
\end{eqnarray*}%
or in other words,%
\begin{equation*}
\frac{1}{\left\vert Q\right\vert _{\sigma _{8\varepsilon }}}\int_{Q}\mathcal{%
M}\left( \mathbf{1}_{Q}\sigma _{8\varepsilon }\right) ^{2}d\omega
_{\varepsilon }\leq C^{2} D^{\prime }\left( \mathfrak{T}_{\mathcal{M}%
}^{D}\left( \Gamma \right) \left( \sigma ,\omega \right) \right) ^{2},\ \ \
\ \ \text{for }0<\varepsilon \leq \alpha \ell \left( Q\right) .
\end{equation*}

If on the other hand, we have $\ell \left( Q\right) <\frac{1}{\alpha }%
\varepsilon $, then%
\begin{eqnarray*}
&&\left\Vert \mathbf{1}_{Q}\omega _{\varepsilon }\right\Vert _{\infty
}=\sup_{x\in Q}\int \varphi _{\varepsilon }\left( x-z\right) d\omega \left(
z\right) \le\sup_{x\in Q}\frac{2^n}{\left\vert B_{\varepsilon }\right\vert }%
\int_{B_{\varepsilon }+x}d\omega \left( z\right) \leq 2^n \frac{\left\vert
B_Q(\varepsilon ) \right\vert _{\omega }}{\left\vert B_{\varepsilon
}\right\vert }, \\
&&\text{and similarly }\left\Vert \mathbf{1}_{Q}\sigma _{8\varepsilon
}\right\Vert _{\infty }\leq 2^n\frac{\left\vert B_Q(8\varepsilon
)\right\vert _{\sigma }}{\left\vert B_{8\varepsilon }\right\vert },
\end{eqnarray*}%
and so%
\begin{eqnarray*}
\int_{Q}\mathcal{M}\left( \mathbf{1}_{Q}\sigma _{8\varepsilon }\right)
^{2}d\omega _{\varepsilon } &\leq &\left\Vert \mathbf{1}_{Q}\omega
_{\varepsilon }\right\Vert _{\infty }\int_{Q}\mathcal{M}\left( \mathbf{1}%
_{Q}\sigma _{8\varepsilon }\right) ^{2}dx \\
&\leq &\left\Vert \mathbf{1}_{Q}\omega _{\varepsilon }\right\Vert _{\infty
}C_{\mathrm{class}}^{2}\int_{Q}\left( \sigma _{8\varepsilon }\right)
^{2}\leq C_{\mathrm{class}}^{2}\left\Vert \mathbf{1}_{Q}\omega _{\varepsilon
}\right\Vert _{\infty }\left\Vert \mathbf{1}_{Q}\sigma _{8\varepsilon
}\right\Vert _{\infty }\int_{Q}\sigma _{8\varepsilon } \\
&\leq &C_{\mathrm{class}}^{2}4^n\frac{\left\vert B_Q(\varepsilon )
\right\vert _{\omega }\left\vert B_Q(8\varepsilon )\right\vert _{\sigma }}{%
\left\vert B_{\varepsilon }\right\vert \left\vert B_{8\varepsilon
}\right\vert }\left\vert Q\right\vert _{\sigma _{8\varepsilon }} \\
&\lesssim& C_{\mathrm{class}}^{2}A_{2}\left( \sigma ,\omega \right)
\left\vert Q\right\vert _{\sigma _{8\varepsilon }}.
\end{eqnarray*}

Thus altogether we have%
\begin{eqnarray*}
\frac{1}{\left\vert Q\right\vert _{\sigma _{8\varepsilon }}}\int_{Q}\mathcal{%
M}\left( \mathbf{1}_{Q}\sigma _{8\varepsilon }\right) ^{2}d\omega
_{\varepsilon } &\leq &C^{2}D^{\prime }\left( \mathfrak{T}_{\mathcal{M}%
}^{D}\left( \Gamma \right) \left( \sigma ,\omega \right) \right)
^{2}+CA_{2}\left( \sigma ,\omega \right) , \\
\text{provided }Q &\in &\mathcal{P}\text{ and }\left\vert \Gamma ^{\prime
}Q\right\vert _{\sigma _{8\varepsilon }}\leq D^{\prime }\left\vert
Q\right\vert _{\sigma _{8\varepsilon }}.
\end{eqnarray*}%
Upon taking the supremum over all $Q\in \mathcal{P}$ satisfying $\left\vert
\Gamma ^{\prime }Q\right\vert _{\sigma _{8\varepsilon }}\leq D^{\prime
}\left\vert Q\right\vert _{\sigma _{8\varepsilon }}$ it follows that%
\begin{equation*}
\mathfrak{T}_{\mathcal{M}}^{D^{\prime }}\left( \Gamma ^{\prime }\right)
\left( \sigma _{8\varepsilon },\omega _{\varepsilon }\right) \leq C\left[
\left( \mathfrak{T}_{\mathcal{M}}^{D}\left( \Gamma \right) \left( \sigma
,\omega \right) \right) +\sqrt{A_{2}\left( \sigma ,\omega \right) }\right] .
\end{equation*}%
This proves (\ref{prove both}), and so we obtain (\ref{claim D}) and hence (%
\ref{final'}), which completes the proof of Theorem \ref{weak}.

\end{document}